\documentclass[11pt]{article}
\usepackage{amsmath,amsthm,amssymb}
\usepackage{lineno,hyperref}
\usepackage{geometry}
\usepackage{color}
\usepackage{listings}

\newtheorem{lemma}{Lemma}[section]

\newtheorem{define}{Definition}[section]
\newtheorem{remark}{Remark}[section]
\newtheorem{proc}{Procedure}[section]

\def\R{{\mathbb{R}}}
\def\Q{{\mathbb{Q}}}

\def\N{{\mathbb{N}}}
\def\a{{\mathbf{a}}}
\def\ord{{\rm{ord}}}

\def\P{{\mathcal{P}}}
\def\C{{\mathcal{C}}}
\def\V{{\mathcal{V}}}

\def\M{{\mathcal{M}}}

\def\span{{\hbox{\rm{Span}}}}

\def\d{{\hbox{\rm{d}}}}

\lstdefinelanguage{Maple}{
   keywords={if, while, do, else, end, for, from, to,then},
   keywordstyle=\color{blue}\bfseries,
   ndkeywords={class, export, boolean, throw, implements, import, this},
   ndkeywordstyle=\color{darkgray}\bfseries,
   identifierstyle=\color{black},
   sensitive=false,
   comment=[l]{//},
   morecomment=[s]{/*}{*/},
   commentstyle=\color{purple}\ttfamily,
   stringstyle=\color{red}\ttfamily,
   morestring=[b]',
   morestring=[b]"
}

\lstdefinelanguage{SOStools}{
   keywords={syms,sosprogram,monomials,sosineq,sossetobj,sossolve,sosgetsol,sospolyvar},
   keywordstyle=\color{blue}\bfseries,
   ndkeywords={syms,sosprogram,monomials,sosineq,sossetobj,sossolve,sosgetsol},
   ndkeywordstyle=\color{blue}\bfseries,
   identifierstyle=\color{black},
   sensitive=false,
   comment=[l]{//},
   morecomment=[s]{/*}{*/},
   commentstyle=\color{purple}\ttfamily,
   stringstyle=\color{red}\ttfamily,
   morestring=[b]',
   morestring=[b]"
}

\modulolinenumbers[5]



\begin{document}


\title{\bf Prove Costa's Entropy Power Inequality and
High Order Inequality for Differential Entropy
\\ with Semidefinite Programming}
%
\author{Laigang Guo, Chun-Ming Yuan, Xiao-Shan Gao\\
KLMM, Academy of Mathematics and Systems Science\\
 Chinese Academy of Sciences, Beijing 100190, China\\
University of Chinese Academy of Sciences, Beijing 100049, China}
\date{\today}
\maketitle

\begin{abstract}
\noindent
Costa's entropy power inequality is an important generalization of
Shannon's  entropy power inequality.
Related with Costa's entropy power inequality and a conjecture proposed by McKean in 1966, Cheng-Geng recently conjectured that
$D(m,n): (-1)^{m+1}(\d^m/\d^m t)H(X_t)\ge0$,
where $X_t$ is the $n$-dimensional random variable in Costa's entropy power inequality and $H(X_t)$ the differential entropy of $X_t$.
$D(1,n)$ and $D(2,n)$ were proved by Costa as consequences of Costa's entropy power inequality.
Cheng-Geng proved $D(3,1)$ and $D(4,1)$.
In this paper, we propose a systematical procedure to prove $D(m,n)$
and Costa's entropy power inequality based on semidefinite programming.
Using software packages based on this procedure, we prove $D(3,n)$ for $n=2,3,4$
and give a new proof for Costa's entropy power inequality.
We also show that with the currently known constraints,
$D(5,1)$ and $D(4,2)$ cannot be proved with the procedure.

\end{abstract}



\section{Introduction}

Shannon's {\em entropy power inequality (EPI)}
is one of the most important information inequalities~\cite{Shannon1948}, which has many proofs, generalizations, and applications~\cite{Stam1959,Blachman1965,Lieb1978,VerduGuo2006,Rioul2011,Bergmans1974,Zamir1993,Liu2007,Wang2013}.
In particular,
Costa presented a stronger version of the EPI in his seminal paper~\cite{Costa1985}.

Let $X$ be an n-dimensional random variable with {\em probability density} $p(x)$.
For $t>0$, define $X_t\triangleq X+Z_t$, where $Z_t$ is an independent standard Gaussian random vector and is normally distributed with covariance matrix $t\times I$, that is, $Z_t\thicksim N_n(0,tI)$.
The {\em probability density} of $X_t$ is
\begin{equation}
\label{1.7}
p_t(x_t)=\dfrac{1}{(2\pi t)^{n/2}}\int_{\mathbf{R}^n}p(x)\exp\left(-\dfrac{\|x_t-x\|^2}{2t}\right) \d x_t,
\end{equation}
where  $\|\cdot\|$ denotes the Euclidean norm.
The {\em differential entropy} of $X_t$ is defined as
\begin{equation}
\label{1.1}
H(X_t)=-\int_{-\infty}^{+\infty} p_t(x_t)\log p_t(x_t)\d x_t.
\end{equation}
%
Costa~\cite{Costa1985} proved that
the {\em entropy power} of $X_t$, given by
\begin{equation}\label{1.5}
N(X_t)=\dfrac{1}{2\pi e}e^{(2/n)H(X_t)}
\end{equation}
is a concave function in $t$. More precisely,
Costa proved $\frac{\d}{\d t}N(X_t)\ge0$ and $\frac{\d^2}{\d^2 t}N(X_t)\le0$.
%
%

Due to its importance, several new proofs for Costa's EPI were given.
Dembo~\cite{Dembo1989} gave a simple proof for Costa's EPI via the Fisher information inequality.
Villani~\cite{Villani2000} proved Costa's EPI with advanced techniques as well as the heat equation.
Cheng and Geng  proved that the third order derivative of $H(X_t)$ is nonnegative and the fourth order derivative of $H(X_t)$ is nonpositive in the case of univariate random variables~\cite{Cheng2015}.

Related with a conjecture proposed by McKean in 1966~\cite{McKean1966}, Cheng-Geng~\cite{Cheng2015} conjectured
that $H(X_t)$ is {\em completely monotone} in $t$, that is,
\begin{equation}
\label{eq-D}
D(m,n): (-1)^{m+1}(\d^m/\d^m t)H(X_t)\ge0.
\end{equation}
Costa's EPI implies $D(1,n)$ and $D(2,n)$~\cite{Costa1985}
and Cheng-Geng proved $D(3,1)$ and $D(4,1)$~\cite{Cheng2015}.

In this paper, we propose a systematical and effective procedure to prove
Costa's EPI and $D(m,n)$.
The procedure consists of three main ingredients.
First, a systematic method is proposed to compute constraints
$R_i,i=1,\ldots,N$ satisfied by $p_t(x_t)$ and its derivatives.
Second, proof for $D(m,n)$ or Consta's EPI is reduced to the following problem
\begin{equation}
\label{eq-prob}
\exists (p_1,\ldots,p_N)\in\R^N \hbox{ s.t. } E -\sum_{i=1}^N p_i R_i = S
\end{equation}
where $E$ and $R_i$ are quadratic forms in certain newly introduced variables
and $S$ is a sum of squares (SOS) of linear forms in the new variables.
%
%
Third, problem \eqref{eq-prob} is solved  with the semidefinite programming (SDP)~\cite{Boyd1,Boyd2}.
There exists no  guarantee that the procedure will generate a proof,
but when succeeds, it gives an exact and strict proof for the problem under consideration.

Using the procedure proposed in this paper, we first give a new proof
for Costa's EPI which implies $D(2,n)$
and then prove $D(3,n)$ for $n=2,3,4$.
We also show that $D(5,1)$ and $D(4,2)$ cannot be proved under the current constraints
using this approach, detailed discussion about which can be found in Section \ref{sec-conc}.

\begin{table}[ht]
\centering
\begin{tabular}{lcccccccc}
\hline
 &$D(3,1)$ & $D(4,1)$ & $D(5,1)$ & $D(3,2)$ & $D(3,3)$ & $D(3,4)$& $D(4,2)$ & Costa's EPI\\
Vars & 3 & 5 & 7 & 14 & 38 & 80 & 80 & 6\\
Cons & 2 & 8 & 16& 63 & 512 & 1966 & 417 & 8\\
Time & 0.22 & 0.30 & 0.28 & 0.37 & 0.56 & 1.80 & 0.59 & 0.2\\
Proof & Yes & Yes & No & Yes & Yes & Yes & No & Yes\\
\hline
\end{tabular}
\caption{Data in computing the SOS  with SDP}
\label{tab1}
\end{table}

In Table \ref{tab1}, we give the  data for computing the SOS representation \eqref{eq-prob}
using the Matlab software package in Appendix B,
where
Vars is the number of variables,
Cons is the number of constraints,
Time is the running time in seconds collected on a desktop PC with a 3.40GHz CPU
and 16G memory,
and Proof means whether a proof is given.

The procedure is inspired by the work \cite{Costa1985,Villani2000,Cheng2015},
and uses basic ideas introduced therein.
The main difference is that heuristics are used to write derivatives of $N(X_t)$ or $H(X_t)$
as SOSs in \cite{Costa1985,Cheng2015}. We quote a remark from \cite{Cheng2015}
``One can apply the same technique to deal with the fifth derivative, or even
higher. However, the manipulation by hand is huge and hence
it is prohibitive in computational cost."
%
%
%
As shown in Table \ref{tab1}, our procedure basically solves the difficulty
of manipulating large expressions in solving problem \eqref{eq-prob} related with Costa's EPI.

The rest of this paper is organized as follows.
In Section 2, we give the method to generate constraints and the proof procedure.
In Section 3, a new proof for Costa's EPI is given based on the procedure.
In Section 4, $D(3,n)$ is proved for $n=2,3,4$ based on the procedure.
In Section 5, conclusion and discussion are presented.

\section{A procedure to prove entropy inequalities}
In this section, we give a general procedure to prove  Costa's EPI
and $D(m,n)$.

\subsection{Preliminaries}
Firstly, we give some notations. Let $x_t=[x_{1,t},x_{2,t},\ldots,x_{n,t}]$ and
$$
\begin{array}{ll}
\d^{(i)} x_t=\d x_{1,t}\d x_{2,t}\ldots \d x_{i-1,t}\d x_{i+1,t}\ldots \d x_{n,t},\ i=1,2\ldots,n.
\end{array}
$$
Let $[n]_0 = \{0,1,\ldots,n\}$ and $[n] = \{1,\ldots,n\}$.
To simplify the notations, we use $p_t$ to denote $p_t(x_t)$ in the rest of the paper.
Denote
$$\mathcal{P}_n =\{\frac{\partial^h p_t}{\partial^{h_1} x_{a_1,t}\cdots \partial^{h_n} x_{a_n,t}}:
h = \sum_{i=1}^n h_i, h_i\in \N\}$$
to be the set of all derivatives of $p_t$ with respect to the differential operators
$\frac{\partial}{\partial x_{i,t}},i=1,\ldots,n$
and $\R[\mathcal{P}_n]$ to be the set of polynomials in $\mathcal{P}_n$.
For $v\in\mathcal{P}_n$, let $\ord(v)$   be the order of $v$
and we use $p_t^{(h)}$ to denote an $h$th-order derivative of $p_t$  if no confusing is caused.
For a monomial $\prod_{i=1}^r v_i^{d_i}$ with $v_i\in {\mathcal{P}}$,
its {\em degree}, {\em order}, and {\em total order}
are defined to be $\sum_{i=1}^r d_i$, $\max_{i=1}^r \ord(v_i)$,
 and  $\sum_{i=1}^r d_i\cdot \ord(v_i)$, respectively.
%

A polynomial in $\R[\mathcal{P}_n]$ is called a $k$th-order
{\em differentially homogenous polynomial} or simply
a  $k$th-order {\em differential form},
if all its monomials have degree $k$ and total order $k$.
Let $\M_{k,n}$ be the set of all monomials which have degree $k$ and total order $k$.
Then all the $k$th-order differential  forms generate
an $\R$-linear vector space, which is denoted as $\span_\R(\M_{k,n})$.

We will use $\R$-Gaussian elimination in $\span_\R(\M_{k,n})$ by treating the monomials as variables.
We always use the {\em lexicographic order for the monomials} to be defined below if not mentioned otherwise.
Consider two distinct derivatives
$v_1=\frac{\partial^k p_t}{\partial^{h_1} x_{1,t}\cdots \partial^{h_n} x_{n,t}}$
and
$v_2=\frac{\partial^k p_t}{\partial^{s_1} x_{1,t}\cdots \partial^{s_n} x_{n,t}}$.
We say $v_1>v_2$ if
$h_l>s_l$ and $h_j=s_j$ for $j=l+1,\ldots,n$.
Consider two distinct monomials $m_1=\prod_{i=1}^{r} v_i^{d_i}$
and $m_2=\prod_{i=1}^{r} v_i^{e_i}$,
where $v_i\in {\mathcal{P}}_n$ and   $v_i< v_j$ for $i < j$.
We define $m_1 > m_2$ if  $d_l > e_l$, and $d_i = e_i$ for $i=l+1,\ldots,r$.
%

Costa~\cite{Costa1985} proved the following basic properties for $p_t$ and $H(X_t)$:
\begin{eqnarray}
&\frac{\d p_t}{\d t} =\frac{1}{2}\nabla^2p_t,\label{H1}\\
&\frac{\d H(X_t)}{\d t}=\frac{1}{2}\int_{\mathbf{R}^n}\frac{\|\nabla p_t\|^2}{p_t}\d x_t,\label{H2}\\
&\mathbb{E}[\nabla^2\log p_t(x_t)]=-\int_{\mathbf{R}^n}\frac{\|\nabla p_t(x_t)\|^2}{p_t(x_t)}\d x_t.
\label{H3}
\end{eqnarray}
where
$\nabla p_t =(\frac{\partial p_t}{\partial x_{1,t}},\ldots,
\frac{\partial p_t}{\partial x_{n,t}})$, $\nabla^2p_t=\sum\limits_{i=1}^{n}\frac{\partial^2p_t}{\partial^2 x_{i,t}}$,
and
$\mathbb{E}[\nabla^2\log p_t(x_t)]  :=
\int_{\mathbf{R}^n}p_t(x_t)\nabla^2\log p_t(x_t)\d x_t$
is the {\em expecttation} with respect to $x_t$.

Equation \eqref{H1} shows that $p_t$ satisfies the {\em heat equation}
and \eqref{H2} implies $D(1,n)$: $\frac{\d }{\d t} H(X_t)\ge0$.
The following lemma gives the general descrption for
$\frac{\d^m }{\d^m t} H(X_t)$.

\begin{lemma}
\label{lm-pr2}
For $m\in\N_{m>1}$, we have
\begin{equation}
\label{eq-tt1}
\begin{array}{ll}
\frac{\d^mH(X_t)}{\d^m t}
=\displaystyle{\int_{\mathbf{R}^n}\frac{F_{m,n}}{p_t^{2m-1}(x_t)}\d x_t},\\
F_{m,n} = \frac{p_t^{2m-1}}{2} \frac{\d^{m-1}}{\d^{m-1} t}\left(\frac{\|\nabla p_t\|^2}{p_t}\right)
=\sum_{a_1=1}^n\cdots\sum_{a_m=1}^n F_{m,n,\a_m}
\end{array}
\end{equation}
where $\a_m=(a_1,\ldots,a_m)$ and $F_{m,n,\a_m}$ is a $2m$th-order differentially form in $\R[\P_{m,n}]$ for
\begin{equation}
\label{eq-pm}
{\P}_{m,n} =\{
\frac{\partial^h p_t}{\partial^{h_1} x_{a_1,t}\cdots \partial^{h_{m}} x_{a_m,t}}:
h=\sum_{i=1}^m h_i\in[2m-1]_0; a_i\in[n],i=1,\ldots,m\}.
\end{equation}
\end{lemma}
\begin{proof}
We prove the lemma by induction on $m$.
Equation \eqref{H2} implies that the lemma is valid for $m=1$.
Assume that the lemma is valid for $m-1$, that is,
$F_{m-1,n}=\sum_{a_1=1}^n\cdots\sum_{a_{m-1}=1}^n F_{m-1,n,\a_{m-1}}
$ is a $2(m-1)$th-order differential form in $\R[\P_{m-1,n}]$. So we have
\begin{equation*}
\frac{\d^{m}H(X_t)}{\d^{m} t}
={\int_{\mathbf{R}^n}
 \frac{\d}{\d t}\left(\frac{F_{m-1,n}}{p_t^{2m-3}}\right)\d x_t}
={\int_{\mathbf{R}^n}\frac{F_{m,n}}{p_t^{2m-1}}\d x_t},
\end{equation*}
where $F_{m,n}=p_t^2\frac{\d F_{m-1,n}}{\d t}-(2m-3)p_t\frac{\d p_t}{\d t}F_{m-1,n}$.
For any $p_t^{(h)}\in\P_{m-1,n}$, we have
\begin{equation}
\label{eq-pm3}
\frac{\d p_t^{(h)}}{\d t}
= (\frac{\d p_t}{\d t}) ^{(h)}
\overset{\eqref{H1}}{=} 1/2\sum\limits_{i=1}^{n}\frac{\partial^2 p_t^{(h)}}{\partial^2 x_{i,t}}\in\R[\P_{m,n}]
\end{equation}
is of total order $h+2$.
From \eqref{eq-pm3} and the induction hypothesis, we can easily very
that $F_{m,n}$ is a $2m$th-order differential form in $\R[\P_{m,n}]$ and has the form in \eqref{eq-tt1}.
\end{proof}

\subsection{Compute the constraints}
In this section, we show how to compute the constraints systematically.

\begin{define}\label{def-41}
An $m$th-order  {\em  constraint} is a $2m$th-order differential form
$R$ in $\R[\P_{n}]$
such that $\int_{\mathbf{R}^n}\ \frac{R}{p_t^{2m-1}}\d x_t=0$.
%
%
%
\end{define}

There exist two types of constraints and the first type  comes from the following lemma,
whose proof can be found in Appendix A.
\begin{lemma}\label{lemma4}
Let $a,r,m_i,k_i \in \N_{>0}$ and $p_t^{(m_i)}$ an $m_i$th-order derivative of $p_t$. Then
\begin{equation}\label{lem4}
\begin{array}{ll}
\displaystyle{\int_{-\infty}^{\infty}\ldots\int_{-\infty}^{\infty}p_t\left[\prod\limits_{i=1}^{r}
\frac{[p_t^{(m_i)}]^{k_i}}{p_t^{k_i}}\right]\Bigg|_{x_{a,t}=-\infty}^{\infty}\d x_t^{(a)}}=0.\\[0.5cm]
\end{array}
\end{equation}
\end{lemma}

The following lemma shows how to compute an $m$th-order constraint.
\begin{lemma}
\label{lm-cons1}
Let $M$ be a monomial in $\R[\P_{m,n}]$ with degree $2m$ and total order $2m$,
where $\P_{m,n}$ is defined in \eqref{eq-pm}.
Then, we can obtain an $m$th-order constraint from $M$.
\end{lemma}
\begin{proof}
Let $v\in\P_{m,n}$ be a factor of $M$ such that $\ord(v)>0$.
%
Assume $v =\frac{\partial^o p_t}{\partial^ox_{a,t}}\ (o\geq1)$ without loss of generality.
By the integration by parts, we have
{\small
\begin{equation*}\label{3.I}
\begin{array}{ll}
&\displaystyle{\int_{\mathbf{R}^n} \frac{M}{p_t^{2m-1}}\d x_t}\\[0.4cm]
&=\displaystyle{\int_{-\infty}^{\infty}\ldots\int_{-\infty}^{\infty}\int_{-\infty}^{\infty}
\left(\frac{M_1}{p_t^{2m-1}}\frac{\partial^{o}p_t}{\partial^{o}x_{a,t}}\right)
  dx_{a,t}\d x_t^{(a)}}\\[0.4cm]
&=\displaystyle{\int_{-\infty}^{\infty}\ldots\int_{-\infty}^{\infty}
 \left(\frac{M_1}{p_t^{2m-1}}\frac{\partial^{o-1}p_t}{\partial^{o-1}x_{a,t}}\right)
 \Bigg|_{x_{a,t}=-\infty}^{\infty}\d x_t^{(a)}}
-\displaystyle{\int_{\mathbf{R}^n}
\left[\frac{\partial^{o-1}p_t}{\partial^{o-1}x_{a,t}}
\frac{\partial \left(\frac{M_1}{p_t^{2m-1}}\right)}{\partial x_{a,t}}\right]\d x_t}\\
&
\overset{\eqref{lem4}}{=}
-\displaystyle{\int_{\mathbf{R}^n}
\left[\frac{\partial^{o-1}p_t}{\partial^{o-1}x_{a,t}}
\frac{\partial \left(\frac{M_1}{p_t^{2m-1}}\right)}{\partial x_{a,t}}\right]\d x_t},\\
\end{array}
\end{equation*}}
which leads to
{\small \begin{equation*}
\label{3.I1}
\displaystyle{\int_{\mathbf{R}^n}
\frac{1}{p_t^{2m-1}}
\left[M +
p_t^{2m-1}\frac{\partial^{o-1}p_t}{\partial^{o-1}x_{a,t}}
\frac{\partial \left(\frac{M_1}{p_t^{2m-1}}\right)}{\partial x_{a,t}}\right]\d x_t=
\int_{\mathbf{R}^n}\frac{R}{p_t^{2m-1}}\d x_t=0}
\end{equation*}}
where
$R=M +
\frac{1}{p_t}\frac{\partial^{o-1}p_t}{\partial^{o-1}x_{a,t}}
\left(\frac{\partial M_1}{\partial x_{a,t}}p_t - (2m-1)M_1\frac{\partial p_t}{\partial x_{a,t}}\right)$. 
Then, it suffices to show that $R$ is a $2m$th-order differential form.

Since $M=vM_1$ has degree $2m$ and total order $2m$, $M_1$ has degree $2m-1$ and total order $2m-o$. If $o=1$, then
$R=M +\frac{\partial M_1}{\partial x_{a,t}}p_t - (2m-1)M_1\frac{\partial p_t}{\partial x_{a,t}}$ is
a $2m$th-order differential form.
If $o>1$, then $2m-o<2m-1$ and thus $M_1=p_tM_2$, where $M_2$ has degree $2m-2$ and total order $2m-o$.
Then we have $R=M +
\frac{\partial^{o-1}p_t}{\partial^{o-1}x_{a,t}}
\left[(\frac{\partial p_t}{\partial x_{a,t}}M_2+p_t\frac{\partial M_2}{\partial x_{a,t}}) - (2m-1)M_2\frac{\partial p_t}{\partial x_{a,t}}\right]$ is
a $2m$th-order differential form. The lemma is proved.
\end{proof}

For the second type constraints, we need to define the following operators.
For any differentiable function $\psi$ in $x_{i,t}$, let $\nabla^{(0)}\psi=\psi$ and
\begin{equation}
\label{eq-Del}
\begin{array}{ll}
\nabla^{(k)}\psi=\underset{total\ i}{\underbrace{\nabla^2(\nabla^2(\cdots(\nabla^2}}\psi)\cdots))
\hbox{ if }\ k=2i,\\[0.4cm]
\nabla^{(k)}\psi=\nabla(\nabla^{(2i)}\psi)
\hbox{ if}\ k=2i+1,\\[0.4cm]
%
\prod_{i=1}^{j}\nabla^{(k_i)}\psi=\nabla^{(k_1)}\psi*\nabla^{(k_2)}\psi*\cdots*
\nabla^{(k_j)}\psi.
\end{array}
\end{equation}
In \eqref{eq-Del},  $A*B$ represents the inner product  if $A$ and $B$ are vectors,
and the multiplication if $A$ and/or $B$ are scalars.
Also, $\nabla^{(k)}\psi$ is a scalar if $k$ is even and a vector if $k$ is odd.

\begin{lemma}\cite{Amazigo1980}\label{lemma1}
If $\phi(x)$ and $\psi(x)$ are twice continuously differentiable functions in $\R^n$ and $V$ is any set bounded by a piecewise smooth, closed, oriented surface $S$ in $\R^n$, then
\begin{equation}
\int_{V}\nabla\phi\cdot\nabla\psi \d V =\int_{S}\phi\nabla\psi\cdot \d\mathbf{s}-\int_{V}\phi\nabla^2\psi \d V,
\end{equation}
where $\d\mathbf{s}$ denotes the elementary area vector and $\nabla\psi\cdot \d\mathbf{s}$ is the inner product of these two vectors.
\end{lemma}

The second type constraints come from the following lemma,
whose proof is in Appendix A.
\begin{lemma}\label{lemma2}
Let $V_r$ be the $n$-dimensional sphere of radius $r$ centered at the origin and having
boundary surface denoted by $S_r$.
If $s,\ m_i,\ k_i\in \N_{>0}$, such that $\sum_{i=1}^{s}k_im_i$ is even and $m_{s+1}$ is odd (i.e., $\prod_{i=1}^s[\nabla^{(m_i)}p_t]^{k_i}$ is a scalar and $\nabla^{(m_{s+1})}p_t$ is a vector). Then
\begin{equation}
\label{eq-lem2}
\begin{array}{ll}
\displaystyle{\lim\limits_{r\to\infty}\int_{S_r}\prod\limits_{i=1}^{s}\frac{[\nabla^{(m_i)}p_t]^{k_i}}{p_t^{k_i}}\nabla^{(m_{s+1})} p_t\cdot {\rm d}\mathbf{s_r}}=0.
\end{array}
\end{equation}
\end{lemma}
The following lemma shows how to compute an $m$th-order constraint.
\begin{lemma}
\label{lm-cons2}
If $Q=\prod_{i=1}^{2m}\nabla^{(k_i)}p_t$, with  $\sum_{i=1}^{2m}k_i=2m$ and $k_i\in [2m-1]_0$,
then we can obtain an $m$th-order   constraint from $Q$.
\end{lemma}
\begin{proof}
Since $\sum_{i=1}^{2m}k_i=2m$,  $Q$ is scalar.
Actually, $Q$ is a $2m$th-order differential form,
since $\nabla^{(k_i)}p_t$ or its components is of total order $k_i$ and degree 1 by  \eqref{eq-Del}.
Let $w$ be a factor of $Q$ with the largest order and $Q = Q_1*w$.
Without loss of generality, assume  $w =\nabla^{(o)}p_t$ and $o$ is odd. Then we know that $w$ and $Q_1$ are vectors. By Lemma \ref{lemma1}, we have
\begin{equation*}\label{3.II}
\begin{array}{ll}
\displaystyle{\int_{\mathbb{R}^n}\frac{Q}{p_t^{2m-1}}\d x_t}\\[0.3cm]
=\displaystyle{\lim\limits_{r\rightarrow\infty}\int_{S_r} \nabla^{(o-1)}p_t*\frac{Q_1}{p_t^{2m-1}}\cdot dS_r}
-\displaystyle{\int_{\mathbb{R}^n}\nabla^{(o-1)}p_t*\nabla\left(\frac{Q_1}{p_t^{2m-1}}\right)\d x_t}\\[0.3cm]
\overset{\eqref{eq-lem2}}{=}
-\displaystyle{\int_{\mathbb{R}^n}\nabla^{(o-1)}p_t*\nabla\left(\frac{Q_1}{p_t^{2m-1}}\right)\d x_t}.
\end{array}
\end{equation*}
We thus have
\begin{equation*}\label{3.II1}
\begin{array}{ll}
\displaystyle{\int_{\mathbb{R}^n}\frac{1}{p_t^{2m-1}}\left[Q
+p_t^{2m-1}\nabla^{(o-1)}p_t*\nabla\left(\frac{Q_1}{p_t^{2m-1}}\right)
\right]\d x_t}=0,
\end{array}
\end{equation*}
and obtain a constraint
$R=Q+p_t^{2m-1}\nabla^{(o-1)}p_t*\nabla\left(\frac{Q_1}{p_t^{2m-1}}\right)$.
Similar to the proof of Lemma \ref{lm-cons1}, we can verify that $R$ is a $2m$-th order differential form.
\end{proof}

From lower order constraints, we can obtain higher order constraints.
\begin{lemma}
\label{lm-cons3}
From an $m$th-order constraint $R$,
we can obtain an $(m+1)$th-order constraint.
\end{lemma}
\begin{proof}
Since $R$ is an $m$th-order constraint, we have
\begin{equation*}
\displaystyle{
\frac{\d}{\d t} \int_{\mathbb{R}^n} \frac{R}{p_t^{2m-1}}\d x_t
=\int_{\mathbb{R}^n} \frac{\d}{\d t} \frac{R}{p_t^{2m-1}}\d x_t
=\int_{\mathbb{R}^n} \frac{R_1}{p_t^{2m+1}}\d x_t=0,
}\end{equation*}
where $R_1=p_t^2\frac{\d R}{\d t}-(2m-1) Rp_t\frac{\d p_t}{\d t}$.
By \eqref{eq-pm3}, $R_1$ is a $2(m+2)$th-order differential form and the lemma is proved.
%
\end{proof}

\subsection{The proof procedure}
In this section, we give the procedure to prove $D(m,n)$ and Costa's EPI,
which is called a procedure instead of an algorithm because
there is no guarantee of success.

\begin{proc}
\label{proc-H}
{\rm Input}: $E_{m,n}\in\R[\P_{m,n}]$ is a $2m$th-order differential form.

{\rm Output}: A proof for $\int_{\mathbf{R}^n}\ \frac{E_{m,n}}{p_t^{2m-1}}\d x_t\ge0$
for specific values of $m$ and $n$, or fail.
\end{proc}
{\bf Step 1}.
Compute the $m$th-order constraints: $\C_{m,n}=\{R_{i}, i=1,\ldots,N_1\}$
by using Lemmas \ref{lm-cons1}, \ref{lm-cons2},
and by applying Lemma \ref{lm-cons3} to $\C_{m-1,n}$.
Note that any element in the vector space $\span_\R(\C_{m,n})$ is also an $m$th-order constraint.

{\bf Step 2}.
Treat the monomials in $\M_{m,n}$ as new variables $m_i,i=1,\ldots,N_{m,n}$,
which are all the monomials in $\R[\P_{m,n}]$ with degree $m$ and total order $m$.
We call $m_im_j$ a {\em quadratic monomial}.

{\bf Step 3}.
Write each monomial in $\C_{m,n}$ as  a quadratic monomial if  possible.
Doing $\R$-Gaussian elimination to $\C_{m,n}$ by treating the monomials
as variables and according to
a monomial order such that a quadratic monomial is less than a non-quadratic monomial, we obtain
$$\widetilde{\C}_{m,n}={ \C}_{m,n,1}\cup { \C}_{m,n,2},$$
where ${ {\C}}_{m,n,1}=\{\widehat{R}_{i}, i=1,\ldots,\widehat{N}_{1}\}$ is the set of quadratic forms in $m_i$,
${ {\C}}_{m,n,2}=\{\widetilde{R}_{i}, i=1,\ldots,\widetilde{N}_{1}\}$ is the set of   non-quadratic forms, and $\span_\R(\C_{m,n})=\span_\R(\widetilde{\C}_{m,n})$.

{\bf Step 4}.
There may exist relations among the variables $m_i$, which are called {\em intrinsic
constraints}.
For instance, for $m_1=p_t^2 (\frac{\partial^2 p_t}{\partial^2 x_{1,t}})^2$,
$m_2=p_t (\frac{\partial p_t}{\partial x_{1,t}})^2 \frac{\partial^2 p_t}{\partial^2 x_{1,t}}$,
and $m_3=(\frac{\partial p_t}{\partial x_{1,t}})^4$ in $\M_{4,n}$,
an intrinsic constraint is $m_1m_3-m_2^2=0$.
Add those intrinsic constraints which are quadratic forms in $m_i$ to
${\C}_{m,n,1}$ to obtain
$$\widehat{\C}_{m,n,1}=\{\widehat{R}_i,i=1,\ldots,N_2\}.$$
Note that each  $\widehat{R}_i$  is an $m$th-order constraint.

{\bf Step 5}.
Let  $\widehat{E}_{m,n}$ be obtained from $E_{m,n}$
by eliminating the non-quadratic monomials using ${\C}_{m,n,2}$ such that
$E_{m,n}-\widehat{E}_{m,n}\in\span_\R(\C_{m,n,2})\subset \span_\R(\C_{m,n})$.
If $\widehat{E}_{m,n}$ is not a quadratic form in $m_i$, return fail.

{\bf Step 6}. Since  $\widehat{E}_{m,n}$ and $\widehat{R}_i,i=1,\ldots,N_2$
are quadratic forms in $m_i$,
we can use the Matlab program given in Appendix B to compute $p_i\in\R$ such that
\begin{equation}
\label{eq-tt3}
\widehat{E}_{m,n}-\sum_{i=1}^{N_2} p_i \widehat{R}_i =S
\end{equation}
where $S=\sum_{i=1}^{N_{m,n}} c_i (\sum_{j=i}^{N_{m,n}} e_{ij} m_j)^2$ is an SOS,
$c_i,e_{ij}\in\R$ and $c_i\ge0$.
If \eqref{eq-tt3} cannot be found, return fail.

{\bf Step 7}.
If \eqref{eq-tt3} is found, we have a proof for $\int_{\mathbf{R}^n}\ \frac{E_{m,n}}{p_t^{2m-1}}\d x_t\ge0$.
%
Since $E_{m,n}-\widehat{E}_{m,n}\in\span_\R(\C_{m,n})$,
each $\widehat{R}_i$ is an $m$th-order constraint,
$p_t\ge0$, and $S\ge0$,   we have the following proof:
\begin{equation}
\label{eq-t1}
\int \frac{{E}_{m,n}}{p_t^{2m-1}}\d x_t
=\int \frac{\widehat{E}_{m,n}}{p_t^{2m-1}}\d x_t
\overset{\eqref{eq-tt3}}{=}
\int \frac{\sum_{i=1}^{N_2} q_i \widehat{R}_i + S}{p_t^{2m-1}}\d x_t
=\int \frac{S}{p_t^{2m-1}}\d x_t\ge0.
\end{equation}

%

\begin{remark}
After Step 1, we can use the Matlab program in Appendix B to find  $p_i\in\R$ such that
\begin{equation}
\label{eq-tt5}
{E}_{m,n}-\sum_{i=1}^{N_1} p_i {R}_i =S
\end{equation}
where $S$ is an SOS in polynomials of degree $m$.
But doing so, we need to solve a much larger problem, because we need to treat all monomials of degree $m$ as new variables,
and in Procedure \ref{proc-H}, we need only consider monomials with degree $m$ and total order $m$ as   variables.
On the other hand, if Steps 5 or 6 fails, we can try to compute \eqref{eq-tt5}.
Theoretically, \eqref{eq-tt5} is more general than \eqref{eq-tt3}.
Practically, we do not have one case yet such that \eqref{eq-tt3} fails
and  \eqref{eq-tt5} succeeds.
\end{remark}

\subsection{An illustrative example}
\label{sec-p4}
We use $D(3,1): \frac{\d^3}{\d^3 t}H(X_t)\ge0$ first proved in~\cite{Cheng2015} as an illustrative example
for Procedure \ref{proc-H}.
Since $n=1$, denote
$x_t = x_{1,t}, f:=f_0:=p_t,\ \ f_n:=\frac{\partial^n p_t}{\partial^n x_{1,t}},\, n\in\N_{>0}$ for simplicity.
We have
\begin{equation*}
\frac{\d^3H(X_t)}{\d^3 t}
 \overset{\eqref{H2}}{=}\frac{1}{2}
 \frac{\d^2}{\d^2 t}\displaystyle{\int\frac{f_1^2}{f}\d x_t}
=\frac{1}{2}\displaystyle{\int\frac{\d^2}{\d^2 t}\left(\frac{f_1^2}{f}\right)\d x_t}
\overset{\eqref{H1}}{=}\displaystyle{\int\frac{F_{3,1}}{f^5} \d x_t}
\end{equation*}
where
$F_{3,1}=\frac{1}{4}f^4f_{3}^2-\frac{1}{2}f^3f_{1}f_{3}f_{2}+\frac{1}{4}f^4f_{1}f_{5}+\frac{1}{4}f^2f_{1}^2f_{2}^2-\frac{1}{8}f^3f_{1}^2f_{4}$
is a $6$th-order differential form and we can use  Procedure \ref{proc-H} to prove $D(3,1)$ with $F_{3,1}$ as the input.

In {Step} 1, we find 6 third order constraints: $\C_{3,1}=\{R_{i}, i=1,\ldots,6\}$ using Lemma
\ref{lm-cons1} from 6 monomials in $f_0,\ldots,f_5$ with degree 6 and total order 6. For instance,
from  monomial $ff_2f_1^4$, by the integration by parts, we have
{\small
\begin{equation*}
\displaystyle{\int\frac{ff_2f_1^4}{f^5}\d x_t=\left.\frac{f_1^5}{f^4}\right|^{\infty}_{x_t=-\infty}-\int f_1\frac{\partial}{\partial x_t}\left(\frac{f_1^4}{f^4}\right)\d x_t
\overset{\eqref{lem4}}{=}-\int \left(\frac{4ff_1^4f_2-4f_1^6}{f^5}\right)\d x_t}
\end{equation*}}
and obtain a constraint $R=5ff_1^4f_2-4f_1^6$. Here are the 6 constraints:

{\small
\begin{equation*}
\begin{array}{ll}
{R}_{1} = 5ff_{1}^4f_{2}-4f_{1}^6,
&{R}_{2} = 2f^3f_{1}f_{2}f_{3}+f^3f_{2}^3-2f^2f_{1}^2f_{2}^2,\\
{R}_{3} = f^4f_{1}f_{5}+f^4f_{2}f_{4}-f^3f_{1}^2f_{4},
&{R}_{4} = f^3f_{1}^2f_{4}+2f^3f_{1}f_{2}f_{3}-2f^2f_{1}^3f_{3},\\
{R}_{5} = f^2f_{1}^3f_{3}+3f^2f_{1}^2f_{2}^2-3ff_{1}^4f_{2},
&{R}_{6} = f^4f_{2}f_{4}+f^4f_{3}^2-f^3f_{1}f_{2}f_{3}.
\end{array}\end{equation*}
Since $m=3$ and $n=1$, we do not need Lemmas \ref{lm-cons2} and \ref{lm-cons3}.

In {Step} 2,  introduce new variables $\M_{3,1}$, which are all the monomials in $f_i$ with degree 3 and total order 3 and are listed from high to low in the lexicographical monomial order:
$
m_{1} = f^2f_{3}, m_{2} = ff_{1}f_{2}, m_{3} = f_{1}^3.
$

In {Step} 3,
writing each monomial in $\C_{3,1}$ as   a quadratic monomial in $m_i$ if  possible and
doing $\R$-Gaussian elimination to $\C_{3,1}$, we have
\begin{equation*}
\begin{array}{lll}
{{\C}}_{3,1,1}=&\{
 \widehat{R}_1=5m_{2}m_{3}-4m_{3}^2,\quad
 &\widehat{R}_2=m_{1}m_{3}+3m_{2}^2-\frac{12}{5}m_{3}^2\},\\
{{\C}}_{3,1,2}=&\{
\widetilde{R}_1=f^3f_2^3+2m_1m_2-2m_2^2,
&\widetilde{R}_2=f^4f_1f_5-m_1^2+3m_1m_2+6m_2^2-\frac{24}{5}m_3^2,\\
&\hskip5pt \widetilde{R}_3=f^4f_2f_4+m_1^2-m_1m_2,
\quad
&\widetilde{R}_4=f^3f_1^2f_4+2m_1m_2+6m_2^2-\frac{24}{5}m_3^2\}.
\end{array}\end{equation*}

In {Step} 4, there exist no intrinsic constraints
and thus ${\widehat{\C}}_{3,1,1}=\{\widehat{R}_1,\widehat{R}_2\}$ and $N_2=2$.

In {Step} 5, trying to write monomials of $F_{3,1}$ as quadratic forms in $m_i$, we have
${F}_{3,1}=\frac{1}{4}m_{1}^2-\frac{1}{2}m_{1}m_{2}+\frac{1}{4}f^4f_{1}f_{5}+\frac{1}{4}m_{2}^2-\frac{1}{8}f^3f_{1}^2f_{4}$.
Eliminating the non-quadratic monomials from $\widetilde{F}_{3,1}$ using $\C_{3,1,2}$, we have
\small $$\begin{array}{ll}
\widehat{F}_{3,1}=
{F}_{3,1} - (\frac{1}{4}\widetilde{R}_2-\frac{1}{8}\widetilde{R}_4) = \frac{1}{2}m_{1}^2-m_{1}m_{2}-\frac{1}{2}m_{2}^2+\frac{3}{5}m_{3}^2.
\end{array}$$

In {Step 6}, since  $\widehat{F}_{3,1}$, $\widehat{R}_1$, and $\widehat{R}_2$
are quadratic forms in $m_i$,
we can use the Matlab program in Appendix B to find the following SOS representation
\begin{equation}
\label{eq-cc3}
\begin{array}{ll}
\widehat{F}_{3,1}=\sum_{i=1}^{2}p_i\widehat{R}_i+ \sum_{i=1}^3 c_i (\sum_{j=i}^3 e_{i,j} m_j)^2
\end{array}
\end{equation}
where $p_1=\frac{213}{1444},\ p_2=-\frac{407}{1009}$,
$c_1=\frac{1}{2036162}, c_2 =\frac{1}{1784109085952}, c_3=\frac{69485907702371}{9000830338627840}$,
$e_{1,1}=1009, e_{1,2}=-1009, e_{1,3}=407$,
$e_{2,2}=612256,e_{2,3}=-486877$,
$e_{3,3}=1$.

In {Step} 7, from the above SOS representation \eqref{eq-cc3}, we give a proof \eqref{eq-t1} for $D(3,1)$.
%
%
Note that equation (\ref{eq-cc3}) is different from that given in~\cite{Cheng2015}.

\section{Costa's EPI}
\label{sec-epi}
In this section, we give a new proof for Costa's EPI using Procedure \ref{proc-H}.

\subsection{Compute the second order constraints}
\label{sec-epi1}
To prove Costa's EPI, we need the following second order constraints:
\begin{equation}
\label{eq-2cons}
\C_{2,n}=\{R_{i,a,b}^{(2)},R_{j}^{(0)}  \,:\,i=1,\ldots,17; j=1,2; a,b\in[n]\},
\end{equation}
where $R_{j,a,b}^{(2)}$ and $R_{i}^{(0)}$  will be given in Lemmas \ref{lem-42} and \ref{lem-41} below.

Due to the summation structure of $\frac{\d^2 }{\d^2} H(X_t)$ in \eqref{eq-tt1},
we introduce the following notations
\begin{equation}
\label{eq-v2}
{\mathcal V}_{a,b} =\{ \frac{\partial^h p_t}{\partial^{h_1} x_{a,t} \partial^{h_2} x_{b,t}}:
h=h_1+h_2\in[3]_0\}
\end{equation}
where $a,b$ are variables taking values in $[n]$.
Then $\P_{2,n}=\cup_{a=1}^n\cup_{b=1}^n \V_{a,b}$.
%
\begin{lemma}\label{lem-42}
Based on Lemma \ref{lm-cons1}, we obtain 17 second order constraints $R^{(2)}_{i,a,b}\in\R[\V_{a,b}],\ i=1,\ldots,17$, which can be found in Appendix C.
\end{lemma}
\begin{proof}
From Lemma \ref{lm-cons1}, the constraints are from
monomials in $\R[\V_{a,b}]$ with degree 4 and total order 4.
There exist 60 such monomials and the proof of Costa's EPI
need only 17 constraints which are from monomials
of the form $p_t^2p_t^{(1)}p_t^{(3)}$ and $p_t p_t^{(1)}p_t^{(1)}p_t^{(2)}$.
We give an illustrative example.
For the monomial
$p_t^2 \dfrac{\partial p_t}{\partial x_{a,t}}
 \dfrac{\partial^3p_t}{\partial^2 x_{a,t}\partial x_{b,t}}$,
using formula for integration by parts, we have
{\small \begin{equation*}
\begin{array}{ll}
&\displaystyle{\int_{\mathbf{R}^n}\left(\dfrac{1}{p_t}\dfrac{\partial^3p_t}{\partial^2 x_{a,t}\partial x_{b,t}}\dfrac{\partial p_t}{\partial x_{a,t}}\right)\d x_t}\\[0.4cm]
=&\displaystyle{\int_{-\infty}^{\infty}\cdots\int_{-\infty}^{\infty}\int_{-\infty}^{\infty}\left(\dfrac{1}{p_t}\dfrac{\partial^3p_t}{\partial^2 x_{a,t}\partial x_{b,t}}\dfrac{\partial p_t}{\partial x_{a,t}}\right)dx_{a,t}\d x_t^{(a)}}\\
=&\displaystyle{\int_{-\infty}^{\infty}\cdots\int_{-\infty}^{\infty}\left(\dfrac{1}{p_t}\dfrac{\partial^2p_t}{\partial x_{a,t}\partial x_{b,t}}\dfrac{\partial p_t}{\partial x_{a,t}}\right)\Bigg|_{x_{a,t}=-\infty}^{\infty}\d x_t^{(a)}}
-\displaystyle{\int_{\mathbf{R}^n}\left[\dfrac{\partial^2p_t}{\partial x_{a,t}\partial x_{b,t}}\dfrac{\partial}{\partial x_{a,t}}\left(\dfrac{\frac{\partial p_t}{\partial x_{a,t}}}{p_t}\right)\right]\d x_t}.
\end{array}\end{equation*}}
By Lemma \ref{lemma4},
{\small\begin{equation*}
\begin{array}{ll}
&\displaystyle{\int_{\mathbf{R}^n}\left(\dfrac{1}{p_t}\dfrac{\partial^3p_t}{\partial^2 x_{a,t}\partial x_{b,t}}\dfrac{\partial p_t}{\partial x_{a,t}}+\dfrac{\partial^2p_t}{\partial x_{a,t}\partial x_{b,t}}\dfrac{\partial}{\partial x_{a,t}}\left(\dfrac{\frac{\partial p_t}{\partial x_{a,t}}}{p_t}\right)\right)\d x_t}
=\displaystyle{\int_{\mathbf{R}^n}\dfrac{C_{a,b}}{p_t^3}\d x_t}=0,
\end{array}\end{equation*}}
where  $C_{a,b}=
p_t^2\frac{\partial^3p_t}{\partial^2 x_{a,t}\partial x_{b,t}}\frac{\partial p_t}{\partial x_{a,t}}
+\frac{\partial^2p_t}{\partial x_{a,t}\partial x_{b,t}}(p_t^2\frac{\partial^2 p_t}{\partial^2 x_{a,t}}-p_t(\frac{\partial p_t}{\partial x_{a,t}})^2)$
is a second order constraint.
\end{proof}

\begin{lemma}\label{lem-41}
From Lemma \ref{lm-cons2}, we obtain constraints
$R^{(0)}_i = \sum_{a=1}^n\sum_{b=1}^n R^{(0)}_{i,a,b}$ for $i=1,2$, where
{\small
\begin{equation}\label{2.6}
\begin{array}{ll}
%
%
%
R^{(0)}_{1,a,b}= p_t^2\dfrac{\partial^3p_t}{\partial x_{a,t}\partial^2x_{b,t}}\dfrac{\partial p_t}{\partial x_{a,t}}
+\dfrac{\partial^2p_t}{\partial^2 x_{a,t}}\left[p_t^2\dfrac{\partial^2p_t}{\partial^2x_{b,t}}
-p_t\left(\dfrac{\partial p_t}{\partial x_{b,t}}\right)^2\right],\\[0.3 cm]
R^{(0)}_{2,a,b}=p_t\dfrac{\partial^2p_t}{x_{a,t}^2}\left(\dfrac{\partial p_t}{\partial x_{b,t}}\right)^2
+2\dfrac{\partial p_t}{\partial x_{a,t}}\left[p_t\dfrac{\partial^2p_t}{\partial x_{a,t}\partial x_{b,t}}\dfrac{\partial p_t}{\partial x_{b,t}}
-\dfrac{\partial p_t}{\partial x_{a,t}}\left(\dfrac{\partial p_t}{\partial x_{b,t}}\right)^2\right].
\end{array}
\end{equation}}
\end{lemma}
\begin{proof}
For $m=2$, $Q$ in Lemma \ref{lm-cons2} has 4  forms:
$\nabla^{(3)}p_t\nabla p_t p_t^2$, $(\nabla^2 p_t)^2p_t^2$,
$\nabla^2p_t\nabla p_t \nabla p_t p_t$, $(\nabla p_t)^4$.
From these $Q$, we obtain 2 distinct constraints.
We give an illustrative example.
Based on Lemma \ref{lemma2}, we have
\begin{equation*}
\begin{array}{ll}
\displaystyle{\int_{\mathbf{R}^n}\dfrac{\nabla(\nabla^2p_t)\cdot\nabla p_t}{p_t}\d x_t}\\[0.3cm]
=\lim\limits_{r\rightarrow\infty}\left(\displaystyle{\int_{S_r}\dfrac{\nabla^2p_t}{p_t}\nabla p_t\cdot ds_r}-\displaystyle{\int_{V_r}\nabla^2p_t\nabla\left(\dfrac{\nabla p_t}{p_t}\right)\d V_r}\right)\\
\overset{{\eqref{eq-lem2}}}{=}-\displaystyle{\int_{\mathbf{R}^n}\nabla^2p_t\nabla\left(\dfrac{\nabla p_t}{p_t}\right)\d x_t},
\end{array}
\end{equation*}
which leads to a constraint:
{\small
\begin{equation*}
\begin{array}{ll}
\displaystyle{\int_{\mathbf{R}^n}\dfrac{\nabla(\nabla^2p_t)\cdot\nabla p_t}{p_t}+\nabla^2p_t\nabla\left(\dfrac{\nabla p_t}{p_t}\right)\d x_t}\\
=\displaystyle{\int_{\mathbf{R}^n}\sum\limits_{a=1}^{n}\sum\limits_{b=1}^{n}\left[\dfrac{\frac{\partial^3p_t}{\partial x_{a,t}\partial^2x_{b,t}}\frac{\partial p_t}{\partial x_{a,t}}}{p_t}+\dfrac{\partial^2 p_t}{\partial^2 x_{a,t}}\dfrac{\partial}{\partial x_{b,t}}\left(\dfrac{\frac{\partial p_t}{x_{b,t}}}{p_t}\right)\right]\d x_t} \\
=
\displaystyle{\int_{\mathbf{R}^n}\sum\limits_{a=1}^{n}\sum\limits_{b=1}^{n}
\dfrac{C_{a,b}}{p_t^3} \d x_t}=0
\end{array}
\end{equation*}}
where
$C_{a,b}= p_t^2\frac{\partial^3p_t}{\partial x_{a,t}\partial^2x_{b,t}}\frac{\partial p_t}{\partial x_{a,t}}
 +p_t^2\dfrac{\partial^2p_t}{\partial^2 x_{a,t}}\dfrac{\partial^2p_t}{\partial^2x_{b,t}}
 -p_t\dfrac{\partial^2p_t}{\partial^2 x_{a,t}}\left(\dfrac{\partial p_t}{\partial x_{b,t}}\right)^2$.
\end{proof}

For $m=2$, Lemma \ref{lm-cons3} is not needed, since there exist no first order constraints.

\subsection{Prove Costa's EPI}
%
It suffices to show  $\frac{\d^2 N(X_t)}{\d^2 t}\le0$, since
$\frac{\d N(X_t)}{\d t}\ge0$ is a direct consequence of \eqref{H2}.
We first compute $\frac{\d^2 N(X_t)}{\d^2 t}$:
{\small
\begin{equation}\label{2.7}
\begin{array}{ll}
&\dfrac{\d^2N(X_t)}{\d^2t}=\dfrac{\d^2}{\d^2 t}e^{(2/n)H(X_t)}\\[0.2cm]
&=e^{(2/n)H(X_t)}\left[\left(\dfrac{2}{n}\dfrac{\mathrm{d}}{\mathrm{d}t}H(X_t)\right)^2+\dfrac{2}{n}\dfrac{\mathrm{d}^2}{\mathrm{d}t^2}H(X_t)\right]\\[0.3cm]
&\overset{{\eqref{H2}}}{=}
\dfrac{1}{n}e^{(2/n)H(X_t)}\left[\left(\displaystyle{\int_{\mathbf{R}^n}\dfrac{\|\nabla p_t\|^2}{p_t}\mathrm{d}x_t}\right)^2+\dfrac{\mathrm{d}}{\mathrm{d}t}\displaystyle{\int_{\mathbf{R}^n}\dfrac{\|\nabla p_t\|^2}{p_t}\mathrm{d}x_t}\right]\\[0.3cm]
&\overset{{\eqref{H3}}}{=}
\dfrac{1}{n}e^{(2/n)H(X_t)}\left[\left(\mathbb{E}[\nabla^2\log p_t]\right)^2
 +\dfrac{\mathrm{d}}{\mathrm{d}t}\displaystyle{\int_{\mathbf{R}^n}\dfrac{\|\nabla p_t\|^2}{p_t}\mathrm{d}x_t}\right]\\[0.3cm]
&\leq \dfrac{1}{n}e^{(2/n)H(X_t)}\left\{\mathbb{E}\left[\sum\limits_{a=1}^n\left(\dfrac{\partial^2\log p_t}{\partial^2 x_{a,t}}\right)^2\right]
+\dfrac{\mathrm{d}}{\mathrm{d}t}\displaystyle{\int_{\mathbf{R}^n}\dfrac{\|\nabla p_t\|^2}{p_t}\mathrm{d}x_t}
\right\}\\[0.3cm]
&= \dfrac{1}{n}e^{(2/n)H(X_t)}\left\{\displaystyle{\int_{\mathbf{R}^n}\sum\limits_{a=1}^{n}\left[p_t\left(\dfrac{\partial^2\log p_t}{\partial^2 x_{a,t}}\right)^2\right]\mathrm{d}x_t
+\int_{\mathbf{R}^n}\sum\limits_{a=1}^{n}\dfrac{\mathrm{d}}{\mathrm{d}t}\left(\dfrac{\left(\frac{\partial p_t}{\partial x_{a,t}}\right)^2}{p_t}\right)\mathrm{d}x_t}\right\}\\[0.3cm]
%
%
&\overset{{\eqref{H1}}}{=}
\dfrac{1}{n}e^{(2/n)H(X_t)}\left[\displaystyle{\int_{\mathbf{R}^n}
\frac{J_{2,n}}{p_t^3}\mathrm{d}x_t}\right]
\end{array}
\end{equation}}
where $J_{2,n} = \sum\limits_{a=1}^{n}S_a+\sum\limits_{a=1}^{n}\sum\limits_{b=1}^{n}T_{a,b}$ and
\begin{equation*}
\begin{array}{l}
S_a=p_t^4\left(\frac{\partial^2\log p_t}{\partial^2 x_{a,t}}\right)^2
   = \left(p_t\frac{\partial^2 p_t}{\partial^2 x_{a,t}}-(\frac{\partial p_t}{\partial x_{a,t}})^2 \right)^2,\\
T_{a,b}=p_t^2\frac{\partial p_t}{x_{a,t}}\frac{\partial^3 p_t}{\partial x_{a,t}\partial^2x_{b,t}}-\frac{p_t}{2}\left(\frac{\partial p_t}{x_{a,t}}\right)^2\frac{\partial^2 p_t}{\partial^2x_{b,t}}.
\end{array}
\end{equation*}}

Using Procedure \ref{proc-H} to prove $\frac{\d^2 N(X_t)}{\d^2 t}\le0$, it suffices to write $J_{2,n}$ as
\begin{equation}
\label{eq-sosepi}
J_{2,n} = \sum_{R\in\C_{2,n}} c_R R - S
\end{equation}
where $c_R\in\R$ and $S$ is an SOS. Since $S\ge0, p_t\ge0$, we have a proof for Costa's EPI:
\begin{equation*}
\begin{array}{l}
\frac{\d^2 N(X_t)}{\d^2 t}
\overset{{\eqref{2.7}}}{\le} B\displaystyle{\int_{\mathbf{R}^n}
\frac{J_{2,n}}{p_t^3}\mathrm{d}x_t}
\overset{{\eqref{eq-sosepi}}}{=} B\displaystyle{\int_{\mathbf{R}^n}
\frac{\sum_{R\in\C_{2,n}} c_R R - S}{p_t^3}\mathrm{d}x_t}
= B\displaystyle{\int_{\mathbf{R}^n}
\frac{-S}{p_t^3}\mathrm{d}x_t}\le0
\end{array}
\end{equation*}
where $B=\dfrac{1}{n}e^{(2/n)H(X_t)}\ge0$.
We will find \eqref{eq-sosepi} in two cases, which will be given in the
next two subsections, and thus give a proof for Costa's EPI.

\subsubsection{The univariate case $(n=1)$}

The univariate case is trivial and can be proved using Procedure \ref{proc-H}.
Since $n=1$, the constraints in Lemma \ref{lem-41} are not needed. From \eqref{eq-2cons},
the constraints are
$$
\begin{array}{ll}
\C_{2,1}
&=\{R_{i}^{(0)}, R_{j,a,b}^{(2)} \,:\,i=1,2; j=1,\ldots,17; a,b\in[1]\},\\
&=\{
 R_1=p_t^2\frac{\partial p_t}{\partial x_{1,t}}\frac{\partial^3 p_t}{\partial^3 x_{1,t}}
 +p_t^2(\frac{\partial^2 p_t}{\partial^2 x_{1,t}})^2
 -p_t(\frac{\partial p_t}{\partial x_{1,t}})^2\frac{\partial^2 p_t}{\partial^2 x_{1,t}},\\
&\hskip15pt R_2=3p_t(\frac{\partial p_t}{\partial x_{1,t}})^2\frac{\partial^2 p_t}{\partial^2 x_{1,t}} -2(\frac{\partial p_t}{\partial x_{1,t}})^4\}.
\end{array}
$$
Since $n=1$, from \eqref{2.7},
$J_{2,1}=S_1+T_{1,1} =
 2p_t^2\frac{\partial p_t}{\partial x_{1,t}}\frac{\partial^3 p_t}{\partial^3 x_{1,t}}
 +2p_t^2(\frac{\partial^2 p_t}{\partial^2 x_{1,t}})^2
 -5p_t(\frac{\partial p_t}{\partial x_{1,t}})^2\frac{\partial^2 p_t}{\partial^2 x_{1,t}}
 +2(\frac{\partial p_t}{\partial x_{1,t}})^4
 =2R_1-R_2$, which gives equation \eqref{eq-sosepi} for $n=1$
 and thus proves Costa's EPI for $n=1$.
%
%

\subsubsection{The general case $(n>1)$}

The general case    cannot be proved directly with Procedure \ref{proc-H},
due to the existence of the parameter $n$.
We will reduce the general case to a ``finite" problem which can be solved
with Procedure \ref{proc-H}.

By Lemma \ref{lem-41}, $R^{(0)}_i = \sum_{a=1}^n\sum_{b=1}^n R^{(0)}_{i,a,b}, i=1,2$ are second order  constraints.
From \eqref{2.7}, to prove $\dfrac{\d^2}{\d^2 t}N(X_t)\le0$, it suffices to solve

\noindent{\bf Problem I}.
There exist $c_1,c_2\in\R$ such that
$$
\begin{array}{l}
L_{2,n}=\sum\limits_{a=1}^{n}S_a+\sum\limits_{a=1}^{n}\sum\limits_{b=1}^{n}(T_{a,b}
+ c_1R^{(0)}_{1,a,b} + c_2R^{(0)}_{2,a,b}) \leq0
\end{array}
$$
under the constraints $R^{(2)}_{i,a,b},i=1,\ldots,17$ given in Lemma \ref{lem-42}.

Motivated by   symmetric functions,  for any function $f(a,b)$, we have
\begin{equation}\begin{array}{ll}
\label{2.9a}
\sum\limits_{a,b=1}^{n}f(a,b)=\sum\limits_{1\leq a<b}^{n}\left\{\frac{1}{n-1}[f(a,a)+f(b,b)]+[f(a,b)+f(b,a)]\right\}.
\end{array}\end{equation}
By \eqref{2.9a}, we have
\begin{equation*}\label{2.12}
\begin{array}{ll}
L_{2,n}&=\sum\limits_{a=1}^{n}S_a+\sum\limits_{a=1}^{n}\sum\limits_{b=1}^{n}(T_{a,b}+
   c_1R^{(0)}_{1,a,b}+c_2R^{(0)}_{2,a,b})\\
&=\frac{1}{n-1}\sum\limits_{a<b}(S_{a}+S_{b})
  +\sum\limits_{a<b}\left[\frac{1}{n-1}(T_{a,a}+T_{b,b}
    +c_1(R^{(0)}_{1,a,a}+R^{(0)}_{1,b,b})
    +c_2(R^{(0)}_{2,a,a}+R^{(0)}_{2,b,b}))\right.\\
   &\left.
  + T_{a,b}+T_{b,a}+c_1(R^{(0)}_{1,a,b}+R^{(0)}_{1,b,a})
   +c_2(R^{(0)}_{2,a,b}+R^{(0)}_{2,b,a})\right]\\
&=\sum\limits_{a<b}\left(\frac{1}{n-1}L_{1,a,b}+L_{2,a,b}\right),
\end{array}
\end{equation*}
where
$$\begin{array}{ll}
L_{1,a,b}=    S_{a}+S_{b}+ T_{a,a}+T_{b,b}
    +c_1({R}^{(0)}_{1,a,a}+{R}^{(0)}_{1,b,b})
    +c_2({R}^{(0)}_{2,a,a}+{R}^{(0)}_{2,b,b}),\\[0.15cm]
L_{2,a,b}=T_{a,b}+T_{b,a}+c_1({R}^{(0)}_{1,a,b}+{R}^{(0)}_{1,b,a})
   +c_2({R}^{(0)}_{2,a,b}+{R}^{(0)}_{2,b,a}).
\end{array}
$$

To prove {\bf Problem I}, it suffices to prove

\noindent{\bf Problem II}. There exist $c_1, c_2\in\R$ such that $L_{1,a,b},L_{2,a,b}\leq0$ under the  constraints $R^{(2)}_{i,a,b},i=1,\ldots,{17}$.

In {\bf Problem II}, the subscripts $a$ and $b$ are fixed and we will not
explicitly give them anymore. Thus, we denote $L_1=L_{1,a,b},L_2=L_{2,a,b}$ in the rest of this section.
We now prove {\bf Problem II} with  Procedure \ref{proc-H}

In Step 1, for {\bf Problem II}, the constraints are
$\overline{\C}_{2,n}=\{R^{(2)}_{j,a,b}\,:\,  j=1,\ldots,17\}$.

In Step 2, the new variables are all the monomials in $\R[{\mathcal V}_{a,b}]$
with degree 2 and total order 2 (${\mathcal V}_{a,b}$ is defined in \eqref{eq-v2}):
$$
\begin{array}{ll}
&
m_1=\left(\frac{\partial p_t(x_t)}{x_{a,t}}\right)^2,\
m_2=\left(\frac{\partial p_t(x_t)}{x_{b,t}}\right)^2,\
m_3=\frac{\partial p_t(x_t)}{\partial x_{a,t}}\frac{\partial p_t(x_t)}{x_{b,t}},\\
&
 m_4=p_t(x_t)\frac{\partial^2p_t(x_t)}{\partial x_{a,t}\partial x_{b,t}},\
m_5=p_t(x_t)\frac{\partial^2p_t(x_t)}{\partial^2 x_{a,t}},\ m_6=p_t(x_t)\frac{\partial^2p_t(x_t)}{\partial^2x_{b,t}}.
\end{array}
$$

In Step 3, we  obtain
${\C}_{2,n,1}=\{\widehat{R}_i,i=1,\ldots,7\}$
and ${\C}_{2,n,2}=\{\widetilde{R}_i,i=1,\ldots,10\}$, where
$$\begin{array}{ll}
\widehat{R}_1=m_1m_6-2m_3^2+2m_3m_4,\ &\widehat{R}_2=-2m_2m_3+m_2m_4+2m_3m_6,\\
\widehat{R}_3=-2m_2^2+3m_2m_6,\ &\widehat{R}_4=-2m_1m_3+m_1m_4+2m_3m_5,\\
\widehat{R}_5=m_2m_5-2m_3^2+2m_3m_4,\ &\widehat{R}_6=-2m_2m_3+3m_2m_4,\\
\widehat{R}_7=-2m_1^2+3m_1m_5. &\\
\widetilde{R}_1=p_t^2\frac{\partial p_t}{\partial x_{b,t}}\frac{\partial^3p_t}{\partial^3 x_{b,t}}-m_2m_6+m_6^2
& \widetilde{R}_2=p_t^2\frac{\partial p_t}{\partial x_{a,t}}\frac{\partial^3p_t}
{\partial^3 x_{a,t}}-m_1m_5+m_5^2\\
\widetilde{R}_3=p_t^2\frac{\partial p_t}{\partial x_{a,t}}\frac{\partial^3p_t}{\partial x_{a,t}\partial^2x_{b,t}} -m_3m_4+m_4^2
& \widetilde{R}_4=p_t^2\frac{\partial p_t}{\partial x_{b,t}}\frac{\partial^3p_t}{\partial^2 x_{a,t}x_{b,t}}-m_3m_4+m_4^2.
\end{array}$$
$\widetilde{R}_k,k=5,\ldots,10$ are not given, because they are not used in the proof.

In Step 4, there exists one intrinsic constraint: $\widehat{R}_{8}=m_1m_2-m_3^2$ and $N_2=8$.

In Step 5, eliminating the non-quadratic monomials in $L_{1}$ and $L_{2}$
using $\C_{2,n,2}$, we have
\begin{equation*}
\begin{array}{ll}
\widehat{L}_1
&={L}_1-(c_1+1)(\widetilde{R}_1+\widetilde{R}_2)\\
&=-\frac{1}{2}(2c_2-1)(2m_1^2-3m_1m_5+2m_2^2-3m_2m_6),\\
\widehat{L}_2
&={L}_2-(c_1+1)(\widetilde{R}_3+\widetilde{R}_4)\\
&=(c_2-c_1-\frac{1}{2})m_1m_6+(c_2-c_1-\frac{1}{2})m_2m_5-4c_2m_3^2\\
&+(2c_1+4c_2+2)m_3m_4-(2+2c_1)m_4^2+2c_1m_5m_6\\
\end{array}\end{equation*}
which are quadratic forms in $m_i$.

In Step 6,
using the Matlab program in Appendix B, we obtain the following SOS representation
\begin{equation}\label{2.16}
\begin{array}{ll}
\widehat{L}_1=\sum\limits_{k=1}^{8}p_k\widehat{R}_{k},\ \ \widehat{L}_2=\sum\limits_{k=1}^{8}q_k\widehat{R}_{k}-2(m_3-m_4)^2,
\end{array}
\end{equation}
where
$p_3=-\frac{329}{657},\ p_7= -\frac{329}{657},\
 q_1=- \frac{329}{657},\ q_5= -\frac{329}{657},\
c_2= -\frac{1}{1314},\
c_1=p_1=p_2=p_4=p_5=p_6=p_{8}=q_2=q_3=q_4=q_6=q_7=q_{8}=0.$
%
%
So, {\bf Problem II} is solved and thus Costa's EPI is proved.


\section{Third derivative of $H(X_t)$}
\label{sec-3}
In this section, we prove $D(3,n):\frac{\d^3}{\d^3t}H(X_t)\ge 0$ for $n=2,3,4$ with Procedure \ref{proc-H}.

\subsection{Compute the third order constraints}
\label{sec-31}
In this section, we give the third order constraints
\begin{equation}
\label{eq-3consn}
\C_{3,n}=\{R_{i,a,b,c}^{(3)}, R_{j}^{(0)},R_{k,a,b}^{(2)},\,:\,i=1,\ldots,955, j=1,\ldots,8,k=1,\ldots,20,a,b,c\in[n]\},
\end{equation}
where $R_{i,a,b,c}^{(3)}, R_{j}^{(0)},R_{k,a,b}^{(2)}$ are given in Lemmas \ref{lm-3cons1}, \ref{lm-3cons2}, \ref{lm-3cons3} below.
%
%

Similar to \eqref{eq-v2}, we introduce the notation
\begin{equation}\label{eq-v3}
{\mathcal V}_{a,b,c} =\{
\dfrac{\partial^h p_t}{\partial^{h_1} x_{a,t}\partial^{h_2} x_{b,t}\partial^{h_3} x_{c,t}}:
h=h_1+h_2+h_3 \in\{0,1,\cdots,5\}\}
\end{equation}
where $a,b,c$ are variables taking values in $[n]$.
Then $\P_{3,n}=\cup_{a=1}^n\cup_{b=1}^n\cup_{c=1}^n \V_{a,b,c}$.

\begin{lemma}
\label{lm-3cons1}
Based on Lemma \ref{lm-cons1}, we obtain 955 third order  constraints $R^{(3)}_{i,a,b,c},\ i=1,\ldots,955$,
in  $\R[\V_{a,b,c}]$ where, $a,b,c$ are variables taking values in $[n]$.
These constraints and the Maple program to compute them are given in Appendix $E_1$.
\end{lemma}
\begin{proof}
%
%
From Lemma \ref{lm-cons1}, these constraints are from
the monomials in $\R[\V_{a,b,c}]$ with total order 6 and degree 6.
In total, we have 3222 such monomials and  obtain 955 distinct constraints.
\end{proof}

%
\begin{lemma}
\label{lm-3cons2}
Based on Lemma \ref{lm-cons2}, we obtain 6 third order constraints
${R}^{(0)}_i =  \sum\limits_{a=1}^{n}\sum\limits_{b=1}^{n}\sum\limits_{c=1}^{n} \mathcal{R}^{(0)}_{i,a,b,c}$,
$i=1,\ldots,6$ in $\R[\P_{2,n}]$, which can be found in Appendix $E_2$.
\end{lemma}
\begin{proof}
From Lemma \ref{lm-cons2}, we need to consider all expressions
{$Q=\prod_{i=1}^{6}\nabla^{(k_i)}p_t$ satisfying $k_i\in \{0,1,\ldots,5\}$ and $\sum_{i=1}^{6}k_i=6$.}
There exist 10 such expressions:
\begin{equation*}
\begin{array}{l}
p_t^4\nabla^{(5)}p_t\cdot\nabla p_t,\  p_t^4\nabla^{(4)}p_t\nabla^{(2)}p_t,
p_t^3\nabla^{(4)}p_t\parallel\nabla p_t\parallel^2,\\
p_t^3\nabla^{(3)}p_t\cdot\nabla p_t\nabla^{2}p_t,
p_t^2\nabla^{(3)}p_t\cdot\nabla p_t\parallel\nabla p_t\parallel^2,\
p_t\nabla^{2}p_t\parallel\nabla p_t\parallel^4,\\
p_t^4\nabla^{(3)}p_t\cdot\nabla^{(3)}p_t,
p_t^2(\nabla^{2}p_t)^2 \parallel\nabla p_t\parallel^2,
p_t^3(\nabla^{2}p_t)^3,
\parallel\nabla p_t\parallel^6
.
\end{array}
\end{equation*}
From the above 10 expressions, we obtain 6 distinct constraints.
\end{proof}

Apply Lemma \ref{lm-cons3} to the second order constraints
$\C_{2,n}$ in \eqref{eq-2cons} to obtain third order constraints.
\begin{lemma}
\label{lm-3cons3}
We obtain 2 third order constraints
${R}^{(0)}_i =  \sum\limits_{a=1}^{n}\sum\limits_{b=1}^{n}\sum\limits_{c=1}^{n} \mathcal{R}^{(0)}_{i,a,b,c}$,
$i=7,8$
and 20 third order constraints
${R}^{(2)}_{i,a,b} = \sum\limits_{c=1}^{n} \mathcal{R}^{(2)}_{i,a,b,c}$,
$i=1,\ldots,20$,
where $\mathcal{R}^{(0)}_{i,a,b,c}$ and $\mathcal{R}^{(2)}_{i,a,b,c}$
can be found in Appendix $E_3$.
\end{lemma}
\begin{proof}
From \eqref{eq-2cons}, we need to consider two cases: $R^{(0)}_i$ and $R^{(2)}_{j,a,b}$.
First, we obtain 2 constraints from $R^{(0)}_i,i=1,2$ given in  Lemma \ref{lem-41}.
Take $R^{(0)}_1 = \sum_{a=1}^n\sum_{b=1}^n R^{(0)}_{1,a,b}$ as an example.
By direct computation, we obtain a third order constraint
{\small
$$\begin{array}{ll}
\displaystyle{\dfrac{\d}{\d t}\int_{\mathbf{R}^n}\frac{R^{(0)}_1}{p_t^3}\d x_t}
%
=\displaystyle{\int_{\mathbf{R}^n}\sum_{a=1}^{n}\sum_{b=1}^{n}\dfrac{\d}{\d t}\left(
\frac{R^{(0)}_{1,a,b}}{p_t^3}\right)\d x_t}
\overset{\eqref{H1}}{=}\displaystyle{\int_{\mathbf{R}^n}\sum_{a=1}^{n}\sum_{b=1}^{n}\sum_{c=1}^{n}
\frac{{\breve{R}}_{a,b,c}}{p_t^5}\d x_t}
=0,
\end{array}
$$}
where
{\small
$$\begin{array}{ll}
\breve{\mathcal{R}}_{a,b,c}&=\dfrac{p_t^4}{2}\dfrac{\partial^3p_t}{\partial x_{a,t}\partial^2x_{c,t}}\dfrac{\partial^3p_t}{\partial x_{a,t}\partial^2x_{b,t}}
+\dfrac{p_t^4}{2}\dfrac{\partial p_t}{\partial x_{a,t}}\dfrac{\partial^5p_t}{\partial x_{a,t}\partial^2x_{b,t}\partial^2x_{c,t}}
-\dfrac{p_t^3}{2}\dfrac{\partial p_t}{\partial x_{a,t}}\dfrac{\partial^3p_t}{\partial x_{a,t}\partial^2x_{b,t}}\dfrac{\partial^2p_t}{\partial^2x_{c,t}}\\[0.2cm]
&+\dfrac{p_t^4}{2}\dfrac{\partial^4p_t}{\partial^2 x_{a,t}\partial^2x_{c,t}}\dfrac{\partial^2p_t}{\partial^2x_{b,t}}
+\dfrac{p_t^4}{2}\dfrac{\partial^2p_t}{\partial^2 x_{a,t}}\dfrac{\partial^4p_t}{\partial^2x_{b,t}\partial^2x_{c,t}}
-\dfrac{p_t^3}{2}\dfrac{\partial^2 p_t}{\partial^2 x_{a,t}}\dfrac{\partial^2p_t}{\partial^2x_{b,t}}\dfrac{\partial^2p_t}{\partial^2x_{c,t}}
\\[0.2cm]
&
-\dfrac{p_t^3}{2}\dfrac{\partial^4 p_t}{\partial^2 x_{a,t}\partial^2x_{c,t}}\left(\dfrac{\partial p_t}{\partial x_{b,t}}\right)^2
-p_t^3\dfrac{\partial^2 p_t}{\partial^2 x_{a,t}}\dfrac{\partial p_t}{\partial x_{b,t}}\dfrac{\partial^3p_t}{\partial x_{b,t}\partial^2x_{c,t}}
+p_t^2\dfrac{\partial^2 p_t}{\partial^2 x_{a,t}}\left(\dfrac{\partial p_t}{\partial x_{b,t}}\right)^2\dfrac{\partial^2p_t}{\partial^2x_{c,t}}.
\end{array}
$$}

Second, we obtain 20 third order constraints from
the 60 second order  constraints of the form $R^{(2)}_{6,a,b}$ in Lemma \ref{lem-42} (only 17 of them are given in Lemma \ref{lem-42}. See the proof of Lemma \ref{lem-42} for explanation).
%
For instance, consider $R^{(2)}_{6,a,b}=
p_t^2\frac{\partial^3p_t}{\partial^2 x_{a,t}\partial x_{b,t}}\frac{\partial p_t}{\partial x_{a,t}}
+\frac{\partial^2p_t}{\partial x_{a,t}\partial x_{b,t}}(p_t^2\frac{\partial^2 p_t}{\partial^2 x_{a,t}}-p_t(\frac{\partial p_t}{\partial x_{a,t}})^2)$ in Appendix C,
which satisfies
$\int_{\mathbf{R}^n}\frac{R^{(2)}_{6,a,b}}{p_t^3}\d x_t=0$. We obtain a third order constraint:
%
%
{\small
$$\begin{array}{ll}
\displaystyle{\dfrac{\d}{\d t}\int_{\mathbf{R}^n}\frac{R^{(2)}_{6,a,b}}{p_t^3}\d x_t}
=\displaystyle{\int_{\mathbf{R}^n}\dfrac{\d}{\d t}\left(\frac{R^{(2)}_{6,a,b}}{p_t^3}\right)\d x_t}
\overset{\eqref{H1}}{=}\displaystyle{\int_{\mathbf{R}^n}\sum_{c=1}^{n}\frac{\bar{\mathcal{R}}_{a,b,c}}{p_t^5}\d x_t}=0,
\end{array}
$$}
where
{\small
$$\begin{array}{ll}
\bar{\mathcal{R}}_{a,b,c}\!\!\!&=\frac{p_t^4}{2}\frac{\partial^5p_t}{\partial^2 x_{a,t}\partial x_{b,t}\partial^2x_{c,t}}\frac{\partial p_t}{\partial x_{a,t}}
+\frac{p_t^4}{2}\frac{\partial^3p_t}{\partial^2 x_{a,t}\partial x_{b,t}}\frac{\partial^3 p_t}{\partial x_{a,t}\partial^2x_{c,t}}
-\frac{p_t^3}{2}\frac{\partial^3p_t}{\partial^2 x_{a,t}\partial x_{b,t}}\frac{\partial p_t}{\partial x_{a,t}}\frac{\partial^2 p_t}{\partial^2x_{c,t}}\\
&+\frac{p_t^4}{2}\frac{\partial^4p_t}{\partial x_{a,t}\partial x_{b,t}\partial^2x_{c,t}}\frac{\partial^2 p_t}{\partial^2 x_{a,t}}
-\frac{p_t^3}{2}\frac{\partial^4p_t}{\partial x_{a,t}\partial x_{b,t}\partial^2x_{c,t}}\left(\frac{\partial p_t}{\partial x_{a,t}}\right)^2
+\frac{p_t^4}{2}\frac{\partial^2p_t}{\partial x_{a,t}\partial x_{b,t}}\frac{\partial^4 p_t}{\partial^2 x_{a,t}\partial^2x_{c,t}}\\[0.2cm]
&-\frac{p_t^3}{2}\frac{\partial^2p_t}{\partial x_{a,t}\partial x_{b,t}}\frac{\partial^2 p_t}{\partial^2 x_{a,t}}\frac{\partial^2 p_t}{\partial^2x_{c,t}}
-p_t^3\frac{\partial^2p_t}{\partial x_{a,t}\partial x_{b,t}}\frac{\partial p_t}{\partial x_{a,t}}\frac{\partial^3 p_t}{\partial x_{a,t}\partial^2x_{c,t}}
+p_t^2\frac{\partial^2p_t}{\partial x_{a,t}\partial x_{b,t}}\left(\frac{\partial p_t}{\partial x_{a,t}}\right)^2\frac{\partial^2 p_t}{\partial^2x_{c,t}}.
\end{array}$$
}
Removing the repeated ones in the 60 constraints, we obtain 20 third order constraints.
\end{proof}

\subsection{Proof of $D(3,n)\geq0$ for $n=2,3,4$}
\label{section31}

We first compute $F_{3,n}$ in \eqref{eq-tt1}:
\begin{equation}
\label{3.1}
F_{3,n} = \displaystyle{\dfrac{p_t^5}{2}\dfrac{\d^2}{\d  t^2}
\left(\dfrac{\|\nabla p_t\|^2}{p_t}\right)}
 \overset{\eqref{H1}}{=}
\displaystyle{\int_{\mathbf{R}^n}\sum\limits_{a=1}^{n}\sum\limits_{b=1}^{n}\sum\limits_{c=1}^{n}
F_{3,a,b,c}}
\end{equation}
where
{\small
\begin{equation*}\begin{array}{ll}
F_{3,a,b,c}&=\dfrac{p_t^4}{4}\dfrac{\partial^3p_t}{\partial x_{a,t}\partial^2x_{c,t}}\dfrac{\partial^3p_t}{\partial x_{a,t}\partial^2x_{b,t}}
-\dfrac{p_t^3}{4}\dfrac{\partial p_t}{\partial x_{a,t}}\dfrac{\partial^3p_t}{\partial x_{a,t}\partial^2x_{b,t}}\dfrac{\partial^2p_t}{\partial^2x_{c,t}}\\[0.2cm]
&
+\dfrac{p_t^4}{4}\dfrac{\partial p_t}{\partial x_{a,t}}\dfrac{\partial^5p_t}{\partial x_{a,t}\partial^2x_{b,t}\partial^2x_{c,t}}
-\dfrac{p_t^3}{4}\dfrac{\partial p_t}{\partial x_{a,t}}\dfrac{\partial^3p_t}{\partial x_{a,t}\partial^2x_{c,t}}\dfrac{\partial^2p_t}{\partial^2x_{b,t}}\\[0.2cm]
&
+\dfrac{p_t^2}{4}\left(\dfrac{\partial p_t}{\partial x_{a,t}}\right)^2\dfrac{\partial^2p_t}{\partial^2x_{b,t}}\dfrac{\partial^2p_t}{\partial^2x_{c,t}}
-\dfrac{p_t^3}{8}\left(\dfrac{\partial p_t}{\partial x_{a,t}}\right)^2\dfrac{\partial^4p_t}{\partial^2x_{b,t}\partial^2x_{c,t}}.
\end{array}\end{equation*}}

\subsubsection{Proof of $D(3,2)\geq0$}
\label{section3.1}

The proof follows Procedure \ref{proc-H} with  $F_{3,2}$ given in \eqref{3.1} as input.
To make the proof explicit, we will give the key expressions.

In Step 1,
from \eqref{eq-3consn}, the constraints are $\C_{3,2}=\{{R}^{(0)}_{i},R^{(3)}_{j,a,b,c}\,:\,i=1,\ldots,8; j=1,\ldots,955;a,b,c\in[2]\}$.
The constraints $R^{(2)}_{k,a,b}$'s are not used in this case.
Removing the repeated ones, we have $N_1=143$.

In Step 2, the new variables are $\M_{3,2}$ and are listed  in the lexicographical monomial order:
$$\begin{array}{ll}
m_1=p_t^2\frac{\partial p_t^3}{\partial^3 x_{2,t}},
m_2=p_t^2\frac{\partial^3 p_t}{\partial x_{1,t}\partial^2 x_{2,t}},
m_3=p_t^2\frac{\partial^3 p_t}{\partial^2 x_{1,t}\partial x_{2,t}},
m_4=p_t^2\frac{\partial p_t^3}{\partial^3 x_{1,t}},\\
m_5=p_t\frac{\partial^2 p_t}{\partial^2 x_{2,t}}\frac{\partial p_t}{\partial x_{2,t}},
m_6=p_t\frac{\partial^2 p_t}{\partial^2 x_{2,t}}\frac{\partial p_t}{\partial x_{1,t}},
m_7=p_t\frac{\partial^2 p_t}{\partial x_{1,t}\partial x_{2,t}}\frac{\partial p_t}{\partial x_{2,t}},\\
m_8=p_t\frac{\partial^2 p_t}{\partial x_{1,t}\partial x_{2,t}}\frac{\partial p_t}{\partial x_{1,t}},
m_9=p_t\frac{\partial^2 p_t}{\partial x_{1,t}^2}\frac{\partial p_t}{\partial x_{2,t}},
m_{10}=p_t\frac{\partial^2 p_t}{\partial x_{1,t}^2}\frac{\partial p_t}{\partial x_{1,t}},\\
m_{11}=\left(\frac{\partial p_t}{\partial x_{2,t}}\right)^3,
m_{12}=\left(\frac{\partial p_t}{\partial x_{2,t}}\right)^2\frac{\partial p_t}{\partial x_{1,t}},
m_{13}=\frac{\partial p_t}{\partial x_{2,t}}\left(\frac{\partial p_t}{\partial x_{1,t}}\right)^2,
m_{14}=\left(\frac{\partial p_t}{\partial x_{1,t}}\right)^3.
\end{array}$$

In Step 3, we  obtain $\C_{3,2,1}$ and $\C_{3,2,2}$ which contain
48 and 52 constraints, respectively.

In Step 4, there exist 15 intrinsic constraints:
{\small
\begin{equation*}
\begin{array}{ll}
 m_{5}m_{8}=m_{6}m_{7}, m_{5}m_{10}=m_{6}m_{9}, m_{5}m_{12}=m_{6}m_{11}, m_{5}m_{13}=m_{6}m_{12}, m_{5}m_{14}=m_{6}m_{13},\\
  m_{7}m_{10}=m_{8}m_{9},  m_{7}m_{12}=m_{8}m_{11}, m_{7}m_{13}=m_{8}m_{12}, m_{7}m_{14}=m_{8}m_{13}, m_{9}m_{12}=m_{10}m_{11},\\
   m_{9}m_{13}=m_{10}m_{12},
 m_{9}m_{14}=m_{10}m_{13}, m_{11}m_{13}=m_{12}^2, m_{11}m_{14}=m_{12}m_{13}, m_{12}m_{14}=m_{13}^2.
\end{array}
\end{equation*}}
Thus,  ${\widehat{\C}}_{3,2,1}$ contains 63 constraints and $N_2 = 63$.

In Step 5, eliminating the non-quadratic monomials in $F_{3,2}$
using $\C_{3,2,2}$ to obtain a quadratic form in $m_i$ and then simplifying the quadratic form using $\C_{3,2,1}$, we have
{\small
\begin{equation*}
\begin{array}{ll}
\widehat{F}_{3,2}\!\!\!\!
&=
F_{3,2}-(\frac{3}{4}\widehat{R}_{17}-\frac{1}{6}\widehat{R}_{12}-\frac{1}{6}\widehat{R}_{13}+\frac{7}{6}\widehat{R}_{18}-\frac{1}{2}\widehat{R}_{32}
-\frac{1}{2}\widehat{R}_{34}-\frac{5}{8}\widehat{R}_{35}-\frac{1}{2}\widehat{R}_{40}\\
&-\frac{1}{12}\widetilde{R}_{2}-\frac{1}{8}\widetilde{R}_{5}
-\frac{1}{4}\widetilde{R}_{6}+\frac{1}{2}\widetilde{R}_{7}
+\frac{1}{4}\widetilde{R}_{8}+\frac{1}{2}\widetilde{R}_{18}+\frac{1}{4}\widetilde{R}_{19}
-\frac{1}{8}\widetilde{R}_{39}-\frac{1}{4}\widetilde{R}_{46}+\frac{1}{2}\widetilde{R}_{48}-\frac{1}{8}\widetilde{R}_{49}+\frac{1}{4}\widetilde{R}_{53})\\
&=\frac{1}{2}m_{1}^2-m_{1}m_{5}+\frac{3}{2}m_{2}^2-3m_{2}m_{6}+\frac{3}{2}m_{3}^2+\frac{1}{2}m_{4}^2-2m_{4}m_{6}-m_{4}m_{7}
-m_{4}m_{10}-\frac{1}{2}m_{5}^2+\frac{3}{2}m_{6}^2-3m_{7}^2\\
&
\quad -2m_{7}m_{10}+3m_{8}^2-\frac{5}{2}m_{9}^2-\frac{3}{2}m_{9}m_{11}+21m_{9}m_{13}-\frac{1}{2}m_{10}^2
+\frac{3}{5}m_{11}^2+3m_{12}^2-15m_{13}^2+\frac{3}{5}m_{14}^2.
\end{array}
\end{equation*}}

In Step 6, using the Matlab program in Appendix B
with $\widehat{F}_{3,2}$ and $\widehat{\C}_{3,2,1}$ as input,
we  find an SOS representation \eqref{eq-tt3} for $\widehat{F}_{3,2}$. Thus, $D(3,2)$ is proved.
The Maple  program to prove $D(3,2)$ can be found in
\begin{verbatim}
http://www.mmrc.iss.ac.cn/~xgao/software/maple-epid32.zip
\end{verbatim}

\subsubsection{Proof of $D(3,3)\geq0$}
The proof follows Procedure \ref{proc-H} with  $F_{3,3}$ given in \eqref{3.1} as input.

In Step 1, from \eqref{eq-3consn},
the constraints are:
$\C_{3,3}=\{R_{i,a,b,c}^{(3)}, R_{j}^{(0)},R_{k,a,b}^{(2)},\,:\,i=1,\ldots,955, j=1,\ldots,8,k=1,\ldots,20,a,b,c\in[3]\}$. Removing the repeated ones, we have $N_1=1029$.

%

In Step 2, the new variables are $\M_{3.3}=\{m_i,i=1,\ldots,38\}$ listed   in the lexicographical monomial order.

In Step 3, we obtain   $\C_{3,3,1}$ and $\C_{3,3,2}$, which contain
350 and 328 constraints, respectively.

In Step 4, there exist 189 intrinsic constraints.
In total, ${\widehat{\C}}_{3,3,1}$ contains 539 constraints.
Using $\R$-Gaussian elimination in $\span_\R({\widehat{\C}}_{3,3,1})$
shows that 512 of these 539 constraints are linearly independent,  so $N_2 = 512$.

In Step 5,
eliminating the non-quadratic monomials in $F_{3,3}$
using $\C_{3,3,2}$ and then simplify the expression using $\C_{3,3,1}$, we have
{\small
$$\begin{array}{ll}
\widehat{F}_{3,3}\!\!\!\!&=
F_{3,3}-(-\frac{1}{8}\widetilde{R}_{30}-\frac{1}{8}\widetilde{R}_{31}-\frac{1}{8}\widetilde{R}_{32}-\frac{1}{4}\widetilde{R}_{39}
-\frac{1}{4}\widetilde{R}_{41}-\frac{1}{4}\widetilde{R}_{42}+\frac{1}{4}\widetilde{R}_{43}+\frac{1}{4}\widetilde{R}_{44}+\frac{1}{4}\widetilde{R}_{45}
+\frac{1}{2}\widetilde{R}_{65}+\frac{1}{2}\widetilde{R}_{66}+\frac{1}{4}\widetilde{R}_{67}\\&
+\frac{1}{2}\widetilde{R}_{69}+\frac{1}{4}\widetilde{R}_{71}
+\frac{1}{4}\widetilde{R}_{72}-\frac{1}{4}\widetilde{R}_{78}-\frac{1}{4}\widetilde{R}_{79}-\frac{1}{4}\widetilde{R}_{81}+\frac{1}{4}\widetilde{R}_{89}
+\frac{1}{4}\widetilde{R}_{90}+\frac{1}{2}\widetilde{R}_{91}+\frac{1}{4}\widetilde{R}_{93}+\frac{1}{2}\widetilde{R}_{94}+\frac{1}{2}\widetilde{R}_{96}
+\frac{1}{2}\widetilde{R}_{134}\\&
+\frac{1}{2}\widetilde{R}_{136}+\frac{1}{2}\widetilde{R}_{155}-\frac{1}{8}\widetilde{R}_{206}-\frac{1}{8}\widetilde{R}_{209}
-\frac{1}{8}\widetilde{R}_{210}-\frac{1}{4}\widetilde{R}_{215}-\frac{1}{8}\widetilde{R}_{242}-\frac{1}{8}\widetilde{R}_{243}-\frac{1}{8}\widetilde{R}_{244}
-\frac{1}{4}\widetilde{R}_{306}-\frac{1}{4}\widetilde{R}_{328}-\frac{21}{4}\widehat{R}_{313}\\&
-\frac{1}{2}\widehat{R}_{153}-\frac{1}{2}\widehat{R}_{155}
-\frac{1}{2}\widehat{R}_{161}+\frac{1}{2}\widehat{R}_{215}-\frac{1}{6}\widehat{R}_{39}-\frac{1}{6}\widehat{R}_{41}-\frac{1}{6}\widehat{R}_{42}
-\frac{1}{6}\widehat{R}_{43}-\frac{1}{6}\widehat{R}_{45}-\frac{1}{6}\widehat{R}_{46}+\frac{1}{2}\widehat{R}_{71}+\frac{1}{2}\widehat{R}_{73}
+\frac{1}{2}\widehat{R}_{77}\\&
-\frac{1}{2}\widehat{R}_{95}+\frac{5}{2}\widehat{R}_{258}-\frac{1}{2}\widehat{R}_{248}+\frac{1}{2}\widehat{R}_{265}
+\frac{1}{2}\widehat{R}_{266}+\frac{1}{2}\widehat{R}_{267}-\frac{1}{2}\widehat{R}_{280}-\frac{1}{2}\widehat{R}_{281}-\frac{1}{2}\widehat{R}_{282}
+\frac{1}{3}\widehat{R}_{283}+\frac{1}{3}\widehat{R}_{284}+\frac{1}{3}\widehat{R}_{285}\\&
-\frac{1}{2}\widehat{R}_{296}-\frac{1}{2}\widehat{R}_{303}
-\frac{1}{2}\widehat{R}_{304}-\frac{1}{2}\widehat{R}_{305}-\frac{1}{2}\widehat{R}_{307}-\frac{1}{2}\widehat{R}_{321}+\frac{1}{2}\widehat{R}_{338}
-4\widehat{R}_{346}-\frac{1}{2}\widehat{R}_{347}-\frac{1}{2}\widehat{R}_{348}-\widehat{R}_{349}-\frac{1}{2}\widehat{R}_{350}\\&
-\frac{3}{4}\widehat{R}_{175}+\frac{5}{8}\widehat{R}_{333}
)\\
&=21m_{26}m_{34}+21m_{27}m_{37}-\frac{1}{2}m_{28}^2+3m_{30}^2+3m_{31}^2-15m_{32}^2+18m_{33}^2-15m_{34}^2+3m_{36}^2
-15m_{37}^2+3m_{15}^2\\&
-2m_{7}m_{12}-m_{7}m_{14}+\frac{1}{2}m_{1}^2-m_{1}m_{11}+3m_{19}^2-3m_{14}^2-2m_{14}m_{21}-9m_{18}^2-2m_{24}m_{28}
+3m_{25}^2-m_{7}m_{21}\\&
-6m_{8}m_{13}+\frac{3}{2}m_{6}^2+3m_{16}^2+3m_{5}^2+\frac{3}{5}m_{35}^2-3m_{17}^2+\frac{3}{2}m_{3}^2-2m_{17}m_{28}
+\frac{3}{2}m_{12}^2+\frac{1}{2}m_{10}^2-2m_{10}m_{13}-m_{10}m_{17}\\&
-18m_{23}m_{33}-3m_{24}^2-3m_{2}m_{12}+3m_{23}^2-9m_{26}m_{32}+6m_{9}m_{12}
+\frac{3}{5}m_{38}^2-3m_{3}m_{13}+21m_{20}m_{32}+9m_{20}m_{34}\\&
-\frac{1}{2}m_{21}^2-2m_{10}m_{22}-m_{10}m_{24}+\frac{3}{2}m_{13}^2-3m_{8}m_{22}
-\frac{5}{2}m_{20}^2-m_{10}m_{28}-6m_{6}m_{20}-\frac{1}{2}m_{11}^2+\frac{1}{2}m_{7}^2+\frac{3}{2}m_{9}^2\\&
+\frac{3}{2}m_{4}^2+\frac{3}{5}m_{29}^2
-\frac{3}{2}m_{26}m_{29}-\frac{5}{2}m_{26}^2-\frac{5}{2}m_{27}^2+\frac{3}{2}m_{8}^2-\frac{3}{2}m_{27}m_{35}+\frac{3}{2}m_{2}^2+\frac{3}{2}m_{22}^2
-\frac{3}{2}m_{20}m_{29}.
\end{array}$$}

In Step 6, using the Matlab program in Appendix B
with $\widehat{F}_{3,3}$ and $\widehat{\C}_{3,3,1}$ as input, we  find an SOS representation \eqref{eq-tt3} for $\widehat{F}_{3,3}$.
Thus, $D(3,3)$ is proved.
%
%
The Maple  program to prove $D(3,3)$ can be found in
\begin{verbatim}
http://www.mmrc.iss.ac.cn/~xgao/software/maple-epid33.zip
\end{verbatim}

\subsubsection{Proof of $D(3,4)\geq0$}
The proof  follows Procedure \ref{proc-H} with  $F_{3,4}$ given in \eqref{3.1} as input.

In Step 1, from \eqref{eq-3consn},
$\C_{3,4}=\{R_{i,a,b,c}^{(3)}, R_{j}^{(0)},R_{k,a,b}^{(2)},\,:\,i=1,\ldots,955, j=1,\ldots,8,k=1,\ldots,20,a,b,c\in[4]\}$.
%
Removing the repeated ones, we have $N_1=3172$.

In Step 2, the new variables are $M_{3,4}=\{m_i, i=1,\ldots,80\}$
  listed  in the lexicographical monomial order.

In Step 3, we obtain  $\C_{3,4,1}$ and $\C_{3,4,2}$ which contain
1120 and 975 constraints, respectively.

In Step 4, there exist 1080 intrinsic constraints.
In total, ${\widehat{\C}}_{3,4,1}$ contains 2200 constraints.
Only 1966 constraints in ${\widehat{\C}}_{3,4,1}$ are $\R$-linearly independent,
so  $N_2 = 1966$.

In Step 5, eliminating the non-quadratic monomials in $F_{3,4}$
using $\C_{3,4,2}$ to obtain a quadratic form in $m_i$  and then simplifying the quadratic form with $\C_{3,4,1}$, we have
{\footnotesize
$$\begin{array}{ll}
&F_{3,4}-\widehat{F}_{3,4}=\\
&\frac{1}{2}\widetilde{R}_{410}+\frac{1}{2}\widetilde{R}_{411}+\frac{1}{2}\widetilde{R}_{412}+\frac{1}{2}\widetilde{R}_{413}+\frac{1}{2}\widetilde{R}_{414}
+\frac{1}{2}\widetilde{R}_{415}+\frac{1}{2}\widetilde{R}_{416}+\frac{1}{2}\widetilde{R}_{417}+\frac{1}{2}\widetilde{R}_{147}+\frac{1}{4}\widetilde{R}_{148}
+\frac{1}{2}\widetilde{R}_{150}+\frac{1}{2}\widetilde{R}_{153}+\frac{1}{4}\widetilde{R}_{155}\\&
+\frac{1}{4}\widetilde{R}_{156}+\frac{1}{2}\widetilde{R}_{159}
+\frac{1}{4}\widetilde{R}_{163}+\frac{1}{4}\widetilde{R}_{164}+\frac{1}{4}\widetilde{R}_{165}-\frac{1}{4}\widetilde{R}_{180}-\frac{1}{4}\widetilde{R}_{181}
-\frac{1}{4}\widetilde{R}_{183}-\frac{1}{4}\widetilde{R}_{184}-\frac{1}{4}\widetilde{R}_{186}+\frac{1}{2}\widetilde{R}_{233}-\frac{1}{4}\widetilde{R}_{101}
+\frac{1}{4}\widetilde{R}_{102}\\&
+\frac{1}{4}\widetilde{R}_{103}+\frac{1}{4}\widetilde{R}_{104}+\frac{1}{4}\widetilde{R}_{105}+\frac{1}{2}\widetilde{R}_{142}
+\frac{1}{2}\widetilde{R}_{144}-\frac{1}{8}\widetilde{R}_{76}-\frac{1}{8}\widetilde{R}_{77}-\frac{1}{8}\widetilde{R}_{78}-\frac{1}{8}\widetilde{R}_{79}
-\frac{1}{4}\widetilde{R}_{92}-\frac{1}{4}\widetilde{R}_{94}-\frac{1}{4}\widetilde{R}_{97}-\frac{1}{4}\widetilde{R}_{98}-\frac{1}{4}\widetilde{R}_{100}\\&
-\frac{1}{4}\widetilde{R}_{189}+\frac{1}{4}\widetilde{R}_{210}+\frac{1}{4}\widetilde{R}_{211}+\frac{1}{4}\widetilde{R}_{213}+\frac{1}{2}\widetilde{R}_{216}
+\frac{1}{4}\widetilde{R}_{218}+\frac{1}{4}\widetilde{R}_{220}+\frac{1}{2}\widetilde{R}_{222}+\frac{1}{2}\widetilde{R}_{224}+\frac{1}{4}\widetilde{R}_{227}
+\frac{1}{2}\widetilde{R}_{228}+\frac{1}{2}\widetilde{R}_{230}+\frac{3}{2}\widehat{R}_{262}\\&
-\frac{1}{2}\widehat{R}_{416}-\frac{1}{2}\widehat{R}_{418}
-\frac{1}{2}\widehat{R}_{420}-\frac{1}{2}\widehat{R}_{422}-\frac{1}{2}\widehat{R}_{424}-\frac{1}{2}\widehat{R}_{425}-\frac{1}{2}\widehat{R}_{426}
-\frac{1}{2}\widehat{R}_{427}-\frac{1}{2}\widehat{R}_{429}-\frac{1}{2}\widehat{R}_{430}-\frac{1}{2}\widehat{R}_{431}-\frac{1}{2}\widehat{R}_{432}
-\frac{1}{2}\widehat{R}_{437}\\&
-\frac{1}{2}\widehat{R}_{439}-\frac{1}{2}\widehat{R}_{441}-\frac{1}{2}\widehat{R}_{443}-\frac{1}{2}\widehat{R}_{444}
-\frac{1}{2}\widehat{R}_{445}-\frac{1}{2}\widehat{R}_{449}-\frac{1}{2}\widehat{R}_{451}+\frac{1}{2}\widehat{R}_{186}+\frac{1}{2}\widehat{R}_{189}
+\frac{1}{2}\widehat{R}_{191}+\frac{1}{2}\widehat{R}_{201}+\frac{1}{2}\widehat{R}_{203}+\frac{1}{2}\widehat{R}_{209}\\&
-\frac{1}{2}\widehat{R}_{212}
-\frac{1}{2}\widehat{R}_{213}-\frac{1}{2}\widehat{R}_{214}-\frac{1}{2}\widehat{R}_{215}-\frac{1}{2}\widehat{R}_{224}-\frac{1}{2}\widehat{R}_{225}
-\frac{1}{2}\widehat{R}_{226}-\frac{1}{2}\widehat{R}_{227}-\frac{3}{2}\widehat{R}_{877}-\frac{5}{8}\widehat{R}_{818}-\frac{3}{2}\widehat{R}_{876}
-3\widehat{R}_{875}\frac{1}{2}-\frac{1}{4}\widetilde{R}_{972}\\&
-\frac{1}{4}\widetilde{R}_{973}-\frac{1}{4}\widetilde{R}_{974}-\frac{1}{4}\widetilde{R}_{975}
+\frac{1}{2}\widetilde{R}_{920}+\frac{1}{2}\widetilde{R}_{921}-\frac{1}{4}\widetilde{R}_{922}+\frac{1}{2}\widetilde{R}_{923}-\frac{1}{4}\widetilde{R}_{924}
-\frac{1}{4}\widetilde{R}_{925}+\frac{1}{2}\widetilde{R}_{926}-\frac{1}{4}\widetilde{R}_{927}-\frac{1}{8}\widetilde{R}_{696}-\frac{1}{8}\widetilde{R}_{697}\\&
-\frac{1}{8}\widetilde{R}_{698}-\frac{1}{8}\widetilde{R}_{699}-\frac{1}{8}\widetilde{R}_{700}-\frac{1}{8}\widetilde{R}_{701}-\frac{1}{8}\widetilde{R}_{541}
-\frac{1}{8}\widetilde{R}_{547}-\frac{1}{8}\widetilde{R}_{548}-\frac{1}{8}\widetilde{R}_{551}-\frac{1}{8}\widetilde{R}_{552}-\frac{1}{8}\widetilde{R}_{553}
-\frac{1}{4}\widetilde{R}_{561}-\frac{1}{4}\widetilde{R}_{563}-\frac{1}{4}\widetilde{R}_{568}\\&
-\frac{1}{4}\widetilde{R}_{570}-\frac{9}{2}\widehat{R}_{1009}
-\frac{9}{2}\widehat{R}_{1011}-\frac{1}{2}\widehat{R}_{1069}-\frac{1}{2}\widehat{R}_{1070}-\frac{1}{2}\widehat{R}_{1071}-\frac{1}{2}\widehat{R}_{1072}
-\frac{1}{2}\widehat{R}_{1093}-\frac{1}{2}\widehat{R}_{1094}-\frac{1}{2}\widehat{R}_{1095}-\frac{1}{2}\widehat{R}_{1096}+\frac{1}{4}\widehat{R}_{1105}\\&
+\frac{1}{4}\widehat{R}_{1106}+\frac{1}{4}\widehat{R}_{1107}+\frac{1}{4}\widehat{R}_{1111}-\frac{3}{2}\widehat{R}_{878}-\frac{1}{2}\widehat{R}_{941}
-\frac{1}{2}\widehat{R}_{942}-\frac{1}{2}\widehat{R}_{943}-\frac{1}{2}\widehat{R}_{944}-\frac{1}{2}\widehat{R}_{945}-\frac{1}{2}\widehat{R}_{946}
+\frac{1}{3}\widehat{R}_{947}+\frac{1}{3}\widehat{R}_{948}+\frac{1}{3}\widehat{R}_{949}\\&
+\frac{1}{3}\widehat{R}_{950}+\frac{1}{3}\widehat{R}_{951}
+\frac{1}{3}\widehat{R}_{952}-3\widehat{R}_{989}-3\widehat{R}_{990}-3\widehat{R}_{991}-3\widehat{R}_{995}-\frac{1}{2}\widehat{R}_{828}
-\frac{1}{2}\widehat{R}_{829}-\frac{1}{2}\widehat{R}_{830}-\frac{1}{2}\widehat{R}_{831}+\frac{1}{2}\widehat{R}_{895}+\frac{1}{2}\widehat{R}_{896}\\&
+\frac{1}{2}\widehat{R}_{897}+\frac{1}{2}\widehat{R}_{898}+\frac{1}{2}\widehat{R}_{899}+\frac{1}{2}\widehat{R}_{900}+\frac{1}{2}\widehat{R}_{670}
+\frac{1}{2}\widehat{R}_{672}+\frac{1}{2}\widehat{R}_{674}+\frac{1}{2}\widehat{R}_{675}+\frac{1}{2}\widehat{R}_{677}+\frac{1}{2}\widehat{R}_{678}
+\frac{1}{2}\widehat{R}_{685}+\frac{1}{2}\widehat{R}_{687}-\frac{5}{8}\widehat{R}_{824}\\&
+\frac{3}{2}\widehat{R}_{260}-\frac{1}{6}\widehat{R}_{91}
-\frac{1}{6}\widehat{R}_{94}-\frac{1}{6}\widehat{R}_{96}-\frac{1}{6}\widehat{R}_{97}-\frac{1}{6}\widehat{R}_{98}-\frac{1}{6}\widehat{R}_{99}
-\frac{1}{6}\widehat{R}_{103}-\frac{1}{6}\widehat{R}_{105}-\frac{1}{6}\widehat{R}_{106}-\frac{1}{6}\widehat{R}_{107}-\frac{1}{6}\widehat{R}_{109}
-\frac{1}{6}\widehat{R}_{110}+\frac{3}{2}\widehat{R}_{261}\\&
-\frac{5}{8}\widehat{R}_{814}+\frac{1}{4}\widehat{R}_{505}+\frac{1}{4}\widehat{R}_{513}
+\frac{1}{4}\widehat{R}_{532}+\frac{1}{4}\widehat{R}_{536}-\frac{1}{2}\widehat{R}_{564}-\frac{1}{2}\widehat{R}_{565}-\frac{1}{2}\widehat{R}_{566}
-\frac{1}{2}\widehat{R}_{574}-\frac{9}{2}\widehat{R}_{1012}-\frac{9}{2}\widehat{R}_{1010}+\frac{3}{2}\widehat{R}_{263}-\frac{5}{8}\widehat{R}_{813}.
\end{array}$$}
$\widehat{F}_{3,4}$ can be written as a quadratic form in $m_i$:
{\footnotesize
$$\begin{array}{ll}
&\widehat{F}_{3,4}=
-18m_{53}m_{69}+3m_{54}^2-18m_{54}m_{75}-3m_{55}^2-3m_{42}^2-2m_{55}m_{60}-6m_{10}m_{37}-6m_{10}m_{49}+3m_{56}^2
\\&
-m_{17}m_{51}+\frac{1}{2}m_{11}^2-3m_{2}m_{22}-9m_{47}^2-3m_{46}^2-18m_{41}m_{66}+3m_{6}^2-9m_{34}^2+3m_{53}^2-\frac{1}{2}m_{38}^2+3m_{45}^2
-m_{20}m_{55}\\&
-m_{20}m_{60}-2m_{20}m_{24}-m_{20}m_{33}+3m_{26}^2-2m_{17}m_{23}-m_{17}m_{29}+\frac{1}{2}m_{17}^2-6m_{16}m_{50}+6m_{16}m_{22}
+3m_{78}^2\\&
-15m_{79}^2-6m_{18}m_{40}-18m_{45}m_{67}-3m_{4}m_{24}-2m_{25}m_{38}-3m_{3}m_{23}-15m_{76}^2+18m_{75}^2-15m_{74}^2+3m_{73}^2
\\&
+3m_{72}^2+3m_{15}^2-15m_{70}^2+\frac{1}{2}m_{1}^2-m_{1}m_{21}+3m_{28}^2+3m_{41}^2-\frac{1}{2}m_{51}^2-3m_{18}m_{52}+21m_{50}m_{74}
+9m_{50}m_{76}\\&
+3m_{31}^2+3m_{48}^2-9m_{30}^2+18m_{69}^2-15m_{68}^2+3m_{36}^2-2m_{42}m_{51}+18m_{67}^2+6m_{14}m_{22}+18m_{66}^2-15m_{65}^2
\\&
+3m_{64}^2-3m_{13}m_{40}-6m_{13}m_{24}+3m_{63}^2+3m_{62}^2-3m_{12}m_{39}+21m_{59}m_{79}-\frac{1}{2}m_{60}^2-6m_{12}m_{23}\\&
-9m_{58}m_{74}+21m_{58}m_{76}-m_{11}m_{38}-9m_{57}m_{65}-m_{11}m_{25}-9m_{57}m_{68}+21m_{57}m_{70}-2m_{11}m_{22}+\frac{3}{2}m_{23}^2\\&
+\frac{3}{2}m_{2}^2-m_{17}m_{42}-2m_{46}m_{60}+\frac{1}{2}m_{20}^2-m_{20}m_{46}-2m_{20}m_{52}-2m_{29}m_{51}+6m_{19}m_{23}+3m_{27}^2-2m_{33}m_{60}\\&
-3m_{33}^2-2m_{17}m_{39}+6m_{19}m_{39}+21m_{37}m_{65}+9m_{37}m_{68}-3m_{25}^2+3m_{43}^2+3m_{9}^2+3m_{32}^2-6m_{8}m_{37}\\&
-6m_{18}m_{24}-9m_{49}m_{65}+21m_{49}m_{68}+9m_{37}m_{70}-\frac{1}{2}m_{21}^2-9m_{35}^2+3m_{7}^2-2m_{20}m_{40}+9m_{49}m_{70}-3m_{29}^2+3m_{44}^2\\&
+\frac{3}{2}m_{18}^2+\frac{3}{2}m_{52}^2+\frac{3}{2}m_{13}^2+\frac{3}{2}m_{12}^2+\frac{3}{2}m_{14}^2+\frac{3}{2}m_{5}^2-\frac{3}{2}m_{37}m_{61}+\frac{3}{2}m_{39}^2
-\frac{5}{2}m_{37}^2+\frac{3}{2}m_{22}^2+\frac{3}{2}m_{4}^2+\frac{3}{2}m_{19}^2+\frac{3}{2}m_{8}^2\\&
+\frac{3}{2}m_{16}^2+\frac{3}{2}m_{24}^2+\frac{3}{2}m_{40}^2-\frac{5}{2}m_{50}^2-\frac{5}{2}m_{49}^2+\frac{3}{2}m_{10}^2-\frac{3}{2}m_{57}m_{61}-\frac{3}{2}m_{58}m_{71}
-\frac{5}{2}m_{57}^2-\frac{5}{2}m_{58}^2-\frac{3}{2}m_{59}m_{77}-\frac{5}{2}m_{59}^2\\&
+\frac{3}{5}m_{61}^2+\frac{3}{5}m_{71}^2
+\frac{3}{5}m_{77}^2+\frac{3}{5}m_{80}^2+\frac{3}{2}m_{3}^2-\frac{3}{2}m_{49}m_{61}-\frac{3}{2}m_{50}m_{71}.
\end{array}$$}

In Step 6, using the Matlab program in Appendix B
with $\widehat{F}_{3,4}$ and $\widehat{\C}_{3,4,1}$ as input, we find an SOS representation \eqref{eq-tt3} for $\widehat{F}_{3,4}$.
Thus, $D(3,4)$ is proved.
%
The Maple  program to prove $D(3,4)$ can be found in
\begin{verbatim}
http://www.mmrc.iss.ac.cn/~xgao/software/maple-epid34.zip
\end{verbatim}

%

\section{Conclusion and discussion}
\label{sec-conc}

We observe that the proofs of Costa's EPI
and McKean's conjecture $D(m,n): (-1)^{m+1}\frac{\d}{\d t}H(X_t)\ge0$
can be reduced to the problem of writing a differentially homogenous
polynomial in $p_t$ and its derivatives as an SOS under certain
constraints, which can be solved with SDP effectively.
Based on this observation and previous work, we propose
a procedure to prove Costa's EPI and $D(m,n)$ automatically using a Matlab program for SDP.
%
%
The procedure is not theoretically grantee to succeed,
but when it succeeds, an explicit and strict proof is given.

Using this procedure, we give  new proofs for
all known results about Costa's EPI: $\frac{\d^2}{\d^2 t} N(X_t)\le0$
and $D(2,n)$ for any $n$,
$D(3,1)$, and $D(4,1)$ (see Appendix D) .
We also prove some new results: $D(3,n)$ for $n=2,3,4$.

The procedure   has limitations for more complicated problems,
as will be shown below.

For $D(5,1)$,
$\C_{5,1}$ contains 52 constraints,
$M_{5,1}=\{m_i\,:\, i=1,\ldots,7\}$,
there exist 3 intrinsic constraints,
$\widehat{\C}_{5,1,1}$ contains 16 constraints,
and $\widehat{F}_{5,1}$ has 14 terms.
Running the Matlab program in Appendix B with input $\widehat{F}_{5,1}$ and $\widehat{\C}_{5,1,1}$,
the program terminates but does not give an answer.

For $D(4,2)$, $\C_{4,2}$ contains  771 constraints,
$M_{4,2}=\{m_i\,:\, i=1,\ldots,80\}$,
there exist 182 intrinsic constraints,
$\widehat{\C}_{4,2,1}$ contains 417 constraints,
and $\widehat{F}_{4,2}$ has 70 terms.
Running the Matlab program in Appendix B with input $\widehat{F}_{4,2}$ and $\widehat{\C}_{4,2,1}$,
the program terminates but does not give an answer.

For $D(3,5)$, $\C_{3,5}$ contains  7168 constraints,
$M_{3,5}=\{m_i\,:\, i=1,\ldots,145\}$,
there exist 4125 intrinsic constraints.
Using $\R$-Gaussian elimination to  $\C_{3,5,1}$ in Maple on a PC with a 3.40GHz CPU and 16G memory, the program fails due to  ``memory overflow".
%

There are two types of reasons for the failure. For $D(3,5)$, the SDP problem is too large for the
program and computer.
For $D(4,2)$ and $D(5,1)$, the SDP program terminates,
but fails to find an  SOS representation.
Since there exists no proof that the SDP algorithm  is complete
for the problem under consideration, we do not know whether
$\widehat{F}_{4,2}$ and $\widehat{F}_{5,1}$
are positive semidefinite under the constraints
$\widehat{\C}_{4,2,1}$ and $\widehat{\C}_{5,1,1}$, respectively.
Based on our extensive experience with the Matlab program for SDP,
we strongly believe that there exist no SOS representations \eqref{eq-tt3}
for $\widehat{F}_{4,2}$/$\widehat{\C}_{4,2,1}$ and $\widehat{F}_{5,1}$/$\widehat{\C}_{5,1,1}$.
Furthermore, $\int\frac{\widehat{F}_{4,2}}{p_t^7} dx_t\le0$
and $\int\frac{\widehat{F}_{5,2}}{p_t^9} dx_t\ge0$
could be true  even there does not exist SOS representations for $\widehat{F}_{4,2}$ and $\widehat{F}_{5,1}$.
%
%
To enhance the power of Procedure \ref{proc-H},
a major problem for further study is to find more constraints besides those given in Lemmas \ref{lm-cons1}, \ref{lm-cons2}, \ref{lm-cons3}.

\section*{Acknowledgments}

This work is partially
supported by  NSFC 11688101, Beijing Natural Science Foundation (No. Z190004)
and China Postdoctoral Science Foundation (No. 2019TQ0343, 2019M660830).
We thank Dr. Bo Bai for informing us the current progress about Costa's EPI in \cite{Cheng2015}.

\section*{Appendix A. Proofs of Lemma \ref{lemma4} and Lemma \ref{lemma2}}

{\bf  Proof of Lemma \ref{lemma4}.} We need to prove

\begin{equation}\label{A1}
\begin{array}{ll}
\displaystyle{\int_{-\infty}^{\infty}\ldots\int_{-\infty}^{\infty}p_t\left[\prod\limits_{i=1}^{r}
\frac{[p_t^{(m_i)}]^{k_i}}{p_t^{k_i}}\right]\Bigg|_{x_{a,t}=-\infty}^{\infty}\d x_t^{(a)}}=0,\\[0.5cm]
\end{array}
\end{equation}
where $a,r,m_i,k_i \in \N_{>0}$, $p_t^{(m_i)}=\frac{\partial^{m_i}p_t}{\partial x_{1,t}^{j_1}\ldots\partial x_{n,t}^{j_n}}$, $j_1+\ldots+j_n=m_i$.
It suffices to prove that
\begin{equation}\label{A2}
\begin{array}{ll}
\displaystyle{\int_{-\infty}^{\infty}\ldots\int_{-\infty}^{\infty}\lim\limits_{x_{a,t}\rightarrow\infty}p_t\left[\prod\limits_{i=1}^{r}
\frac{[p_t^{(m_i)}]^{k_i}}{p_t^{k_i}}\right]\d x_t^{(a)}}=0.
\end{array}
\end{equation}

Assume that the limit in the integrand exists, according to Fatou's lemma, the absolute value of expression \eqref{A2} satisfies the sequence of relations
\begin{equation}\label{A4}
\begin{array}{ll}
\left|\displaystyle{\int_{-\infty}^{\infty}\ldots\int_{-\infty}^{\infty}\lim\limits_{x_{a,t}\rightarrow\infty}p_t\left[\prod\limits_{i=1}^{r}
\frac{[p_t^{(m_i)}]^{k_i}}{p_t^{k_i}}\right]\d x_t^{(a)}}\right|\\[0.4cm]
\leq\displaystyle{\int_{-\infty}^{\infty}\ldots\int_{-\infty}^{\infty}\lim\limits_{x_{a,t}\rightarrow\infty}\left|p_t\left[\prod\limits_{i=1}^{r}
\frac{[p_t^{(m_i)}]^{k_i}}{p_t^{k_i}}\right]\right|\d x_t^{(a)}}\\[0.4cm]
=\displaystyle{\int_{-\infty}^{\infty}\ldots\int_{-\infty}^{\infty}\liminf\limits_{x_{a,t}\rightarrow\infty}\left|p_t\left[\prod\limits_{i=1}^{r}
\frac{[p_t^{(m_i)}]^{k_i}}{p_t^{k_i}}\right]\right|\d x_t^{(a)}}\\[0.4cm]
\leq\displaystyle{\liminf\limits_{x_{a,t}\rightarrow\infty}\int_{-\infty}^{\infty}\ldots\int_{-\infty}^{\infty}\left|p_t\left[\prod\limits_{i=1}^{r}
\left(\frac{p_t^{(m_i)}}{p_t}\right)^{k_i}\right]\right|\d x_t^{(a)}}
\end{array}
\end{equation}

By simple computation, we find that
\begin{equation}\label{A5}
\begin{array}{ll}
\displaystyle{\int_{-\infty}^{\infty}\int_{-\infty}^{\infty}\ldots\int_{-\infty}^{\infty}\left|p_t\left[\prod\limits_{i=1}^{r}
\left(\frac{p_t^{(m_i)}}{p_t}\right)^{k_i}\right]\right|\d x_t}\\[0.3cm]
=\displaystyle{\int_{\mathbf{R}^n}p_t\left[\prod\limits_{i=1}^{r}
\left|\frac{p_t^{(m_i)}}{p_t}\right|^{k_i}\right]\d x_t}\\[0.3cm]
=\mathbb{E}\left\{\prod\limits_{i=1}^{r}\left|\frac{p_t^{(m_i)}}{p_t}\right|^{k_i}\right\}.
\end{array}
\end{equation}
We have
\begin{equation}\label{A6}
\begin{array}{ll}
&\left|\frac{p_t^{(m_i)}}{p_t}\right|\\[0.4cm]
&\leq\displaystyle{\frac{1}{(2\pi t)^{n/2}}\int_{\mathbf{R}^n}\frac{p(x)}{p_t}\exp\left(-\frac{\|x_t-x\|^2}{2t}\right)
\left|\sum\limits_{j_1+\ldots+j_n=0}^{m_i}\left(\prod\limits_{q=1}^{n}a_q(x_{q,t}-x_{q})^{j_q}\right) \right|dx}\\[0.4cm]
&\leq\mathbb{E}\left[\left.\sum\limits_{j_1+\ldots+j_n=0}^{m_i}
\left(\prod\limits_{q=1}^{n}|a_q|\left|X_{q,t}-X_{q}\right|^{j_q}\right)\right|{X_t=x_t}\right]\\[0.4cm]
&=\sum\limits_{j_1+\ldots+j_n=0}^{m_i}\mathbb{E}
\left[\left.\left(\prod\limits_{q=1}^{n}|a_q|\left|Z_{q,t}\right|^{j_q}\right)\right|{X_t=x_t}\right],\\[0.4cm]
\end{array}
\end{equation}
where $a_q$'s in above formulaes are some constants that also depend on $t$.

Because  $Z_t$ is a Gaussian random vector with i.i.d. components having meaning zero and variance $t$, that is, $Z_t\thicksim N_n(0,tI)$, using Jensen's inequality, Cauchy-Schwartz inequality and inductive method can give the following result
\begin{equation}\label{A7a}
\begin{array}{ll}
\sum\limits_{j_1+\ldots+j_n=0}^{m_i}\mathbb{E}
\left[\left.\left(\prod\limits_{q=1}^{n}|a_q|\left|Z_{q,t}\right|^{j_q}\right)\right|{X_t=x_t}\right]
<\infty,
\end{array}\end{equation}
and then
\begin{equation}\label{A7b}
\begin{array}{ll}
\mathbb{E}\left\{\prod\limits_{i=1}^{r}\left|\frac{p_t^{(m_i)}}{p_t}\right|^{k_i}\right\}
<\infty.
\end{array}\end{equation}

Therefore, the $n$-dimensional integral in \eqref{A5} is finite. Then we can get the result
$$
\displaystyle{\liminf\limits_{x_{a,t}\rightarrow\infty}\int_{-\infty}^{\infty}\ldots\int_{-\infty}^{\infty}\left|p_t\left[\prod\limits_{i=1}^{r}
\left(\frac{p_t^{(m_i)}}{p_t}\right)^{k_i}\right]\right|\d x_t^{(a)}}=0,
$$
and the lemma is proved.

{\bf Proof of Lemma \ref{lemma2}}.
If $m=2p$, one can show that $\nabla^{(m)}p_t$ has the following form
\begin{equation*}
\begin{array}{ll}
\nabla^{(m)}p_t=&\sum\limits_{q_1,q_2,\ldots,q_{p}=0}^{1}a_{q_1,q_2,\ldots,q_{p}}\sum\limits_{i_1,i_2,\ldots,i_{p}=1}^{n}\displaystyle{\int_{\mathbf{R}^n}\prod\limits_{j=1}^{p}
[(x_{i_j,t}-x_{i_j})^2]^{q_j}}\\
&\cdot\frac{p(x)}{(2\pi t)^{n/2}}{\rm exp}(-\frac{\parallel x_t-x\parallel^2}{2t})\d x_t,
\end{array}
\end{equation*}
and if $m=2p+1$,
\begin{equation*}
\begin{array}{ll}
\parallel\nabla^{(m)}p_t\parallel=&\left\{\sum\limits_{i_{p+1}=1}^{n}\left[\sum\limits_{q_1,q_2,\ldots,q_{p}=0}^{1}a_{q_1,q_2,\ldots,q_{p}}
\sum\limits_{i_1,i_2,\ldots,i_{p}=1}^{n}\displaystyle{\int_{\mathbf{R}^n}\prod\limits_{j=1}^{p}
[(x_{i_j,t}-x_{i_j})^2]^{q_j}(x_{i_{p+1},t}-x_{i_{p+1}})}\right.\right.\\
&\left.\left.\cdot\frac{p(x)}{(2\pi t)^{n/2}}{\rm exp}(-\frac{\parallel x_t-x\parallel^2}{2t})\d x_t\right]^2\right\}^{1/2},
\end{array}
\end{equation*}
where $a_{q_1,q_2,\ldots,q_{p+1}}$'s and $a_{q_1,q_2,\ldots,q_{p}}$'s in above formulaes are some constants that also depend on $t$.


Because  $Z_t$ is a Gaussian random vector with i.i.d. components having meaning zero and variance $t$, using Jensen's inequality, Cauchy-Schwartz inequality and inductive method can give the following results.

If $m=2p$, we have
\begin{equation}\label{B1}
\begin{array}{ll}
\frac{\nabla^{(m)}p_t}{p_t}&=\sum\limits_{q_1,q_2,\ldots,q_{p}=0}^{1}a_{q_1,q_2,\ldots,q_{p}}\sum\limits_{i_1,i_2,\ldots,i_{p}=1}^{n}\displaystyle{\int_{\mathbf{R}^n}\prod\limits_{j=1}^{p}
[(x_{i_j,t}-x_{i_j})^2]^{q_j}}\\
&\cdot\frac{p(x)}{(2\pi t)^{n/2}p_t}{\rm exp}(-\frac{\parallel x_t-x\parallel^2}{2t})\d x_t\\
&=\sum\limits_{q_1,q_2,\ldots,q_{p}=0}^{1}a_{q_1,q_2,\ldots,q_{p}}\sum\limits_{i_1,i_2,\ldots,i_{p}=1}^{n}E\left(\left.\prod\limits_{j=1}^{p}
[(X_{i_j,t}-X_{i_j})^2]^{q_j}\right|X_t=x_t\right)\\
&=\sum\limits_{q_1,q_2,\ldots,q_{p}=0}^{1}a_{q_1,q_2,\ldots,q_{p}}\sum\limits_{i_1,i_2,\ldots,i_{p}=1}^{n}E\left(\prod\limits_{j=1}^{p}
(Z_{i_j,t})^{2q_j}\right)\\
&<\infty,
\end{array}\end{equation}

if $m=2p+1$, we have
\begin{equation}\label{B2}
\begin{array}{ll}
\frac{\nabla^{(m)}p_t}{p_t}&=\left\{\sum\limits_{i_{p+1}=1}^{n}\left[\sum\limits_{q_1,q_2,\ldots,q_{p}=0}^{1}a_{q_1,q_2,\ldots,q_{p}}\sum\limits_{i_1,i_2,\ldots,i_{p}=1}^{n}\displaystyle{\int_{\mathbf{R}^n}\prod\limits_{j=1}^{p}
[(x_{i_j,t}-x_{i_j})^2]^{q_j}(x_{i_{p+1},t}-x_{i_{p+1}})}\right.\right.\\
&\left.\left.\cdot\frac{p(x)}{(2\pi t)^{n/2}p_t}{\rm exp}(-\frac{\parallel x_t-x\parallel^2}{2t})\d x_t\right]^2\right\}^{1/2}\\
&=\left\{\sum\limits_{i_{p+1}=1}^{n}\left[\sum\limits_{q_1,q_2,\ldots,q_{p}=0}^{1}a_{q_1,q_2,\ldots,q_{p}}\sum\limits_{i_1,i_2,\ldots,i_{p}=1}^{n}
E\left(\prod\limits_{j=1}^{p}[(X_{i_j,t}-X_{i_j})^2]^{q_j}\right.\right.\right.\\
&\left.\left.\left.\left.\cdot(X_{i_{p+1},t}-X_{i_{p+1}})\right|X_t=x_t\right)\right]^2\right\}^{1/2}\\
&=\left\{\sum\limits_{i_{p+1}=1}^{n}\left[\sum\limits_{q_1,q_2,\ldots,q_{p}=0}^{1}a_{q_1,q_2,\ldots,q_{p}}\sum\limits_{i_1,i_2,\ldots,i_{p}=1}^{n}
E\left(\prod\limits_{j=1}^{p}(Z_{i_j,t})^{2q_j}(Z_{i_{p+1},t})\right)\right]^2\right\}^{1/2}\\
&<\infty.
\end{array}
\end{equation}

Thus, we have
\begin{equation}\label{B3}
\begin{array}{ll}
\left|\displaystyle{\int_{0}^{\infty}\int_{S_r}\prod\limits_{i=1}^{s}\frac{[\nabla^{(m_i)}p_t]^{k_i}}{p_t^{k_i}}\nabla^{(m_{s+1})} p_t\cdot {\rm d}\mathbf{s_r}{\rm d}r}\right|\\
\leq \displaystyle{\int_{\mathbf{R}^n}\prod\limits_{i=1}^{s}\left|\frac{\nabla^{(m_i)}p_t}{p_t}\right|^{k_i}\parallel\nabla^{(m_{s+1})} p_t\parallel\parallel{\rm d}\mathbf{s_r}\parallel{\rm d}r}\\
=\displaystyle{\int_{\mathbf{R}^n}p_t\prod\limits_{i=1}^{s}\left|\frac{\nabla^{(m_i)}p_t}{p_t}\right|^{k_i}\frac{\parallel\nabla^{(m_{s+1})} p_t\parallel}{p_t}{\rm d}x_t}\\
=E\left(\prod\limits_{i=1}^{s}\left|\frac{\nabla^{(m_i)}p_t}{p_t}\right|^{k_i}\frac{\parallel\nabla^{(m_{s+1})} p_t\parallel}{p_t}\right)\\
<\infty.
\end{array}
\end{equation}
The finiteness of \eqref{B3} is obtained by induction from the finiteness of \eqref{B1}\eqref{B2} and the Cauchy-Schwartz inequality.

\section*{Appendix B. Sum of square of quadratic forms based on SDP}
\label{app-b}

Let $f(x_1,\ldots,x_n)\ \textrm{and}\ g_i(x_1,\ldots,x_n)\in \R[x_1,\ldots,x_n],\ i=1,\ldots, r$, be quadratic homogenous polynomials in variables $x_1,\ldots,x_n$, or simply quadratic forms.
In this appendix, we will show how to compute  $p_i\in\R$ such that
\begin{equation}
\label{eq-sos1}
f(x_1,\ldots,x_n)-\sum_{i=1}^r p_i  g_i(x_1,\ldots,x_n) = S
\end{equation}
where $S\ge0$ is a sum of squares (SOS) of linear forms in $x_1,\ldots,x_n$.

A polynomial $f$ in $\R[x_1,\ldots,x_n]$ is called {\em positive semidefinite}
and is denoted as $f\succeq0$,
if $\forall \tilde{x}_i\in \R, f(\tilde{x}_1,\ldots,\tilde{x}_n)\ge0$.
The following  known result shows that  a quadratic form $f\succeq0$ if and only if $f$ is an SOS.

{\noindent\bf Lemma B}.
Let $f\in \Q[x_1,\ldots,x_n]$  be a quadratic form. Then $f\succeq0$ if and only if
\begin{equation}\label{2.1t}
f=\sum_{i=1}^n c_i (\sum_{j=i}^n e_{i,j} x_j)^2,
\end{equation}
where $c_i,e_{i,j}\in\Q$, $c_i\ge0$, and $e_{i,i}\ne0$ if $c_i\ne0$,  for $i=1,\ldots,n$ and $j=i,\ldots,n$.
\begin{proof}
We give a proof, since we do not find it in the literature.
Suppose  $f=\sum_{i=1}^na_{i,i} x_i^2 + 2\sum_{1\le i<j\le n} a_{i,j}x_ix_j\succeq0$.
Without loss of generality, we assume $f\ne0$.
Since $f\succeq0$, there exists an $i$, say $i=1$, such that $a_{1,1}\ne0$.
Then, $f= \frac{1}{a_{1,1}}(\sum_{i=1}^n a_{1,i}x_i)^2 + f_2$, where $f_2\in\Q[x_2,\ldots,x_n]$.
Repeat the above procedure for $f_2$, we obtain a formula like (\ref{2.1t}).
Since $e_{i,i}\ne0$, the linear forms $\sum_{j=i}^n e_{i,j} x_j$ are linear independent to each other.
If some $c_k<0$ then $f$ can take negative values, contradicting to $f\succeq0$. The lemma is proved.
\end{proof}

Based on Lemma B, problem \eqref{eq-sos1} is equivalent to the following problem.
%
\begin{equation}
\label{eq-sdp2}
\exists (p_1,\ldots,p_r)\in\R^n,\,
 f(x_1,\ldots,x_n)-\sum_{i=1}^r p_ig_i(x_1,\ldots,x_n)\succeq0.
\end{equation}%

Problem (\ref{eq-sdp2}) can be solved with semidivine programming (SDP).
For details of SDP, please refer to~\cite{Boyd1,Boyd2}.
%
%
A symmetric matrix $\mathcal{M}\in\R^{n\times n}$ is called {\em positive semidefinite}
and is denoted as $\mathcal{M}\succeq0$, if all of its eigenvalues are nonnegative.
Rewrite
$$
f(x_1,\ldots,x_n)=\textbf{x}C\textbf{x}^T,\ \ g_i(x_1,\ldots,x_n)=\textbf{x}A_i\textbf{x}^T,\ i=1,\ldots,r,
$$
where $\textbf{x}=(x_1,\ldots,x_n)$, $C$ and $A_i$ are $n\times n$ real symmetric matrices.
Then, problem (\ref{eq-sdp2}) is equivalent to the following SDP problem:
\begin{equation}
\label{eq-sdp3}
\begin{array}{ll}
\textrm{min}_{l_i\in\R,i=1,\ldots,r}\ \ 0\\
\textrm{subject\ to\ }C-\sum\limits_{i=1}^{r}p_iA_i \succeq0.
\end{array}\end{equation}
The dual of problem (\ref{eq-sdp3}) is
\begin{equation}\begin{array}{ll}
\label{eq-sdp4}
\textrm{min}_{\in\R^{n\times n}}\ \ &\langle\lambda,C\rangle\\
\textrm{subject\ to\ }&\langle\lambda,A_i\rangle=0,\ i=1,2,\ldots,r,\\
&\lambda\succeq0.
\end{array}\end{equation}
where $\lambda$ is a symmetric matrix and $\langle\cdot\rangle$ is the inner product by treating matrices as vectors.

Problem (\ref{eq-sdp4}) can be solved with the following Matlab program
which computes $\hbox{L}=\lambda$ and $P=(p_1,\ldots,p_r)$
with $C$ and $A_i,i=1\ldots,r$ as the input.
This program uses the CVX package in Matlab~\cite{grant2008} to solve SDPs.
\lstset{language=Matlab}
\begin{lstlisting}
cvx_begin
    variable L(n,n) symmetric
    dual variable P
    minimize(trace(C*X))
    subject to
    [trace(A_1*L),trace(A_2*L),...,trace(A_r*L)]'== zeros(r,1):P;
    Lambda == semidefinite(n);
cvx_end
\end{lstlisting}
After $P$ is obtained, it is easy to find the SOS representation \eqref{eq-sos1}
following the proof of Lemma B.
%
%
%

\section*{Appendix C. Constraints in Lemma \ref{lem-42}}
The 17   constraints in Lemma \ref{lem-42} are given below.
For simplicity, we let $R_i = R^{(2)}_{i,a,b}$.

{\footnotesize\parskip=0pt\parindent=2pt
$R_1=p_t^2(x_t)\frac{\partial^3p_t(x_t)}{\partial^3 x_{b,t}}\frac{\partial p_t(x_t)}{\partial x_{b,t}}+\frac{\partial^2p_t(x_t)}{\partial^2x_{b,t}} [p_t^2(x_t)\frac{\partial^2p_t(x_t)}{\partial^2x_{b,t}}-p_t(x_t) (\frac{\partial p_t(x_t)}{\partial x_{b,t}} )^2 ].$

$R_2=p_t^2(x_t)\frac{\partial^3p_t(x_t)}{\partial^3 x_{b,t}}\frac{\partial p_t(x_t)}{\partial x_{a,t}}+\frac{\partial^2p_t(x_t)}{\partial^2x_{b,t}} [p_t^2(x_t)\frac{\partial^2p_t(x_t)}{\partial x_{a,t}\partial x_{b,t}}-p_t(x_t)\frac{\partial p_t(x_t)}{\partial x_{a,t}}\frac{\partial p_t(x_t)}{\partial x_{b,t}} ].$

$R_3=p_t^2(x_t)\frac{\partial^3p_t(x_t)}{\partial x_{a,t}\partial^2x_{b,t}}\frac{\partial p_t(x_t)}{\partial x_{b,t}}+\frac{\partial^2p_t(x_t)}{\partial x_{a,t}\partial x_{b,t}} [p_t^2(x_t)\frac{\partial^2p_t(x_t)}{\partial^2x_{b,t}}-p_t(x_t) (\frac{\partial p_t(x_t)}{\partial x_{b,t}} )^2 ].$

$R_4=p_t^2(x_t)\frac{\partial^3p_t(x_t)}{\partial x_{a,t}\partial^2x_{b,t}}\frac{\partial p_t(x_t)}{\partial x_{a,t}}+\frac{\partial^2p_t(x_t)}{\partial x_{a,t}\partial x_{b,t}} [\frac{1}{p_t(x_t)}\frac{\partial^2p_t(x_t)}{\partial x_{a,t}\partial x_{b,t}}-p_t(x_t)\frac{\partial p_t(x_t)}{\partial x_{a,t}}\frac{\partial p_t(x_t)}{\partial x_{b,t}} ].$

$R_5=p_t^2(x_t)\frac{\partial^3p_t(x_t)}{\partial^2 x_{a,t}\partial x_{b,t}}\frac{\partial p_t(x_t)}{\partial x_{b,t}}+\frac{\partial^2p_t(x_t)}{\partial x_{a,t}\partial x_{b,t}} [p_t^2(x_t)\frac{\partial^2p_t(x_t)}{\partial x_{a,t}\partial x_{b,t}}-p_t(x_t)\frac{\partial p_t(x_t)}{\partial x_{a,t}}\frac{\partial p_t(x_t)}{\partial x_{b,t}} ].$

$R_6=p_t^2(x_t)\frac{\partial^3p_t(x_t)}{\partial^2 x_{a,t}\partial x_{b,t}}\frac{\partial p_t(x_t)}{\partial x_{a,t}}+\frac{\partial^2p_t(x_t)}{\partial x_{a,t}\partial x_{b,t}} [p_t^2(x_t)\frac{\partial^2p_t(x_t)}{\partial^2 x_{a,t}}-p_t(x_t) (\frac{\partial p_t(x_t)}{\partial x_{a,t}} )^2 ].$

$R_7=p_t^2(x_t)\frac{\partial^3p_t(x_t)}{\partial^3 x_{a,t}}\frac{\partial p_t(x_t)}{\partial x_{b,t}}+\frac{\partial^2p_t(x_t)}{\partial^2 x_{a,t}} [p_t^2(x_t)\frac{\partial^2p_t(x_t)}{\partial x_{a,t}\partial x_{b,t}}-p_t(x_t)\frac{\partial p_t(x_t)}{\partial x_{a,t}}\frac{\partial p_t(x_t)}{\partial x_{b,t}} ].$

$R_8=p_t^2(x_t)\frac{\partial^3p_t(x_t)}{\partial^3 x_{a,t}}\frac{\partial p_t(x_t)}{\partial x_{a,t}}+\frac{\partial^2p_t(x_t)}{\partial^2 x_{a,t}} [p_t^2(x_t)\frac{\partial^2p_t(x_t)}{\partial^2 x_{a,t}}-p_t(x_t) (\frac{\partial p_t(x_t)}{\partial x_{a,t}} )^2 ].$

$R_9=p_t(x_t)\frac{\partial^2p_t(x_t)}{\partial^2x_{b,t}} (\frac{\partial p_t(x_t)}{\partial x_{a,t}} )^2+\frac{\partial p_t(x_t)}{\partial x_{b,t}} [2p_t(x_t)\frac{\partial p_t(x_t)}{\partial x_{a,t}}\frac{\partial p_t^2(x_t)}{\partial x_{a,t}\partial x_{b,t}}-2 (\frac{\partial p_t(x_t)}{\partial x_{a,t}} )^2\frac{\partial p_t(x_t)}{\partial x_{b,t}} ].$

$R_{10}=p_t(x_t)\frac{\partial^2p_t(x_t)}{\partial^2x_{b,t}}\frac{\partial p_t(x_t)}{\partial x_{a,t}}\frac{\partial p_t(x_t)}{\partial x_{b,t}}+\frac{\partial p_t(x_t)}{\partial x_{b,t}} [p_t(x_t)\frac{\partial p_t(x_t)}{\partial x_{b,t}}\frac{\partial p_t^2(x_t)}{\partial x_{a,t}\partial x_{b,t}}-2\frac{\partial p_t(x_t)}{\partial x_{a,t}} (\frac{\partial p_t(x_t)}{\partial x_{b,t}} )^2$
%
$+p_t(x_t)\frac{\partial p_t(x_t)}{\partial x_{a,t}}\frac{\partial p_t^2(x_t)}{\partial^2x_{b,t}} ].$

$R_{11}=p_t(x_t)\frac{\partial^2p_t(x_t)}{\partial^2x_{b,t}} (\frac{\partial p_t(x_t)}{\partial x_{b,t}} )^2+\frac{\partial p_t(x_t)}{\partial x_{b,t}} [2p_t(x_t)\frac{\partial p_t(x_t)}{\partial x_{b,t}}\frac{\partial p_t^2(x_t)}{\partial^2x_{b,t}}-2 (\frac{\partial p_t(x_t)}{\partial x_{b,t}} )^3 ].$

$R_{12}=p_t(x_t)\frac{\partial p_t(x_t)}{\partial x_{a,t}\partial x_{b,t}} (\frac{\partial p_t(x_t)}{\partial x_{a,t}} )^2+\frac{\partial p_t(x_t)}{\partial x_{b,t}} [2p_t(x_t)\frac{\partial p_t(x_t)}{\partial x_{a,t}}\frac{\partial p_t^2(x_t)}{\partial^2 x_{a,t}}-2 (\frac{\partial p_t(x_t)}{\partial x_{a,t}} )^3 ].$

$R_{13}=p_t(x_t)\frac{\partial^2p_t(x_t)}{\partial x_{a,t}\partial x_{b,t}}\frac{\partial p_t(x_t)}{\partial x_{a,t}}\frac{\partial p_t(x_t)}{\partial x_{b,t}}+\frac{\partial p_t(x_t)}{\partial x_{b,t}} [p_t(x_t)\frac{\partial p_t^2(x_t)}{\partial^2 x_{a,t}}\frac{\partial p_t(x_t)}{\partial x_{b,t}}-2 (\frac{\partial p_t(x_t)}{\partial x_{a,t}} )^2\frac{\partial p_t(x_t)}{\partial x_{b,t}}$
%
$ +p_t(x_t)\frac{\partial p_t(x_t)}{\partial x_{a,t}}\frac{\partial^2 p_t(x_t)}{\partial x_{a,t}\partial x_{b,t}} ].$

$R_{14}=p_t(x_t)\frac{\partial^2 p_t(x_t)}{\partial x_{a,t}\partial x_{b,t}} (\frac{\partial p_t(x_t)}{\partial x_{b,t}} )^2+\frac{\partial p_t(x_t)}{\partial x_{b,t}} [2p_t(x_t)\frac{\partial p_t^2(x_t)}{\partial x_{a,t}\partial x_{b,t}}\frac{\partial p_t(x_t)}{\partial x_{b,t}}-2\frac{\partial p_t(x_t)}{\partial x_{a,t}} (\frac{\partial p_t(x_t)}{\partial x_{b,t}} )^2 ].$

$R_{15}=p_t(x_t)\frac{\partial^2 p_t(x_t)}{\partial^2 x_{a,t}} (\frac{\partial p_t(x_t)}{\partial x_{a,t}} )^2+\frac{\partial p_t(x_t)}{\partial x_{a,t}} [2p_t(x_t)\frac{\partial p_t(x_t)}{\partial x_{a,t}}\frac{\partial^2 p_t(x_t)}{\partial^2 x_{a,t}}-2 (\frac{\partial p_t(x_t)}{\partial x_{a,t}} )^3 ].$

$R_{16}=p_t(x_t)\frac{\partial^2p_t(x_t)}{\partial^2 x_{a,t}}\frac{\partial p_t(x_t)}{\partial x_{a,t}}\frac{\partial p_t(x_t)}{\partial x_{b,t}}+\frac{\partial p_t(x_t)}{\partial x_{a,t}} [p_t(x_t)\frac{\partial p_t^2(x_t)}{\partial^2 x_{a,t}}\frac{\partial p_t(x_t)}{\partial x_{b,t}}-2 (\frac{\partial p_t(x_t)}{\partial x_{a,t}} )^2\frac{\partial p_t(x_t)}{\partial x_{b,t}}$
%
$+p_t(x_t)\frac{\partial p_t(x_t)}{\partial x_{a,t}}\frac{\partial^2 p_t(x_t)}{\partial x_{a,t}\partial x_{b,t}} ].$

$R_{17}=p_t(x_t)\frac{\partial^2 p_t(x_t)}{\partial^2 x_{a,t}} (\frac{\partial p_t(x_t)}{\partial x_{b,t}} )^2+\frac{\partial p_t(x_t)}{\partial x_{a,t}} [2p_t(x_t)\frac{\partial^2 p_t(x_t)}{\partial x_{a,t}\partial x_{b,t}}\frac{\partial p_t(x_t)}{\partial x_{b,t}}-2\frac{\partial p_t(x_t)}{\partial x_{a,t}} (\frac{\partial p_t(x_t)}{\partial x_{b,t}} )^2 ].$
}

\section*{Appendix D. Proof of $D(4,1)$}
%
Let
$x_t = x_{1,t}, f:=f_0:=p_t,\ \ f_n:=\frac{\partial^n p_t}{\partial^n x_{1,t}},\, n\in\N_{>0}$.
We have
$$-\dfrac{\d ^4}{\d  t^4}H(X_t)
\overset{\eqref{H2}}{=}-\dfrac{1}{2}\dfrac{\d ^3}{\d  t^3}\displaystyle{\int\dfrac{f_1^2}{f}dx_t}
=-\dfrac{1}{2}\displaystyle{\int\dfrac{\d ^3}{\d  t^3}\left(\dfrac{f_1^2}{f}\right)dx_t}
\overset{\eqref{H1}}{=}\displaystyle{\int\dfrac{F_{4,1}}{f^7} dx_t}$$
where
$F_{4,1}=-\frac{3}{8}f^6f_{3}f_{5}+\frac{3}{8}f^5f_{3}^2f_{2}-\frac{3}{4}f^4f_{1}f_{3}f_{2}^2
+\frac{3}{8}f^5f_{1}f_{5}f_{2}
+\frac{3}{8}f^5f_{1}f_{3}f_{4}
-\frac{1}{8}f^6f_{1}f_{7}
+\frac{3}{8}f^3f_{1}^2f_{2}^3
-\frac{3}{8}f^4f_{1}^2f_{4}f_{2}
+\frac{1}{16}f^5f_{1}^2f_{6}.$
We can use  Procedure \ref{proc-H} to prove $D(4,1):-\frac{\d^4H(X_t)}{\d^4 t}\ge0$ with $F_{4,1}$ as the input.

In {Step} 1, $\C_{4,1}=\{R_{i}, i=1,\ldots,14\}$ using Lemma
\ref{lm-cons1}, where
{\small
\begin{equation*}
\begin{array}{ll}
R_{1} = 7ff_{1}^6f_{2}-6f_{1}^8,
& R_{2} = 3f^4f_{1}f_{2}^2f_{3}+f^4f_{2}^4-3f^3f_{1}^2f_{2}^3,\\
 R_{3} = f^6f_{1}f_{7}+f^6f_{2}f_{6}-f^5f_{1}^2f_{6},
& R_{4} = 2f^5f_{1}f_{3}f_{4}+f^5f_{2}f_{3}^2-2f^4f_{1}^2f_{3}^2,\\
R_{5} = f^5f_{1}^2f_{6}+2f^5f_{1}f_{2}f_{5}-2f^4f_{1}^3f_{5},
&R_{6} = 2f^3f_{1}^3f_{2}f_{3}+3f^3f_{1}^2f_{2}^3-4f^2f_{1}^4f_{2}^2,\\
R_{7} = f^4f_{1}^3f_{5}+3f^4f_{1}^2f_{2}f_{4}-3f^3f_{1}^4f_{4},
&R_{8} = f^3f_{1}^4f_{4}+4f^3f_{1}^3f_{2}f_{3}-4f^2f_{1}^5f_{3},\\
R_{9} = f^2f_{1}^5f_{3}+5f^2f_{1}^4f_{2}^2-5ff_{1}^6f_{2},
&R_{10} = f^5f_{2}^2f_{4}+2f^5f_{2}f_{3}^2-2f^4f_{1}f_{2}^2f_{3},\\
R_{11} = f^6f_{3}f_{5}+f^6f_{4}^2-f^5f_{1}f_{3}f_{4},
&R_{12} = f^5f_{1}f_{2}f_{5}+f^5f_{1}f_{3}f_{4}+f^5f_{2}^2f_{4}-2f^4f_{1}^2f_{2}f_{4},\\
R_{13} = f^6f_{2}f_{6}+f^6f_{3}f_{5}-f^5f_{1}f_{2}f_{5},
&R_{14} = f^4f_{1}^2f_{2}f_{4}+f^4f_{1}^2f_{3}^2+2f^4f_{1}f_{2}^2f_{3}-3f^3f_{1}^3f_{2}f_{3}.
\end{array}\end{equation*}

In {Step} 2,  $\M_{4,1}=\{m_{1} = f^3f_{4}, m_{2} = f^2f_{3}f_{1}, m_{3} = f^2f_{2}^2, m_{4} = ff_{2}f_{1}^2, m_{5} = f_{1}^4\}$.

In {Step} 3, we have $\C_{4,1,1}=\{\widehat{R}_i,i=1,\ldots,7\}$
and $\C_{4,1,2}=\{\widetilde{R}_i,i=1,\ldots,7\}$
\begin{equation*}
\begin{array}{ll}
\widehat{R}_1=7m_{4}m_{5}-6m_{5}^2,\ \
\widehat{R}_2= 3m_{2}m_{3}+m_{3}^2-3m_{3}m_{4},\\
\widehat{R}_3= 2m_{1}m_{2}-\frac{1}{2}m_{1}m_{3}-2m_{2}^2-\frac{1}{3}m_{3}^2+m_{3}m_{4},\ \
\widehat{R}_4= 2m_{2}m_{4}+3m_{3}m_{4}-4m_{4}^2,\\[0.15cm]
\widehat{R}_5= m_{1}m_{5}-6m_{3}m_{4}+28m_{4}^2-\frac{120}{7}m_{5}^2,\ \
\widehat{R}_6= m_{2}m_{5}+5m_{4}^2-\frac{30}{7}m_{5}^2,\\[0.15cm]
\widehat{R}_7= m_{1}m_{4}+m_{2}^2-\frac{2}{3}m_{3}^2+\frac{13}{2}m_{3}m_{4}-6m_{4}^2.
\\ \\
\widetilde{R}_{1} = f^6f_{1}f_{7}+m_{1}^2-4m_{1}m_{2}-3m_{1}m_{3}-12m_{2}^2+8m_{3}^2-114m_{3}m_{4}+240m_{4}^2-\frac{720}{7}m_{5}^2,\\[0.15cm]
 \widetilde{R}_{2} = f^5f_{1}^2f_{6}-\frac{5}{2}m_{1}m_{3}-12m_{2}^2+\frac{19}{3}m_{3}^2-100m_{3}m_{4}+228m_{4}^2-\frac{720}{7}m_{5}^2,\\[0.15cm]
 \widetilde{R}_{3} = f^4f_{1}^3f_{5}-3m_{2}^2+2m_{3}^2-\frac{75}{2}m_{3}m_{4}+102m_{4}^2-\frac{360}{7}m_{5}^2,\\[0.15cm]
 \widetilde{R}_{4} = 2f^5f_{3}^2f_{2}+m_{1}m_{3}+\frac{2}{3}m_{3}^2-2m_{3}m_{4},\\
 \widetilde{R}_{5} = f^6f_{3}f_{5}+m_{1}^2-\frac{1}{4}m_{1}m_{3}-m_{2}^2-\frac{1}{6}m_{3}^2+\frac{1}{2}m_{3}m_{4},\\[0.15cm]
 \widetilde{R}_{6} = f^6f_{2}f_{6}-m_{1}^2+\frac{3}{2}m_{1}m_{3}+4m_{2}^2-m_{3}^2+12m_{3}m_{4}-12m_{4}^2,\\[0.15cm]
 \widetilde{R}_{7} = f^5f_{1}f_{5}f_{2}+\frac{5}{4}m_{1}m_{3}+3m_{2}^2-\frac{7}{6}m_{3}^2+\frac{25}{2}m_{3}m_{4}-12m_{4}^2.
\end{array}
\end{equation*}


In {Step} 4, there exists one intrinsic constraint: $\widehat{R}_{8}=m_3m_5-m_{4}^2$
and $N_2=8$.

In {Step} 5, we eliminate the non-quadratic monomials in $F_{4,1}$ with $\C_{4,1,2}$ and then simplify
the expression with $\C_{4,1,1}$:
\small $$\begin{array}{ll}
\widehat{F}_{4,1}\!\!\!
&= F_{4,1}-(-\frac{1}{4}\widehat{R}_2-\frac{1}{16}\widehat{R}_3-\frac{3}{8}\widehat{R}_7
-\frac{1}{8}\widetilde{R}_1+\frac{1}{16}\widetilde{R}_2+\frac{3}{16}\widetilde{R}_4
-\frac{3}{8}\widetilde{R}_5+\frac{3}{8}\widetilde{R}_7)\\[0.15cm]
&
=\frac{1}{2}m_{1}^2-m_{1}m_{3}-2m_{2}^2+\frac{5}{6}m_{3}^2-10m_{3}m_{4}+18m_{4}^2-\frac{45}{7}m_{5}^2.
\end{array}$$

In {Step 6}, using the Matlab program in Appendix B,
we find the following SOS representation
\begin{equation}\label{eq-cc4}
\widehat{F}_{4,1}=\sum_{i=1}^{8}p_i\widehat{R}_i+ \sum_{i=1}^5 c_i (\sum_{j=i}^5 e_{ij} m_j)^2,
\end{equation}
where

$p_{1} =\frac{100}{779}, p_{2} =\frac{-131}{465}, p_{3} =\frac{643}{1030}, p_{4} =\frac{251}{243}, p_{5} =\frac{127}{264}, p_{6} =\frac{-108}{191}, p_{7} =\frac{-1037}{640}, p_{8} =\frac{-177}{496}$.

$c_1=\frac{1}{2},
 c_2=\frac{3035889}{{33948800}},
 c_3=\frac{187043606491}{32053482761280}$,
 $c_4=\frac{34954092290170394422572362033}{1201809187649038672153249873920}$,

$ c_5= \frac{13282901500034025857972998812856415717525}{23595304673035248961784746798960118147608956}$; %

{\small
$e_{1,1}=1, e_{1,2}=-\frac{643}{515},
e_{1,3}=-\frac{1417}{2060}, e_{1,4}=\frac{1037}{640},
e_{1,5}=-\frac{127}{264}$;

$e_{2,2}=1, e_{2,3}=-\frac{2397232}{31370853},
e_{2,4}=-\frac{353185661}{1475442054}, e_{2,5}=-\frac{3763912520}{19135208367}$;

$e_{3,3}= 1, e_{3,4}= -\frac{26403161591590331}{7999480962407088}, e_{3,5}= \frac{2349867732846895}{1178935851712773}$;

$e_{4,4}=1,
 e_{4,5}=-\frac{23761583089169700449786558901840}{27229237894042737255183870023707}$;

$e_{5,5}= 1$.}

In {Step} 7, from  above SOS representation \eqref{eq-cc4}, $D(4,1)$ is proved.
%
%
Also, equation (\ref{eq-cc4}) is different from that given in~\cite{Cheng2015}.

\section*{Appendix E. Constraints $C(3,n)$ in Section \ref{sec-3}}

\subsection*{Appendix $E_1$. Constraints in Lemma \ref{lm-3cons1}}

The Maple program to compute the 955 constraints in
Lemma \ref{lm-3cons1} can be found in \\
http://www.mmrc.iss.ac.cn/\~\,xgao/software/maple-epicons31.zip.

We  explain how the program works. The program is divided into three
steps.

Step 1. As in the proof of Lemma \ref{lm-3cons1}, we need to consider
all monomials in the variables in ${\mathcal V}_{a,b,c}$
with total order 6 and degree 6,
which can be written as
$M=\prod\limits_{i=1}^{6}\frac{\d ^{h_{i}} p_t}{\d ^{h_{i,1}} x_{a,t}\d ^{h_{i,2}} x_{b,t}\partial^{h_{i,3}} x_{c,t}}$.
Then, 10 cases for $H_6=(h_{1},h_{2},h_{3},h_{4},h_{5},h_{6})$ are considered:
\{$(5,1,0,0,0,0),(4,2,0,0,0,0),
(4,1,1,0,0,0),$ $ (3,2,1,0,0,0),
(3,1,1,1,$ $0,0),$ $ (2,1,1,1,1,0),$
$(3,3,0,0,0,0),(2,2,2,0,0,0),(2,2,1,1,0,0),
(1,1,1,1,$ $1,1)$.\}

Step 2. For each specific value of $H_6$, all possible $h_{i,j}$ are considered.
For instance, if $H_6=(5,1,0,0,0,0)$,
we consider all $h_{1,1}+h_{1,2}+h_{1,3}=5$, $h_{2,1}+h_{2,2}+h_{2,3}=1$,
and $h_{i,j}=0,i>2$.

Step 3. When the values of $h_{i,j}$ are fixed, $M$ is determined.
We find constraints from $M$ by considering all factor $v\in\P$ of $M$ with
$\ord(v)>0$ and write $M=M_1v$. Then, we can compute the constraint from $M_1v$ as in the proof of Lemma \ref{lm-cons1}.

\subsection*{Appendix $E_2$. Constraints in Lemma \ref{lm-3cons2}}
{\footnotesize\parskip=0pt\parindent=2pt
$\mathcal{R}^{(0)}_{1,a,b,c}=p_t^4\frac{\partial p_t}{\partial x_{a,t}}\frac{\partial^5p_t}{\partial x_{a,t}\partial^2x_{b,t}\partial^2x_{c,t}}
+p_t^4\frac{\partial^4p_t}{\partial^2 x_{a,t}\partial^2x_{b,t}}\frac{\partial^2p_t}{\partial^2x_{c,t}}
-p_t^3(\frac{\partial p_t}{\partial x_{c,t}})^2\frac{\partial^4p_t}{\partial^2 x_{a,t}\partial^2x_{b,t}}$

$\mathcal{R}^{(0)}_{2,a,b,c}=p_t^4\frac{\partial^2 p_t}{\partial^2x_{c,t}}\frac{\partial^4p_t}{\partial^2 x_{a,t}\partial^2x_{b,t}}
+p_t^4\frac{\partial^3p_t}{\partial x_{a,t}\partial^2x_{c,t}}\frac{\partial^3p_t}{\partial x_{a,t}\partial^2x_{b,t}}
-p_t^3\frac{\partial p_t}{\partial x_{a,t}}\frac{\partial^3p_t}{\partial x_{a,t}\partial^2x_{b,t}}\frac{\partial^2p_t}{\partial^2x_{c,t}}$

$\mathcal{R}^{(0)}_{3,a,b,c}=p_t^3(\frac{\partial p_t}{\partial x_{c,t}})^2\frac{\partial^4p_t}{\partial^2 x_{a,t}\partial^2x_{b,t}}
+2p_t^3\frac{\partial^3p_t}{\partial x_{a,t}\partial^2x_{b,t}}\frac{\partial p_t}{\partial x_{c,t}}\frac{\partial^2p_t}{\partial x_{a,t}\partial x_{c,t}}
-2p_t^2\frac{\partial^3 p_t}{\partial x_{a,t}\partial^2x_{b,t}}(\frac{\partial p_t}{\partial x_{c,t}})^2\frac{\partial p_t}{\partial x_{a,t}}$

$\mathcal{R}^{(0)}_{4,a,b,c}=p_t^3\frac{\partial p_t}{\partial x_{a,t}}\frac{\partial^3p_t}{\partial x_{a,t}\partial^2x_{b,t}}\frac{\partial^2p_t}{\partial^2x_{c,t}}
+p_t^3\frac{\partial p_t}{\partial x_{b,t}}\frac{\partial^2 p_t}{\partial^2 x_{a,t}}\frac{\partial^3p_t}{\partial x_{b,t}\partial^2x_{c,t}}
+p_t^3\frac{\partial^2 p_t}{\partial^2 x_{a,t}}\frac{\partial^2 p_t}{\partial^2x_{c,t}}\frac{\partial^2 p_t}{\partial^2x_{b,t}}
-2p_t^2(\frac{\partial p_t}{\partial x_{b,t}})^2\frac{\partial^2 p_t}{\partial^2 x_{a,t}}\frac{\partial^2 p_t}{\partial^2x_{c,t}}
$

$\mathcal{R}^{(0)}_{5,a,b,c}=p_t^2\frac{\partial^3 p_t}{\partial x_{a,t}\partial^2x_{b,t}}(\frac{\partial p_t}{\partial x_{c,t}})^2\frac{\partial p_t}{\partial x_{a,t}}
+p_t^2(\frac{\partial p_t}{\partial x_{c,t}})^2\frac{\partial^2 p_t}{\partial^2 x_{a,t}}\frac{\partial^2p_t}{\partial^2x_{b,t}}
+2p_t^2\frac{\partial^2 p_t}{\partial^2 x_{a,t}}\frac{\partial p_t}{\partial x_{b,t}}\frac{\partial p_t}{\partial x_{c,t}}\frac{\partial^2 p_t}{\partial x_{b,t}\partial x_{c,t}}
-3p_t\frac{\partial^2 p_t}{\partial^2 x_{a,t}}(\frac{\partial p_t}{\partial x_{b,t}})^2(\frac{\partial p_t}{\partial x_{c,t}})^2
$

$\mathcal{R}^{(0)}_{6,a,b,c}=p_t\frac{\partial^2 p_t}{\partial^2 x_{a,t}}(\frac{\partial p_t}{\partial x_{b,t}})^2(\frac{\partial p_t}{\partial x_{c,t}})^2
+2p_t\frac{\partial p_t}{\partial x_{a,t}}\frac{\partial p_t}{\partial x_{b,t}}(\frac{\partial p_t}{\partial x_{c,t}})^2\frac{\partial^2 p_t}{\partial x_{a,t}\partial x_{b,t}}
+2p_t\frac{\partial p_t}{\partial x_{a,t}}(\frac{\partial p_t}{\partial x_{b,t}})^2\frac{\partial p_t}{\partial x_{c,t}}\frac{\partial^2 p_t}{\partial x_{a,t}\partial x_{c,t}}$

\hskip20pt$-4(\frac{\partial p_t}{\partial x_{b,t}})^2(\frac{\partial p_t}{\partial x_{c,t}})^2(\frac{\partial p_t}{\partial x_{a,t}})^2
$

}

\subsection*{Appendix $E_3$. Constraints in Lemma \ref{lm-3cons3}}

The two constraints in Lemma \ref{lm-3cons3}:\\
{\footnotesize\parskip=0pt\parindent=2pt\parskip=2pt
$\mathcal{R}^{(0)}_{7,a,b,c}=\frac{p_t^4}{2}\frac{\partial^3 p_t}{\partial x_{a,t}\partial^2x_{c,t}}\frac{\partial^3 p_t}{\partial x_{a,t}\partial^2x_{b,t}}
+\frac{p_t^4}{2}\frac{\partial p_t}{\partial x_{a,t}}\frac{\partial^5 p_t}{\partial x_{a,t}\partial^2x_{b,t}\partial^2x_{c,t}}
-\frac{p_t^3}{2}\frac{\partial p_t}{\partial x_{a,t}}\frac{\partial^3 p_t}{\partial x_{a,t}\partial^2x_{b,t}}\frac{\partial^2 p_t}{\partial^2x_{c,t}}
+\frac{p_t^4}{2}\frac{\partial^4 p_t}{\partial^2 x_{a,t}\partial^2x_{c,t}}\frac{\partial^2 p_t}{\partial^2x_{b,t}}
+\frac{p_t^4}{2}\frac{\partial^2 p_t}{\partial^2 x_{a,t}}\frac{\partial^4 p_t}{\partial^2x_{b,t}\partial^2x_{c,t}}$
$-\frac{p_t^3}{2}\frac{\partial^2 p_t}{\partial^2 x_{a,t}}\frac{\partial^2 p_t}{\partial^2x_{b,t}}\frac{\partial^2 p_t}{\partial^2x_{c,t}}
-\frac{p_t^3}{2}\frac{\partial^4 p_t}{\partial^2 x_{a,t}\partial^2x_{c,t}}(\frac{\partial p_t}{\partial x_{b,t}})^2$
$-p_t^3\frac{\partial^2 p_t}{\partial^2 x_{a,t}}\frac{\partial p_t}{\partial x_{b,t}}\frac{\partial^3 p_t}{\partial x_{b,t}\partial^2x_{c,t}}
+p_t^2\frac{\partial^2 p_t}{\partial^2 x_{a,t}}(\frac{\partial p_t}{\partial x_{b,t}})^2\frac{\partial^2 p_t}{\partial^2x_{c,t}}
$

$\mathcal{R}^{(0)}_{8,a,b,c}=\frac{p_t^3}{2}\frac{\partial^4 p_t}{\partial^2 x_{a,t}\partial^2x_{c,t}}(\frac{\partial p_t}{\partial x_{b,t}})^2
+p_t^3\frac{\partial^2 p_t}{\partial^2 x_{a,t}}\frac{\partial p_t}{\partial x_{b,t}}\frac{\partial^3 p_t}{\partial x_{b,t}\partial^2x_{c,t}}
-p_t^2\frac{\partial^2 p_t}{\partial^2 x_{a,t}}(\frac{\partial p_t}{\partial x_{b,t}})^2\frac{\partial^2 p_t}{\partial^2x_{c,t}}
+p_t^3\frac{\partial^4 p_t}{\partial x_{a,t}\partial x_{b,t}\partial^2x_{c,t}}\frac{\partial p_t}{\partial x_{b,t}}\frac{\partial p_t}{\partial x_{a,t}}\\
+p_t^3\frac{\partial^2 p_t}{\partial x_{a,t}\partial x_{b,t}}\frac{\partial^3 p_t}{\partial x_{b,t}\partial^2x_{c,t}}\frac{\partial p_t}{\partial x_{a,t}}
+p_t^3\frac{\partial^2 p_t}{\partial x_{a,t}\partial x_{b,t}}\frac{\partial p_t}{\partial x_{b,t}}\frac{\partial^3 p_t}{\partial x_{a,t}\partial^2x_{c,t}}
-2p_t^2\frac{\partial^2 p_t}{\partial x_{a,t}\partial x_{b,t}}\frac{\partial p_t}{\partial x_{b,t}}\frac{\partial p_t}{\partial x_{a,t}}\frac{\partial^2 p_t}{\partial^2x_{c,t}}
-2p_t^2\frac{\partial p_t}{\partial x_{a,t}}(\frac{\partial p_t}{\partial x_{b,t}})^2\frac{\partial^3 p_t}{\partial x_{a,t}\partial^2x_{c,t}}$
$-2p_t^2(\frac{\partial p_t}{\partial x_{a,t}})^2\frac{\partial p_t}{\partial x_{b,t}}\frac{\partial^3 p_t}{\partial x_{b,t}\partial^2x_{c,t}}
+3p_t(\frac{\partial p_t}{\partial x_{a,t}})^2(\frac{\partial p_t}{\partial x_{b,t}})^2\frac{\partial^2 p_t}{\partial^2x_{c,t}}
$}
\\
\\
The 20  constraints in Lemma \ref{lm-3cons3}:\\
{\footnotesize\parskip=0pt\parindent=2pt\parskip=5pt
%

$\mathcal{R}^{(2)}_{1,a,b,c}=\frac{3}{2}\, p_t ^{3}
 ({\frac {\partial p_t}{\partial x_{{a}}}}) ^{2}
 {\frac {\partial ^{4}p_t}{\partial^{2} {x_{{c}}}\partial^{2} {x_{{a}}}}}
 +3\, p_t ^{3}
 {\frac {\partial p_t}{\partial x_{{a}}}}   {\frac {\partial ^{2}p_t }{\partial^{2} {x_{{a}}}}}  {\frac {\partial ^{3}p_t}{\partial^{2} {x_{{c}}}\partial x_{{a}}}} -3\, p_t ^{2}
 ({\frac {\partial p_t}{\partial x_{{a}}}}) ^{2} ( {\frac {\partial ^{2}p_t }{\partial^{2} {x_{{a}}}}} ) {\frac {\partial ^{2}p_t}{\partial^{2} {x_{{c}}}}}
 -4\, p_t ^{2}  {\frac {\partial ^{3}p_t}{\partial^{2} {x_{{c}}}\partial x_{{a}}}}
 ({\frac {\partial p_t}{\partial x_{{a}}}}) ^{3}+3\,p_t
 {\frac {\partial ^{2}p_t}{\partial^{2} {x_{{c}}}}}   ({\frac {\partial p_t}{\partial x_{{a}}}}) ^{4}
$

$\mathcal{R}^{(2)}_{2,a,b,c}=\frac{3}{2}\, p_t ^{3}
( {\frac {\partial p_t}{\partial x_{{b}}}}  )^{2}{\frac {\partial ^{4}p_t}{\partial^{2} {x_{{c}}}\partial^{2} {x_{{b}}}}} +3\, p_t ^{3}  {\frac {\partial p_t}{\partial x_{{b}}}}
 {\frac {\partial ^{2}p_t}{\partial^{2} {x_{{b}}}}}
{\frac {\partial ^{3}p_t}{\partial^{2} {x_{{c}}}\partial x_{{b}}}} -3\, p_t ^{2}
( {\frac {\partial p_t}{\partial x_{{b}}}}  ) ^{2}
 {\frac {\partial ^{2}p_t}{\partial^{2} {x_{{b}}}}}
 {\frac {\partial ^{2}p_t}{\partial^{2} {x_{{c}}}}}
 -4\, p_t ^{2}
  {\frac {\partial ^{3}p_t}{\partial^{2} {x_{{c}}}\partial x_{{b}}}}
  ( {\frac {\partial p_t}{\partial x_{{b}}}}  ) ^{3}+3\,p_t
   {\frac {\partial ^{2}p_t}{\partial^{2} {x_{{c}}}}}
   ( {\frac {\partial p_t}{\partial x_{{b}}}}  ) ^{4}
$

$\mathcal{R}^{(2)}_{3,a,b,c}=\frac{3}{2}\, p_t ^{3}
 {\frac {\partial ^{4}p_t}{\partial^{2} {x_{{c}}}\partial x_{{b}}\partial x_{{a}}}}
 ({\frac {\partial p_t}{\partial x_{{a}}}}) ^{2}+3\, p_t ^{3}
 ( {\frac {\partial ^{2}p_t}{\partial x_{{b}}\partial x_{{a}}}}  )
 ({\frac {\partial p_t}{\partial x_{{a}}}}) {\frac {\partial ^{3}p_t}{\partial^{2} {x_{{c}}}\partial x_{{a}}}} -3\, p_t ^{2} ( {\frac {\partial ^{2}p_t}{\partial x_{{b}}\partial x_{{a}}}}  )  ({\frac {\partial p_t}{\partial x_{{a}}}}) ^{2}{\frac {\partial ^{2}p_t}{\partial^{2} {x_{{c}}}}}
 -3\, p_t ^{2} ( {\frac {\partial p_t}{\partial x_{{b}}}}  )  ({\frac {\partial p_t}{\partial x_{{a}}}}) ^{2}{\frac {\partial ^{3}p_t}{\partial^{2} {x_{{c}}}\partial x_{{a}}}}$
 $+3\,p_t  ( {\frac {\partial p_t}{\partial x_{{b}}}}  )  ({\frac {\partial p_t}{\partial x_{{a}}}}) ^{3}{\frac {\partial ^{2}p_t}{\partial^{2} {x_{{c}}}}} - p_t ^{2} ( {\frac {\partial ^{3}p_t}{\partial^{2} {x_{{c}}}\partial x_{{b}}}}  )  ({\frac {\partial p_t}{\partial x_{{a}}}}) ^{3}
$

$\mathcal{R}^{(2)}_{4,a,b,c}=\frac{3}{2}\, p_t ^{3} {\frac {\partial ^{4}p_t}{\partial^{2} {x_{{c}}}\partial x_{{b}}\partial x_{{a}}}}   ( {\frac {\partial p_t}{\partial x_{{b}}}}  ) ^{2}+3\, p_t ^{3} ( {\frac {\partial ^{2}p_t}{\partial x_{{b}}\partial x_{{a}}}}  )  ( {\frac {\partial p_t}{\partial x_{{b}}}}  ) {\frac {\partial ^{3}p_t}{\partial^{2} {x_{{c}}}\partial x_{{b}}}} -3\, p_t ^{2} ( {\frac {\partial ^{2}p_t}{\partial x_{{b}}\partial x_{{a}}}}  )  ( {\frac {\partial p_t}{\partial x_{{b}}}}  ) ^{2}{\frac {\partial ^{2}p_t}{\partial^{2} {x_{{c}}}}}
 -3\, p_t ^{2} ( {\frac {\partial ^{3}p_t}{\partial^{2} {x_{{c}}}\partial x_{{b}}}}  )  ({\frac {\partial p_t}{\partial x_{{a}}}})  ( {\frac {\partial p_t}{\partial x_{{b}}}}  ) ^{2}$
 $- p_t ^{2} ( {\frac {\partial ^{3}p_t}{\partial^{2} {x_{{c}}}\partial x_{{a}}}}  )  ( {\frac {\partial p_t}{\partial x_{{b}}}}  ) ^{3}+3\,p_t  ({\frac {\partial p_t}{\partial x_{{a}}}})  ( {\frac {\partial p_t}{\partial x_{{b}}}}  ) ^{3}{\frac {\partial ^{2}p_t}{\partial^{2} {x_{{c}}}}}
$

$\mathcal{R}^{(2)}_{5,a,b,c}=\frac{1}{2}\, p_t ^{4} ( {\frac {\partial ^{3}p_t}{\partial^{2} {x_{{c}}}\partial x_{{a}}}}  )
{\frac {\partial ^{3}p_t}{\partial^{3} {x_{{a}}}}} +\frac{1}{2}\, p_t ^{4}
 ({\frac {\partial p_t}{\partial x_{{a}}}})
 {\frac {\partial ^{5}p_t}{\partial^{2} {x_{{c}}}\partial^{3} {x_{{a}}}}} -\frac{1}{2}\, p_t ^{3} ({\frac {\partial p_t}{\partial x_{{a}}}})
  ( {\frac {\partial ^{3}p_t}{\partial^{3} {x_{{a}}}}}  ) {\frac {\partial ^{2}p_t}{\partial^{2} {x_{{c}}}}} + p_t ^{4}  {\frac {\partial ^{4}p_t}{\partial^{2} {x_{{c}}}\partial^{2} {x_{{a}}}}}   {\frac {\partial ^{2}}{\partial^{2} {x_{{a}}}}}p_t$
  $-\frac{1}{2}\, p_t ^{3} ( {\frac {\partial ^{2}p_t }{\partial^{2} {x_{{a}}}}} ) ^{2}{\frac {\partial ^{2}p_t}{\partial^{2} {x_{{c}}}}}
 -\frac{1}{2}\, p_t ^{3} ({\frac {\partial p_t}{\partial x_{{a}}}}) ^{2}{\frac {\partial ^{4}p_t}{\partial^{2} {x_{{c}}}\partial^{2} {x_{{a}}}}}$
  $- p_t ^{3} ({\frac {\partial p_t}{\partial x_{{a}}}})  ( {\frac {\partial ^{2}p_t }{\partial^{2} {x_{{a}}}}} ) {\frac {\partial ^{3}p_t}{\partial^{2} {x_{{c}}}\partial x_{{a}}}} + p_t ^{2} ({\frac {\partial p_t}{\partial x_{{a}}}}) ^{2} ( {\frac {\partial ^{2}p_t }{\partial^{2} {x_{{a}}}}} ) {\frac {\partial ^{2}p_t}{\partial^{2} {x_{{c}}}}}
$

$\mathcal{R}^{(2)}_{6,a,b,c}=\frac{1}{2}\, p_t ^{4} ( {\frac {\partial ^{3}p_t}{\partial^{2} {x_{{c}}}\partial x_{{b}}}}  )
{\frac {\partial ^{3}p_t}{\partial^{3} {x_{{b}}}}} +\frac{1}{2}\, p_t ^{4} ( {\frac {\partial p_t}{\partial x_{{b}}}}  )
 {\frac {\partial ^{5}p_t}{\partial^{2} {x_{{c}}}\partial^{3} {x_{{b}}}}} -\frac{1}{2}\, p_t ^{3} ( {\frac {\partial p_t}{\partial x_{{b}}}}  )  ( {\frac {\partial ^{3}p_t}{\partial^{3} {x_{{b}}}}}  ) {\frac {\partial ^{2}p_t}{\partial^{2} {x_{{c}}}}} + p_t ^{4} ( {\frac {\partial ^{4}p_t}{\partial^{2} {x_{{c}}}\partial^{2} {x_{{b}}}}}  ) {\frac {\partial ^{2}p_t}{\partial^{2} {x_{{b}}}}} -\frac{1}{2}\, p_t ^{3} ( {\frac {\partial ^{2}p_t}{\partial^{2} {x_{{b}}}}}  ) ^{2}{\frac {\partial ^{2}p_t}{\partial^{2} {x_{{c}}}}} -\frac{1}{2}\, p_t ^{3} ( {\frac {\partial p_t}{\partial x_{{b}}}}  ) ^{2}{\frac {\partial ^{4}p_t}{\partial^{2} {x_{{c}}}\partial^{2} {x_{{b}}}}} - p_t ^{3} ( {\frac {\partial p_t}{\partial x_{{b}}}}  )  ( {\frac {\partial ^{2}p_t}{\partial^{2} {x_{{b}}}}}  ) {\frac {\partial ^{3}p_t}{\partial^{2} {x_{{c}}}\partial x_{{b}}}} + p_t ^{2} ( {\frac {\partial p_t}{\partial x_{{b}}}}  ) ^{2} ( {\frac {\partial ^{2}p_t}{\partial^{2} {x_{{b}}}}}  ) {\frac {\partial ^{2}p_t}{\partial^{2} {x_{{c}}}}} $

$\mathcal{R}^{(2)}_{7,a,b,c}=\frac{1}{2}\, p_t ^{4}
  {\frac {\partial ^{5}p_t}{\partial^{2} {x_{{c}}}\partial x_{{b}}\partial^{2} {x_{{a}}}}}  {\frac {\partial p_t}{\partial x_{{a}}}}
  +\frac{1}{2}\, p_t ^{4}
  {\frac {\partial ^{3}p_t }{\partial x_{{b}}\partial^{2} {x_{{a}}}}}
   {\frac {\partial ^{3}p_t}{\partial^{2} {x_{{c}}}\partial x_{{a}}}}
   -\frac{1}{2}\, p_t ^{3}  {\frac {\partial ^{3}p_t }{\partial x_{{b}}\partial^{2} {x_{{a}}}}}
   {\frac {\partial p_t}{\partial x_{{a}}}} {\frac {\partial ^{2}p_t}{\partial^{2} {x_{{c}}}}}   +\frac{1}{2}\, p_t ^{4} {\frac {\partial ^{4}p_t}{\partial^{2} {x_{{c}}}\partial x_{{b}}\partial x_{{a}}}}  {\frac {\partial ^{2}}{\partial^{2} {x_{{a}}}}}p_t
   -\frac{1}{2}\, p_t ^{3} {\frac {\partial ^{4}p_t}{\partial^{2} {x_{{c}}}\partial x_{{b}}\partial x_{{a}}}}   ({\frac {\partial p_t}{\partial x_{{a}}}}) ^{2}$ $+\frac{1}{2}\, p_t ^{4} ( {\frac {\partial ^{2}p_t}{\partial x_{{b}}\partial x_{{a}}}}  ) {\frac {\partial ^{4}p_t}{\partial^{2} {x_{{c}}}\partial^{2} {x_{{a}}}}} -\frac{1}{2}\, p_t ^{3} ( {\frac {\partial ^{2}p_t}{\partial x_{{b}}\partial x_{{a}}}}  )  ( {\frac {\partial ^{2}p_t }{\partial^{2} {x_{{a}}}}} ) {\frac {\partial ^{2}p_t}{\partial^{2} {x_{{c}}}}} - p_t ^{3} ( {\frac {\partial ^{2}p_t}{\partial x_{{b}}\partial x_{{a}}}}  )  ({\frac {\partial p_t}{\partial x_{{a}}}}) {\frac {\partial ^{3}p_t}{\partial^{2} {x_{{c}}}\partial x_{{a}}}} + p_t ^{2} ( {\frac {\partial ^{2}p_t}{\partial x_{{b}}\partial x_{{a}}}}  )  ({\frac {\partial p_t}{\partial x_{{a}}}}) ^{2}{\frac {\partial ^{2}p_t}{\partial^{2} {x_{{c}}}}}
$

$\mathcal{R}^{(2)}_{8,a,b,c}=\frac{1}{2}\, p_t ^{4}
 {\frac {\partial ^{5}p_t }{\partial^{2} {x_{{c}}}\partial^{2} {x_{{b}}}\partial x_{{a}}}} {\frac {\partial p_t}{\partial x_{{b}}}} +\frac{1}{2}\, p_t ^{4}
  {\frac {\partial ^{3}p_t}{\partial^{2} {x_{{b}}}\partial x_{{a}}}}   {\frac {\partial ^{3}p_t}{\partial^{2} {x_{{c}}}\partial x_{{b}}}} -\frac{1}{2}\, p_t ^{3}
   {\frac {\partial ^{3}p_t}{\partial^{2} {x_{{b}}}\partial x_{{a}}}}
      {\frac {\partial p_t}{\partial x_{{b}}}}   {\frac {\partial ^{2}p_t}{\partial^{2} {x_{{c}}}}}   +\frac{1}{2}\, p_t ^{4}
       {\frac {\partial ^{2}p_t}{\partial^{2} {x_{{b}}}}}
  {\frac {\partial ^{4}p_t}{\partial^{2} {x_{{c}}}\partial x_{{b}}\partial x_{{a}}}} -\frac{1}{2}\, p_t ^{3} {\frac {\partial ^{4}p_t}{\partial^{2} {x_{{c}}}\partial x_{{b}}\partial x_{{a}}}}
  ( {\frac {\partial p_t}{\partial x_{{b}}}}  ) ^{2}+\frac{1}{2}\, p_t ^{4}
   {\frac {\partial ^{4}p_t}{\partial^{2} {x_{{c}}}\partial^{2} {x_{{b}}}}}
   {\frac {\partial ^{2}p_t}{\partial x_{{b}}\partial x_{{a}}}} -\frac{1}{2}\, p_t ^{3}
    {\frac {\partial ^{2}p_t}{\partial^{2} {x_{{b}}}}}
     {\frac {\partial ^{2}p_t}{\partial x_{{b}}\partial x_{{a}}}}
      {\frac {\partial ^{2}p_t}{\partial^{2} {x_{{c}}}}} - p_t ^{3}
       {\frac {\partial ^{2}p_t}{\partial x_{{b}}\partial x_{{a}}}}
       {\frac {\partial p_t}{\partial x_{{b}}}}
      {\frac {\partial ^{3}p_t}{\partial^{2} {x_{{c}}}\partial x_{{b}}}} + p_t ^{2}
       {\frac {\partial ^{2}p_t}{\partial x_{{b}}\partial x_{{a}}}}    ( {\frac {\partial p_t}{\partial x_{{b}}}}  ) ^{2}{\frac {\partial ^{2}p_t}{\partial^{2} {x_{{c}}}}}
$

$\mathcal{R}^{(2)}_{9,a,b,c}=\frac{1}{2}\, p_t ^{4}
({\frac {\partial p_t}{\partial x_{{a}}}}) {\frac {\partial ^{5}p_t }{\partial^{2} {x_{{c}}}\partial^{2} {x_{{b}}}\partial x_{{a}}}}+\frac{1}{2}\, p_t ^{4}
  {\frac {\partial ^{3}p_t}{\partial^{2} {x_{{c}}}\partial x_{{a}}}}   {\frac {\partial ^{3}p_t}{\partial^{2} {x_{{b}}}\partial x_{{a}}}}
 -\frac{1}{2}\, p_t ^{3} {\frac {\partial p_t}{\partial x_{{a}}}}
  {\frac {\partial ^{3}p_t}{\partial^{2} {x_{{b}}}\partial x_{{a}}}}   {\frac {\partial ^{2}p_t}{\partial^{2} {x_{{c}}}}}
  + p_t ^{4} {\frac {\partial ^{4}p_t}{\partial^{2} {x_{{c}}}\partial x_{{b}}\partial x_{{a}}}}  {\frac {\partial ^{2}p_t}{\partial x_{{b}}\partial x_{{a}}}} -\frac{1}{2}\, p_t ^{3} {\frac {\partial ^{4}p_t}{\partial^{2} {x_{{c}}}\partial x_{{b}}\partial x_{{a}}}}
   {\frac {\partial p_t}{\partial x_{{b}}}}  {\frac {\partial p_t}{\partial x_{{a}}}}$
   $-\frac{1}{2}\, p_t ^{3} ( {\frac {\partial ^{2}p_t}{\partial x_{{b}}\partial x_{{a}}}}  ) ^{2}{\frac {\partial ^{2}p_t}{\partial^{2} {x_{{c}}}}}
   -\frac{1}{2}\, p_t ^{3}
    {\frac {\partial ^{2}p_t}{\partial x_{{b}}\partial x_{{a}}}}
     {\frac {\partial p_t}{\partial x_{{b}}}}   {\frac {\partial ^{3}p_t}{\partial^{2} {x_{{c}}}\partial x_{{a}}}} + p_t ^{2}
      {\frac {\partial ^{2}p_t}{\partial x_{{b}}\partial x_{{a}}}}
       {\frac {\partial p_t}{\partial x_{{b}}}}
       {\frac {\partial p_t}{\partial x_{{a}}}}
       {\frac {\partial ^{2}p_t}{\partial^{2} {x_{{c}}}}} -\frac{1}{2}\, p_t ^{3}
        {\frac {\partial ^{2}p_t}{\partial x_{{b}}\partial x_{{a}}}}
         {\frac {\partial ^{3}p_t}{\partial^{2} {x_{{c}}}\partial x_{{b}}}}   {\frac {\partial p_t}{\partial x_{{a}}}}
$

$\mathcal{R}^{(2)}_{10,a,b,c}=\frac{1}{2}\, p_t ^{4} ({\frac {\partial p_t}{\partial x_{{a}}}}) {\frac {\partial ^{5}p_t }{\partial^{2} {x_{{c}}}\partial^{2} {x_{{b}}}\partial x_{{a}}}}+\frac{1}{2}\, p_t ^{4} {\frac {\partial ^{3}p_t}{\partial^{2} {x_{{c}}}\partial x_{{a}}}}   {\frac {\partial ^{3}p_t}{\partial^{2} {x_{{b}}}\partial x_{{a}}}} -\frac{1}{2}\, p_t ^{3} ({\frac {\partial p_t}{\partial x_{{a}}}})   {\frac {\partial ^{3}p_t}{\partial^{2} {x_{{b}}}\partial x_{{a}}}}   {\frac {\partial ^{2}p_t}{\partial^{2} {x_{{c}}}}}   +\frac{1}{2}\, p_t ^{4} ( {\frac {\partial ^{2}p_t }{\partial^{2} {x_{{a}}}}} ) {\frac {\partial ^{4}p_t}{\partial^{2} {x_{{c}}}\partial^{2} {x_{{b}}}}} -\frac{1}{2}\, p_t ^{3} ({\frac {\partial p_t}{\partial x_{{a}}}}) ^{2}{\frac {\partial ^{4}p_t}{\partial^{2} {x_{{c}}}\partial^{2} {x_{{b}}}}} +\frac{1}{2}\, p_t ^{4}  {\frac {\partial ^{4}p_t}{\partial^{2} {x_{{c}}}\partial^{2} {x_{{a}}}}}   {\frac {\partial ^{2}p_t}{\partial^{2} {x_{{b}}}}}
 -\frac{1}{2}\, p_t ^{3} ( {\frac {\partial ^{2}p_t }{\partial^{2} {x_{{a}}}}} )  ( {\frac {\partial ^{2}p_t}{\partial^{2} {x_{{c}}}}}  ) {\frac {\partial ^{2}p_t}{\partial^{2} {x_{{b}}}}} - p_t ^{3} ({\frac {\partial p_t}{\partial x_{{a}}}})  ( {\frac {\partial ^{2}p_t}{\partial^{2} {x_{{b}}}}}  ) {\frac {\partial ^{3}p_t}{\partial^{2} {x_{{c}}}\partial x_{{a}}}}
+ p_t ^{2} ({\frac {\partial p_t}{\partial x_{{a}}}}) ^{2} ( {\frac {\partial ^{2}p_t}{\partial^{2} {x_{{b}}}}}  ) {\frac {\partial ^{2}p_t}{\partial^{2} {x_{{c}}}}}
$

$\mathcal{R}^{(2)}_{11,a,b,c}=\frac{1}{2}\, p_t ^{4} ( {\frac {\partial p_t}{\partial x_{{b}}}}  ) {\frac {\partial ^{5}p_t}{\partial^{2} {x_{{c}}}\partial x_{{b}}\partial^{2} {x_{{a}}}}} +\frac{1}{2}\, p_t ^{4} ( {\frac {\partial ^{3}p_t}{\partial^{2} {x_{{c}}}\partial x_{{b}}}}  ) {\frac {\partial ^{3}p_t }{\partial x_{{b}}\partial^{2} {x_{{a}}}}}-\frac{1}{2}\, p_t ^{3} ( {\frac {\partial p_t}{\partial x_{{b}}}}  )  ( {\frac {\partial ^{3}p_t }{\partial x_{{b}}\partial^{2} {x_{{a}}}}} ) {\frac {\partial ^{2}p_t}{\partial^{2} {x_{{c}}}}} +\frac{1}{2}\, p_t ^{4}  {\frac {\partial ^{4}p_t}{\partial^{2} {x_{{c}}}\partial^{2} {x_{{a}}}}}   {\frac {\partial ^{2}p_t}{\partial^{2} {x_{{b}}}}} -\frac{1}{2}\, p_t ^{3} ( {\frac {\partial p_t}{\partial x_{{b}}}}  ) ^{2}{\frac {\partial ^{4}p_t}{\partial^{2} {x_{{c}}}\partial^{2} {x_{{a}}}}} +\frac{1}{2}\, p_t ^{4} ( {\frac {\partial ^{2}p_t }{\partial^{2} {x_{{a}}}}} ) {\frac {\partial ^{4}p_t}{\partial^{2} {x_{{c}}}\partial^{2} {x_{{b}}}}} -\frac{1}{2}\, p_t ^{3} ( {\frac {\partial ^{2}p_t }{\partial^{2} {x_{{a}}}}} )  ( {\frac {\partial ^{2}p_t}{\partial^{2} {x_{{c}}}}}  ) {\frac {\partial ^{2}p_t}{\partial^{2} {x_{{b}}}}} - p_t ^{3} ( {\frac {\partial p_t}{\partial x_{{b}}}}  )  ( {\frac {\partial ^{2}p_t }{\partial^{2} {x_{{a}}}}} ) {\frac {\partial ^{3}p_t}{\partial^{2} {x_{{c}}}\partial x_{{b}}}}
 + p_t ^{2} ( {\frac {\partial p_t}{\partial x_{{b}}}}  ) ^{2}
 ( {\frac {\partial ^{2}p_t }{\partial^{2} {x_{{a}}}}} ) {\frac {\partial ^{2}p_t}{\partial^{2} {x_{{c}}}}}
$

$\mathcal{R}^{(2)}_{12,a,b,c}=\frac{1}{2}\, p_t ^{4}
 {\frac {\partial p_t}{\partial x_{{b}}}}   {\frac {\partial ^{5}p_t}{\partial^{2} {x_{{c}}}\partial x_{{b}}\partial^{2} {x_{{a}}}}} +\frac{1}{2}\, p_t ^{4}
  {\frac {\partial ^{3}p_t}{\partial^{2} {x_{{c}}}\partial x_{{b}}}}   {\frac {\partial ^{3}p_t }{\partial x_{{b}}\partial^{2} {x_{{a}}}}}-\frac{1}{2}\, p_t ^{3}
   {\frac {\partial p_t}{\partial x_{{b}}}}
    {\frac {\partial ^{3}p_t }{\partial x_{{b}}\partial^{2} {x_{{a}}}}}  {\frac {\partial ^{2}p_t}{\partial^{2} {x_{{c}}}}}
 + p_t ^{4} {\frac {\partial ^{4}p_t}{\partial^{2} {x_{{c}}}\partial x_{{b}}\partial x_{{a}}}}  {\frac {\partial ^{2}p_t}{\partial x_{{b}}\partial x_{{a}}}} -\frac{1}{2}\, p_t ^{3} {\frac {\partial ^{4}p_t}{\partial^{2} {x_{{c}}}\partial x_{{b}}\partial x_{{a}}}}
  {\frac {\partial p_t}{\partial x_{{b}}}}   {\frac {\partial p_t}{\partial x_{{a}}}} -\frac{1}{2}\, p_t ^{3} ( {\frac {\partial ^{2}p_t}{\partial x_{{b}}\partial x_{{a}}}}  ) ^{2}{\frac {\partial ^{2}p_t}{\partial^{2} {x_{{c}}}}}   -\frac{1}{2}\, p_t ^{3}
   {\frac {\partial ^{2}p_t}{\partial x_{{b}}\partial x_{{a}}}}
    {\frac {\partial p_t}{\partial x_{{b}}}}   {\frac {\partial ^{3}p_t}{\partial^{2} {x_{{c}}}\partial x_{{a}}}} + p_t ^{2}
     {\frac {\partial ^{2}p_t}{\partial x_{{b}}\partial x_{{a}}}}
       {\frac {\partial p_t}{\partial x_{{b}}}}
       {\frac {\partial p_t}{\partial x_{{a}}}} {\frac {\partial ^{2}p_t}{\partial^{2} {x_{{c}}}}} -\frac{1}{2}\, p_t ^{3}
        {\frac {\partial ^{2}p_t}{\partial x_{{b}}\partial x_{{a}}}}
         {\frac {\partial ^{3}p_t}{\partial^{2} {x_{{c}}}\partial x_{{b}}}}   {\frac {\partial p_t}{\partial x_{{a}}}}
$

$\mathcal{R}^{(2)}_{13,a,b,c}=\frac{1}{2}\, p_t ^{4}
  {\frac {\partial ^{5}p_t}{\partial^{2} {x_{{c}}}\partial^{3} {x_{{a}}}}}
  {\frac {\partial p_t}{\partial x_{{b}}}} +\frac{1}{2}\, p_t ^{4}
  {\frac {\partial ^{3}p_t}{\partial^{3} {x_{{a}}}}}
  {\frac {\partial ^{3}p_t}{\partial^{2} {x_{{c}}}\partial x_{{b}}}} -\frac{1}{2}\, p_t ^{3}
   {\frac {\partial ^{3}p_t}{\partial^{3} {x_{{a}}}}}
   {\frac {\partial p_t}{\partial x_{{b}}}}
   {\frac {\partial ^{2}p_t}{\partial^{2} {x_{{c}}}}} +\frac{1}{2}\, p_t ^{4}
   {\frac {\partial ^{2}p_t}{\partial x_{{b}}\partial x_{{a}}}}
    {\frac {\partial ^{4}p_t}{\partial^{2} {x_{{c}}}\partial^{2} {x_{{a}}}}} -\frac{1}{2}\, p_t ^{3}  {\frac {\partial ^{4}p_t}{\partial^{2} {x_{{c}}}\partial^{2} {x_{{a}}}}}
    {\frac {\partial p_t}{\partial x_{{a}}}} {\frac {\partial p_t}{\partial x_{{b}}}} +\frac{1}{2}\, p_t ^{4} {\frac {\partial ^{4}p_t}{\partial^{2} {x_{{c}}}\partial x_{{b}}\partial x_{{a}}}}  {\frac {\partial ^{2}}{\partial^{2} {x_{{a}}}}}p_t -\frac{1}{2}\, p_t ^{3}
     {\frac {\partial ^{2}p_t}{\partial x_{{b}}\partial x_{{a}}}}
     {\frac {\partial ^{2}p_t }{\partial^{2} {x_{{a}}}}}  {\frac {\partial ^{2}p_t}{\partial^{2} {x_{{c}}}}}
 -\frac{1}{2}\, p_t ^{3}
  {\frac {\partial ^{2}p_t }{\partial^{2} {x_{{a}}}}}    {\frac {\partial ^{3}p_t}{\partial^{2} {x_{{c}}}\partial x_{{a}}}}   {\frac {\partial p_t}{\partial x_{{b}}}} + p_t ^{2}
   {\frac {\partial ^{2}p_t }{\partial^{2} {x_{{a}}}}}
   {\frac {\partial p_t}{\partial x_{{a}}}}
    {\frac {\partial p_t}{\partial x_{{b}}}}
     {\frac {\partial ^{2}p_t}{\partial^{2} {x_{{c}}}}} -\frac{1}{2}\, p_t ^{3}
     {\frac {\partial ^{2}p_t }{\partial^{2} {x_{{a}}}}}
     {\frac {\partial p_t}{\partial x_{{a}}}} {\frac {\partial ^{3}p_t}{\partial^{2} {x_{{c}}}\partial x_{{b}}}}
$

$\mathcal{R}^{(2)}_{14,a,b,c}=\frac{1}{2}\, p_t ^{4}  {\frac {\partial ^{5}p_t}{\partial^{2} {x_{{c}}}\partial x_{{b}}\partial^{2} {x_{{a}}}}}  {\frac {\partial p_t}{\partial x_{{a}}}} +\frac{1}{2}\, p_t ^{4}
 {\frac {\partial ^{3}p_t }{\partial x_{{b}}\partial^{2} {x_{{a}}}}}  {\frac {\partial ^{3}p_t}{\partial^{2} {x_{{c}}}\partial x_{{a}}}} -\frac{1}{2}\, p_t ^{3}
  {\frac {\partial ^{3}p_t }{\partial x_{{b}}\partial^{2} {x_{{a}}}}}
  {\frac {\partial p_t}{\partial x_{{a}}}} {\frac {\partial ^{2}p_t}{\partial^{2} {x_{{c}}}}}
 +\frac{1}{2}\, p_t ^{4}
  {\frac {\partial ^{2}p_t}{\partial x_{{b}}\partial x_{{a}}}}   {\frac {\partial ^{4}p_t}{\partial^{2} {x_{{c}}}\partial^{2} {x_{{a}}}}} -\frac{1}{2}\, p_t ^{3}  {\frac {\partial ^{4}p_t}{\partial^{2} {x_{{c}}}\partial^{2} {x_{{a}}}}}
  {\frac {\partial p_t}{\partial x_{{a}}}} {\frac {\partial p_t}{\partial x_{{b}}}} +\frac{1}{2}\, p_t ^{4} {\frac {\partial ^{4}p_t}{\partial^{2} {x_{{c}}}\partial x_{{b}}\partial x_{{a}}}}  {\frac {\partial ^{2}}{\partial^{2} {x_{{a}}}}}p_t -\frac{1}{2}\, p_t ^{3}
   {\frac {\partial ^{2}p_t}{\partial x_{{b}}\partial x_{{a}}}}
    {\frac {\partial ^{2}p_t }{\partial^{2} {x_{{a}}}}}  {\frac {\partial ^{2}p_t}{\partial^{2} {x_{{c}}}}} -\frac{1}{2}\, p_t ^{3}
     {\frac {\partial ^{2}p_t }{\partial^{2} {x_{{a}}}}}    {\frac {\partial ^{3}p_t}{\partial^{2} {x_{{c}}}\partial x_{{a}}}}   {\frac {\partial p_t}{\partial x_{{b}}}} + p_t ^{2}
      {\frac {\partial ^{2}p_t }{\partial^{2} {x_{{a}}}}}
      {\frac {\partial p_t}{\partial x_{{a}}}}
       {\frac {\partial p_t}{\partial x_{{b}}}}   {\frac {\partial ^{2}p_t}{\partial^{2} {x_{{c}}}}} -\frac{1}{2}\, p_t ^{3}
        {\frac {\partial ^{2}p_t }{\partial^{2} {x_{{a}}}}}
        {\frac {\partial p_t}{\partial x_{{a}}}} {\frac {\partial ^{3}p_t}{\partial^{2} {x_{{c}}}\partial x_{{b}}}}
$

$\mathcal{R}^{(2)}_{15,a,b,c}=\frac{1}{2}\, p_t ^{4} {\frac {\partial ^{5}p_t }{\partial^{2} {x_{{c}}}\partial^{2} {x_{{b}}}\partial x_{{a}}}} {\frac {\partial p_t}{\partial x_{{b}}}} +\frac{1}{2}\, p_t ^{4}  {\frac {\partial ^{3}p_t}{\partial^{2} {x_{{b}}}\partial x_{{a}}}}   {\frac {\partial ^{3}p_t}{\partial^{2} {x_{{c}}}\partial x_{{b}}}} -\frac{1}{2}\, p_t ^{3}  {\frac {\partial ^{3}p_t}{\partial^{2} {x_{{b}}}\partial x_{{a}}}}
 {\frac {\partial p_t}{\partial x_{{b}}}}   {\frac {\partial ^{2}p_t}{\partial^{2} {x_{{c}}}}}
 +\frac{1}{2}\, p_t ^{4}  {\frac {\partial ^{4}p_t}{\partial^{2} {x_{{c}}}\partial^{2} {x_{{b}}}}}   {\frac {\partial ^{2}p_t}{\partial x_{{b}}\partial x_{{a}}}} -\frac{1}{2}\, p_t ^{3}  {\frac {\partial ^{4}p_t}{\partial^{2} {x_{{c}}}\partial^{2} {x_{{b}}}}}
 {\frac {\partial p_t}{\partial x_{{a}}}} {\frac {\partial p_t}{\partial x_{{b}}}} +\frac{1}{2}\, p_t ^{4}
  {\frac {\partial ^{2}p_t}{\partial^{2} {x_{{b}}}}}  {\frac {\partial ^{4}p_t}{\partial^{2} {x_{{c}}}\partial x_{{b}}\partial x_{{a}}}} -\frac{1}{2}\, p_t ^{3}
   {\frac {\partial ^{2}p_t}{\partial^{2} {x_{{b}}}}}
    {\frac {\partial ^{2}p_t}{\partial x_{{b}}\partial x_{{a}}}}   {\frac {\partial ^{2}p_t}{\partial^{2} {x_{{c}}}}} -\frac{1}{2}\, p_t ^{3}
  {\frac {\partial ^{2}p_t}{\partial^{2} {x_{{b}}}}}   {\frac {\partial ^{3}p_t}{\partial^{2} {x_{{c}}}\partial x_{{a}}}}   {\frac {\partial p_t}{\partial x_{{b}}}} + p_t ^{2}
   {\frac {\partial ^{2}p_t}{\partial^{2} {x_{{b}}}}}
   {\frac {\partial p_t}{\partial x_{{a}}}}
   {\frac {\partial p_t}{\partial x_{{b}}}}   {\frac {\partial ^{2}p_t}{\partial^{2} {x_{{c}}}}} -\frac{1}{2}\, p_t ^{3}
    {\frac {\partial ^{2}p_t}{\partial^{2} {x_{{b}}}}}
    {\frac {\partial p_t}{\partial x_{{a}}}} {\frac {\partial ^{3}p_t}{\partial^{2} {x_{{c}}}\partial x_{{b}}}}
$

$\mathcal{R}^{(2)}_{16,a,b,c}=\frac{1}{2}\, p_t ^{4}
 {\frac {\partial ^{5}p_t}{\partial^{2} {x_{{c}}}\partial^{3} {x_{{b}}}}}  {\frac {\partial p_t}{\partial x_{{a}}}} +\frac{1}{2}\, p_t ^{4}
  {\frac {\partial ^{3}p_t}{\partial^{3} {x_{{b}}}}}  {\frac {\partial ^{3}p_t}{\partial^{2} {x_{{c}}}\partial x_{{a}}}} -\frac{1}{2}\, p_t ^{3}
   {\frac {\partial ^{3}p_t}{\partial^{3} {x_{{b}}}}}
   {\frac {\partial p_t}{\partial x_{{a}}}} {\frac {\partial ^{2}p_t}{\partial^{2} {x_{{c}}}}} +\frac{1}{2}\, p_t ^{4}  {\frac {\partial ^{4}p_t}{\partial^{2} {x_{{c}}}\partial^{2} {x_{{b}}}}}   {\frac {\partial ^{2}p_t}{\partial x_{{b}}\partial x_{{a}}}} -\frac{1}{2}\, p_t ^{3}  {\frac {\partial ^{4}p_t}{\partial^{2} {x_{{c}}}\partial^{2} {x_{{b}}}}}
   {\frac {\partial p_t}{\partial x_{{a}}}} {\frac {\partial p_t}{\partial x_{{b}}}} +\frac{1}{2}\, p_t ^{4}
    {\frac {\partial ^{2}p_t}{\partial^{2} {x_{{b}}}}} {\frac {\partial ^{4}p_t}{\partial^{2} {x_{{c}}}\partial x_{{b}}\partial x_{{a}}}} -\frac{1}{2}\, p_t ^{3}
    {\frac {\partial ^{2}p_t}{\partial^{2} {x_{{b}}}}}
     {\frac {\partial ^{2}p_t}{\partial x_{{b}}\partial x_{{a}}}}  {\frac {\partial ^{2}p_t}{\partial^{2} {x_{{c}}}}} -\frac{1}{2}\, p_t ^{3}
     {\frac {\partial ^{2}p_t}{\partial^{2} {x_{{b}}}}}   {\frac {\partial ^{3}p_t}{\partial^{2} {x_{{c}}}\partial x_{{a}}}}   {\frac {\partial p_t}{\partial x_{{b}}}} + p_t ^{2}
      {\frac {\partial ^{2}p_t}{\partial^{2} {x_{{b}}}}}
     {\frac {\partial p_t}{\partial x_{{a}}}}
     {\frac {\partial p_t}{\partial x_{{b}}}} {\frac {\partial ^{2}p_t}{\partial^{2} {x_{{c}}}}} -\frac{1}{2}\, p_t ^{3}
      {\frac {\partial ^{2}p_t}{\partial^{2} {x_{{b}}}}}
      {\frac {\partial p_t}{\partial x_{{a}}}} {\frac {\partial ^{3}p_t}{\partial^{2} {x_{{c}}}\partial x_{{b}}}}
$

$\mathcal{R}^{(2)}_{17,a,b,c}=\frac{1}{2}\, p_t ^{3} {\frac {\partial ^{4}p_t}{\partial^{2} {x_{{c}}}\partial x_{{b}}\partial x_{{a}}}}
({\frac {\partial p_t}{\partial x_{{a}}}}) ^{2}+ p_t ^{3}
 {\frac {\partial ^{2}p_t}{\partial x_{{b}}\partial x_{{a}}}}
 {\frac {\partial p_t}{\partial x_{{a}}}} {\frac {\partial ^{3}p_t}{\partial^{2} {x_{{c}}}\partial x_{{a}}}} - p_t ^{2}
  {\frac {\partial ^{2}p_t}{\partial x_{{b}}\partial x_{{a}}}}
  ({\frac {\partial p_t}{\partial x_{{a}}}}) ^{2}{\frac {\partial ^{2}p_t}{\partial^{2} {x_{{c}}}}} + p_t ^{3}  {\frac {\partial ^{2}p_t }{\partial^{2} {x_{{a}}}}}
  {\frac {\partial p_t}{\partial x_{{a}}}} {\frac {\partial ^{3}p_t}{\partial^{2} {x_{{c}}}\partial x_{{b}}}}
 - p_t ^{2}  {\frac {\partial ^{3}p_t}{\partial^{2} {x_{{c}}}\partial x_{{b}}}}
 ({\frac {\partial p_t}{\partial x_{{a}}}}) ^{3}+ p_t ^{3}
  {\frac {\partial ^{4}p_t}{\partial^{2} {x_{{c}}}\partial^{2} {x_{{a}}}}}
  {\frac {\partial p_t}{\partial x_{{a}}}} {\frac {\partial p_t}{\partial x_{{b}}}} -2\, p_t ^{2}
   {\frac {\partial ^{2}p_t }{\partial^{2} {x_{{a}}}}}
   {\frac {\partial p_t}{\partial x_{{a}}}}
    {\frac {\partial p_t}{\partial x_{{b}}}}  {\frac {\partial ^{2}p_t}{\partial^{2} {x_{{c}}}}} + p_t ^{3}  {\frac {\partial ^{2}p_t }{\partial^{2} {x_{{a}}}}}    {\frac {\partial ^{3}p_t}{\partial^{2} {x_{{c}}}\partial x_{{a}}}}   {\frac {\partial p_t}{\partial x_{{b}}}} -3\, p_t ^{2}
     {\frac {\partial p_t}{\partial x_{{b}}}}
     ({\frac {\partial p_t}{\partial x_{{a}}}}) ^{2}{\frac {\partial ^{3}p_t}{\partial^{2} {x_{{c}}}\partial x_{{a}}}} +3\,p_t   {\frac {\partial p_t}{\partial x_{{b}}}}
     ({\frac {\partial p_t}{\partial x_{{a}}}}) ^{3}{\frac {\partial ^{2}p_t}{\partial^{2} {x_{{c}}}}}
$

$\mathcal{R}^{(2)}_{18,a,b,c}=\frac{1}{2}\, p_t ^{3}
 ({\frac {\partial p_t}{\partial x_{{a}}}}) ^{2}{\frac {\partial ^{4}p_t}{\partial^{2} {x_{{c}}}\partial^{2} {x_{{b}}}}} + p_t ^{3}
 {\frac {\partial p_t}{\partial x_{{a}}}}
  {\frac {\partial ^{2}p_t}{\partial^{2} {x_{{b}}}}}  {\frac {\partial ^{3}p_t}{\partial^{2} {x_{{c}}}\partial x_{{a}}}} - p_t ^{2}
  ({\frac {\partial p_t}{\partial x_{{a}}}}) ^{2}
   {\frac {\partial ^{2}p_t}{\partial^{2} {x_{{b}}}}}  {\frac {\partial ^{2}p_t}{\partial^{2} {x_{{c}}}}} + p_t ^{3}  {\frac {\partial ^{2}p_t}{\partial x_{{b}}\partial x_{{a}}}}
    {\frac {\partial ^{3}p_t}{\partial^{2} {x_{{c}}}\partial x_{{b}}}}  {\frac {\partial p_t}{\partial x_{{a}}}} -2\, p_t ^{2}
    ({\frac {\partial p_t}{\partial x_{{a}}}}) ^{2}
     {\frac {\partial p_t}{\partial x_{{b}}}}  {\frac {\partial ^{3}p_t}{\partial^{2} {x_{{c}}}\partial x_{{b}}}} + p_t ^{3} {\frac {\partial ^{4}p_t}{\partial^{2} {x_{{c}}}\partial x_{{b}}\partial x_{{a}}}}
      {\frac {\partial p_t}{\partial x_{{b}}}}  {\frac {\partial p_t}{\partial x_{{a}}}} -2\, p_t ^{2}
       {\frac {\partial ^{2}p_t}{\partial x_{{b}}\partial x_{{a}}}}
       {\frac {\partial p_t}{\partial x_{{b}}}}
       {\frac {\partial p_t}{\partial x_{{a}}}} {\frac {\partial ^{2}p_t}{\partial^{2} {x_{{c}}}}} + p_t ^{3}
        {\frac {\partial ^{2}p_t}{\partial x_{{b}}\partial x_{{a}}}}
        {\frac {\partial p_t}{\partial x_{{b}}}}  {\frac {\partial ^{3}p_t}{\partial^{2} {x_{{c}}}\partial x_{{a}}}}
 -2\, p_t ^{2}
 {\frac {\partial p_t}{\partial x_{{a}}}}
 ( {\frac {\partial p_t}{\partial x_{{b}}}}  ) ^{2}{\frac {\partial ^{3}p_t}{\partial^{2} {x_{{c}}}\partial x_{{a}}}} +3\,p_t  ({\frac {\partial p_t}{\partial x_{{a}}}}) ^{2}
 ( {\frac {\partial p_t}{\partial x_{{b}}}}  ) ^{2}{\frac {\partial ^{2}p_t}{\partial^{2} {x_{{c}}}}}
$

$\mathcal{R}^{(2)}_{19,a,b,c}=\frac{1}{2}\, p_t ^{3}
( {\frac {\partial p_t}{\partial x_{{b}}}}  ) ^{2}{\frac {\partial ^{4}p_t}{\partial^{2} {x_{{c}}}\partial^{2} {x_{{a}}}}} + p_t ^{3}
 {\frac {\partial p_t}{\partial x_{{b}}}}     {\frac {\partial ^{2}p_t }{\partial^{2} {x_{{a}}}}} {\frac {\partial ^{3}p_t}{\partial^{2} {x_{{c}}}\partial x_{{b}}}} - p_t ^{2}
 ( {\frac {\partial p_t}{\partial x_{{b}}}}  ) ^{2}
  {\frac {\partial ^{2}p_t }{\partial^{2} {x_{{a}}}}} {\frac {\partial ^{2}p_t}{\partial^{2} {x_{{c}}}}} + p_t ^{3}
 {\frac {\partial ^{4}p_t}{\partial^{2} {x_{{c}}}\partial x_{{b}}\partial x_{{a}}}}
  {\frac {\partial p_t}{\partial x_{{b}}}}   {\frac {\partial p_t}{\partial x_{{a}}}} + p_t ^{3}
  {\frac {\partial ^{2}p_t}{\partial x_{{b}}\partial x_{{a}}}}
   {\frac {\partial ^{3}p_t}{\partial^{2} {x_{{c}}}\partial x_{{b}}}}  {\frac {\partial p_t}{\partial x_{{a}}}} + p_t ^{3}
    {\frac {\partial ^{2}p_t}{\partial x_{{b}}\partial x_{{a}}}}
    {\frac {\partial p_t}{\partial x_{{b}}}}   {\frac {\partial ^{3}p_t}{\partial^{2} {x_{{c}}}\partial x_{{a}}}} -2\, p_t ^{2}
     {\frac {\partial ^{2}p_t}{\partial x_{{b}}\partial x_{{a}}}}
       {\frac {\partial p_t}{\partial x_{{b}}}}
       {\frac {\partial p_t}{\partial x_{{a}}}} {\frac {\partial ^{2}p_t}{\partial^{2} {x_{{c}}}}} -2\, p_t ^{2}
       {\frac {\partial p_t}{\partial x_{{a}}}}
       ( {\frac {\partial p_t}{\partial x_{{b}}}}  ) ^{2}{\frac {\partial ^{3}p_t}{\partial^{2} {x_{{c}}}\partial x_{{a}}}}
 -2\, p_t ^{2} ({\frac {\partial p_t}{\partial x_{{a}}}}) ^{2}
  {\frac {\partial p_t}{\partial x_{{b}}}}   {\frac {\partial ^{3}p_t}{\partial^{2} {x_{{c}}}\partial x_{{b}}}} +3\,p_t  ({\frac {\partial p_t}{\partial x_{{a}}}}) ^{2} ( {\frac {\partial p_t}{\partial x_{{b}}}}  ) ^{2}{\frac {\partial ^{2}p_t}{\partial^{2} {x_{{c}}}}}
$

$\mathcal{R}^{(2)}_{20,a,b,c}= p_t ^{3}
 {\frac {\partial ^{4}p_t}{\partial^{2} {x_{{c}}}\partial^{2} {x_{{b}}}}}
 {\frac {\partial p_t}{\partial x_{{a}}}} {\frac {\partial p_t}{\partial x_{{b}}}} + p_t ^{3}
  {\frac {\partial ^{2}p_t}{\partial^{2} {x_{{b}}}}}
  {\frac {\partial p_t}{\partial x_{{a}}}} {\frac {\partial ^{3}p_t}{\partial^{2} {x_{{c}}}\partial x_{{b}}}} + p_t ^{3}
   {\frac {\partial ^{2}p_t}{\partial^{2} {x_{{b}}}}}    {\frac {\partial ^{3}p_t}{\partial^{2} {x_{{c}}}\partial x_{{a}}}}   {\frac {\partial p_t}{\partial x_{{b}}}}
-2\, p_t ^{2}
  {\frac {\partial ^{2}p_t}{\partial^{2} {x_{{b}}}}}
  {\frac {\partial p_t}{\partial x_{{a}}}}
   {\frac {\partial p_t}{\partial x_{{b}}}}  {\frac {\partial ^{2}p_t}{\partial^{2} {x_{{c}}}}} + p_t ^{3}  {\frac {\partial ^{2}p_t}{\partial x_{{b}}\partial x_{{a}}}}
   {\frac {\partial p_t}{\partial x_{{b}}}}  {\frac {\partial ^{3}p_t}{\partial^{2} {x_{{c}}}\partial x_{{b}}}} -3\, p_t ^{2}
    {\frac {\partial ^{3}p_t}{\partial^{2} {x_{{c}}}\partial x_{{b}}}}
    {\frac {\partial p_t}{\partial x_{{a}}}}
    ( {\frac {\partial p_t}{\partial x_{{b}}}}  ) ^{2}
+\frac{1}{2}\, p_t ^{3} {\frac {\partial ^{4}p_t}{\partial^{2} {x_{{c}}}\partial x_{{b}}\partial x_{{a}}}}   ( {\frac {\partial p_t}{\partial x_{{b}}}}  ) ^{2}
- p_t ^{2}
 {\frac {\partial ^{2}p_t}{\partial x_{{b}}\partial x_{{a}}}}
 ( {\frac {\partial p_t}{\partial x_{{b}}}}  ) ^{2}{\frac {\partial ^{2}p_t}{\partial^{2} {x_{{c}}}}}
 - p_t ^{2}  {\frac {\partial ^{3}p_t}{\partial^{2} {x_{{c}}}\partial x_{{a}}}}
 ( {\frac {\partial p_t}{\partial x_{{b}}}}  ) ^{3}+3\,p_t
 {\frac {\partial p_t}{\partial x_{{a}}}}  ( {\frac {\partial p_t}{\partial x_{{b}}}}  ) ^{3}{\frac {\partial ^{2}p_t}{\partial^{2} {x_{{c}}}}}
$
}

%
%
%
%
%
%
%

\end{document}